%% file: rigid_structures.tex
\documentclass[10pt]{amsart}

\usepackage[pdfstartpage={1},pdfstartview={FitH},pdftitle={Rigid structures in the universal enveloping traffic space},pdfauthor={Benson Au and Camille Male},pdfsubject={2010 Mathematics Subject Classification: 15B52; 46L53; 46L54; 60B20},pdfcreator={},pdfproducer={},pdfkeywords={Cactus graph; free cumulant; free probability; non-crossing partition; operad; traffic probability},pagebackref=false,colorlinks=true,linkcolor={red!87.5!black},citecolor={blue!50!black},linktocpage=true]{hyperref}

\usepackage{accents}
\usepackage{afterpage}
\usepackage{amssymb}
\usepackage{appendix}
\usepackage{array}
\usepackage{bm}
\usepackage{bbm}
\usepackage{braket}
\usepackage{caption}
\usepackage{centernot}
\usepackage[noadjust]{cite}
\usepackage{enumitem}
\usepackage{float}
\restylefloat{figure}
\usepackage{graphicx}
\usepackage{indentfirst}
\usepackage{mathdots}
\usepackage{mathrsfs}
\usepackage{mathtools}
\usepackage{mleftright}
\usepackage{pgf}
\usepackage{scalerel}
\usepackage{stackengine}
\usepackage{tikz}
\usetikzlibrary{arrows}
\usepackage{tikz-cd}
\allowdisplaybreaks[4]

\MHInternalSyntaxOn
\renewcommand{\dcases}
 {
  \MT_start_cases:nnnn
    {\quad}
    {$\m@th\displaystyle##$\hfil}
    {$\m@th\displaystyle##$\hfil}
    {\lbrace}
 }
\MHInternalSyntaxOff

\newtheorem{thm}{Theorem}[section]

\newtheorem{cor}[thm]{Corollary}

\newtheorem{lemma}[thm]{Lemma}
\newtheorem{prop}[thm]{Proposition}

\theoremstyle{definition}
\newtheorem{defn}[thm]{Definition}

\newtheorem{eg}[thm]{Example}

\theoremstyle{remark}
\newtheorem{rem}[thm]{Remark}

\def\Circlearrowleft{\ensuremath{%
  \rotatebox[origin=c]{165}{$\circlearrowleft$}}}

\def\updownarrows{\ensuremath{%
  \rotatebox[origin=c]{90}{$\rightleftarrows$}}}

\def\sidecircle{\ensuremath{%
    \rotatebox[origin=c]{90}{$\dcirclearrowleft$}}}

\def\rsidecircle{\ensuremath{%
  \rotatebox[origin=c]{270}{$\dcirclearrowleft$}}}

\def\CCirclearrowleft{\ensuremath{%
  \rotatebox[origin=c]{-15}{$\circlearrowleft$}}}

\def\Circlearrowright{\ensuremath{%
  \rotatebox[origin=c]{195}{$\circlearrowright$}}}

\def\CCirclearrowright{\ensuremath{%
  \rotatebox[origin=c]{15}{$\circlearrowright$}}}

\newcommand\srightleftarrows{\mathrel{%
  \ensurestackMath{\stackengine{-1pt}{}{\rightleftarrows}{O}{c}{F}{T}{L}}%
}}

\newcommand\sleftrightarrows{\mathrel{%
  \ensurestackMath{\stackengine{-.875pt}{}{\leftrightarrows}{O}{c}{F}{T}{L}}%
}}

\newcommand\ssidecircle{\mathrel{%
  \ensurestackMath{\stackengine{-3pt}{}{\sidecircle}{O}{c}{F}{T}{L}}%
}}

\newcommand\srsidecircle{\mathrel{%
  \ensurestackMath{\stackengine{-3pt}{}{\rsidecircle}{O}{c}{F}{T}{L}}%
}}

\newcommand\dcirclearrowleft{\mathrel{%
  \ensurestackMath{\stackengine{3.25pt}{\cdot}{\Circlearrowleft}{O}{c}{F}{T}{L}}%
}}

\newcommand\dcirclearrowright{\mathrel{%
  \ensurestackMath{\stackengine{3.25pt}{\cdot}{\Circlearrowright}{O}{c}{F}{T}{L}}%
}}

\newcommand\vdownarrowv{\mathrel{%
  \ensurestackMath{\stackengine{20pt}{\vdownarrow}{\cdot}{O}{c}{F}{T}{L}}%
}}

\newcommand\dlcirclearrowleft{\mathrel{%
  \ensurestackMath{\stackengine{-3.25pt}{\vdownarrowv}{\CCirclearrowleft}{O}{c}{F}{T}{L}}%
}}

\newcommand\ffgeight{\mathrel{%
  \ensurestackMath{\stackengine{-3.25pt}{\dcirclearrowright}{\CCirclearrowright}{O}{c}{F}{T}{L}}%
}}

\newcommand\fffgeight{\mathrel{%
  \ensurestackMath{\stackengine{-3.25pt}{\dcirclearrowleft}{\CCirclearrowright}{O}{c}{F}{T}{L}}%
}}

\newcommand\vuparrow{\mathrel{%
  \ensurestackMath{\stackengine{10pt}{\cdot}{\uparrow}{O}{c}{F}{T}{L}}%
}}

\newcommand\vvuparrow{\mathrel{%
  \ensurestackMath{\stackengine{20pt}{\vuparrow}{\cdot}{O}{c}{F}{T}{L}}%
}}

\newcommand\vdownarrow{\mathrel{%
  \ensurestackMath{\stackengine{10pt}{\cdot}{\downarrow}{O}{c}{F}{T}{L}}%
}}

\newcommand\vvdownarrow{\mathrel{%
  \ensurestackMath{\stackengine{20pt}{\vdownarrow}{\cdot}{O}{c}{F}{T}{L}}%
}}

\newcommand\ccdot{\mathrel{%
  \ensurestackMath{\stackengine{2.5pt}{\cdot}{\cdot}{O}{c}{F}{T}{L}}%
}}

\newcommand\vupdown{\mathrel{%
  \ensurestackMath{\stackengine{10pt}{\cdot\,}{\updownarrows}{O}{c}{F}{T}{L}}%
}}

\newcommand\loopdown{\mathrel{%
  \ensurestackMath{\stackengine{-19pt}{\dcirclearrowleft}{\,\vupdown}{O}{c}{F}{T}{L}}%
}}

\newcommand\hack{\mathrel{%
  \ensurestackMath{\stackengine{-9pt}{\scalebox{.675}{\txout}}{\ \ \scalebox{.65}{2}}{O}{c}{F}{T}{L}}%
}}

\newcommand\hackk{\mathrel{%
  \ensurestackMath{\stackengine{-9pt}{\phantom{\scalebox{.675}{\txout}}}{\scalebox{.65}{3}}{O}{c}{F}{T}{L}}%
}}

\newcommand{\N}{\mathbb{N}}

\newcommand{\R}{\mathbb{R}}
\newcommand{\C}{\mathbb{C}}

\newcommand{\pr}{\mathbb{P}}
\newcommand{\E}{\mathbb{E}}
\newcommand{\CE}{\mathscr{E}}

\newcommand{\mcal}[1]{\mathcal{#1}}

\newcommand{\deq}{\overset{d}{=}}

\newcommand{\nas}{\text{ (mod } \varphi)}

\newcommand{\source}{\operatorname{src}}
\newcommand{\target}{\operatorname{tar}}

\newcommand{\interior}[1]{\accentset{\circ}{#1}}

\newcommand{\wtilde}[1]{\widetilde{#1}}

\newcommand{\txin}{\text{in}}
\newcommand{\txout}{\text{out}}
\newcommand{\txio}{\text{in/out}}

\newcommand{\vin}{v_{\operatorname{in}}}
\newcommand{\vout}{v_{\operatorname{out}}}
\newcommand{\mbf}[1]{\mathbf{#1}}
\newcommand{\mfk}[1]{\mathfrak{#1}}
\newcommand{\op}[1]{\operatorname{#1}}

\newcommand{\matN}{\op{Mat}_N}

\newcommand{\ssgm}[1]{\mcal{G}\langle\mbf{#1}, \mbf{#1}^*\rangle}
\newcommand{\ssgp}[1]{\C\mcal{G}\langle\mbf{#1}, \mbf{#1}^*\rangle}

\newcommand{\foutput}{\underset{\txout}{\cdot}}
\newcommand{\finput}{\underset{\txin}{\cdot}}
\newcommand{\Tr}{{\Tra}}

\renewcommand{\Re}{\operatorname{Re}}
\renewcommand{\Im}{\operatorname{Im}}

\DeclareMathOperator{\Tra}{Tr}

\DeclareMathOperator{\Cat}{Cat}

\DeclareRobustCommand{\SkipTocEntry}[5]{} 

\makeatletter
\newcommand*{\centerfloat}{%
  \parindent \z@
  \leftskip \z@ \@plus 1fil \@minus \textwidth
  \rightskip\leftskip
  \parfillskip \z@skip}
\makeatother

\begin{document}

\author{Benson Au}
\address{University of California, San Diego\\
         Department of Mathematics\\
         9500 Gilman Drive \# 0112\\
         La Jolla, CA 92093-0112\\
         USA}
       \email{\href{mailto:bau@ucsd.edu}{bau@ucsd.edu}}
\author{Camille Male}
\address{Universit\'{e} de Bordeaux\\
         Institut de Math\'{e}matiques de Bordeaux\\
         351 Cours de la Lib\'{e}ration\\
         33400 Talence,
         France}
       \email{\href{mailto:camille.male@math.u-bordeaux.fr}{camille.male@math.u-bordeaux.fr}}
\date{\today}
\title[Rigid structures in the UE traffic space]{Rigid structures in the universal enveloping traffic space}


\begin{abstract}\label{abstract}
For any tracial non-commutative probability space $(\mcal{A}, \varphi)$, C\'{e}bron, Dahlqvist, and Male showed that one can always construct an enveloping traffic space $(\mcal{G}(\mcal{A}), \tau_\varphi)$ that extends the trace. This construction provides a universal object that allows one to appeal to the traffic probability framework in generic situations, prioritizing an understanding of its structure. In this article, we prove that $(\mcal{G}(\mcal{A}), \tau_\varphi)$ admits a canonical free product decomposition $\mcal{A} * \mcal{A}^\intercal * \Theta(\mcal{G}(\mcal{A}))$. In particular, $\mcal{A}^\intercal$ is an anti-isomorphic copy of $\mcal{A}$, and $\Theta(\mcal{G}(\mcal{A}))$ is, up to degeneracy, a commutative algebra generated by Gaussian random variables with a covariance structure diagonalized by the graph operations. If $(\mcal{A}, \varphi)$ itself is a free product, then we describe how this additional structure lifts into $(\mcal{G}(\mcal{A}), \tau_\varphi)$. Here, we find a connection between free independence and classical independence opposite the usual direction. Up to degeneracy, we further show that $(\mcal{G}(\mcal{A}), \tau_\varphi)$ is spanned by tree-like graph operations. Finally, we apply our results to the study of large (possibly dependent) random matrices. Our analysis relies on the combinatorics of cactus graphs and the resulting cactus-cumulant correspondence.
\end{abstract}

\maketitle
\tableofcontents

\newpage
\section{Introduction}\label{sec:intro}
\subsection{Motivation}\label{sec:motivation}
Non-commutative (NC) probability theory is a generalization of classical probability theory to non-commuting random variables. Historically, the spectacular success of this perspective is due to free probability, where Voiculescu's free independence plays the central role \cite{VDN92}. Notably, Voiculescu proved that free independence governs the large dimension limit behavior of independent unitarily invariant random matrices \cite{Voi91}. Thus, free independence emerges from classical independence via a natural process of ``non-commutification'': we formulate this relationship in the heuristic equation
\begin{equation}\label{eq:non_commutification}
  \text{classical independence} \quad \xRightarrow{\overset{\otimes}{N \to \infty}} \quad \text{free independence.}
\end{equation}
By now, free probability techniques in random matrix theory are known to comprise a robust and increasingly influential set of tools \cite{MS17}.

The motivation for traffic probability comes from the following result in \cite{Mal17}: for a natural class of independent Wigner-like matrices possessing only a discrete distributional symmetry (permutation invariance), one observes a novel behavior in the determination of the joint spectral distribution from the marginals. This behavior was subsequently shown to be a special case of a more general phenomenon now known as \emph{traffic independence}. Notably, traffic independence governs the large dimension limit behavior of independent \emph{permutation invariant} random matrices \cite{Mal11}.

The axiomatization of traffic probability adjoins the standard NC framework with an operad structure based on graph observables. Accordingly, the notion of a so-called traffic distribution enriches that of a usual distribution. At the same time, one can encode the familiar notions of NC independence in the traffic framework. Understanding the relationship between these different notions provides new insight into the spectral behavior of large random matrices \cite{Mal11,CDM16,ACDGM18}.

Roughly speaking, a traffic space is a NC probability space with the additional structure to evaluate graph operations in the random variables. But when does this additional structure actually exist? Surprisingly, one can always augment a tracial NC probability space $(\mcal{A}, \varphi)$ to a traffic space $(\mcal{G}(\mcal{A}), \tau_\varphi)$ via an embedding $\mcal{A} \hookrightarrow \mcal{G}(\mcal{A})$ that extends the trace $\tau_\varphi|_{\mcal{A}} = \varphi$. For this reason, the pair $(\mcal{G}(\mcal{A}), \tau_\varphi)$ is called the \emph{universal enveloping (UE) traffic space}. Crucially, this construction provides a consistent intertwining limit object \cite{CDM16}:
\begin{enumerate}[label=(\roman*)]
\item \label{cdm_r1} extending the convergence in distribution of unitarily invariant random matrices in $(\mcal{A}, \varphi)$ to convergence in traffic distribution in $(\mcal{G}(\mcal{A}), \tau_\varphi)$;
\item \label{cdm_r2} lifting free independence in $(\mcal{A}, \varphi)$ to traffic independence in $(\mcal{G}(\mcal{A}), \tau_\varphi)$.
\end{enumerate}

In this article, we continue the investigation into the structure of the UE traffic space. Our main result shows that $(\mcal{G}(\mcal{A}), \tau_\varphi)$ comes equipped with a canonical free product structure, with implications for random matrices via \ref{cdm_r1}. 
We also complement the lifting in \ref{cdm_r2} by showing that a free product structure in $(\mcal{A}, \varphi)$ can be ``commutified'' in a natural way in $(\mcal{G}(\mcal{A}), \tau_\varphi)$ to produce classical independence (cf. \eqref{eq:non_commutification}). We then extend our results to a more general class of traffic distributions by establishing a correspondence between cactus graphs and free cumulants. We review the necessary background in the next section.

\subsection{Background}\label{sec:background}

First, we recall the basic framework in the NC setting \cite{NS06}.

\begin{defn}[Free probability]\label{defn:free_probability}
A \emph{non-commutative probability space (ncps)} is a pair $(\mcal{A}, \varphi)$ with $\mcal{A}$ a unital algebra over $\C$ and $\varphi: \mcal{A} \to \C$ a unital linear functional. We say that $\varphi$ is \emph{tracial} if it vanishes on the commutators. If $\mcal{A}$ has the additional structure of a $*$-algebra and $\varphi$ is \emph{positive} in the sense that $\varphi(a^*a) \geq 0$ for any $a \in \mcal{A}$, then we say that $(\mcal{A}, \varphi)$ is a \emph{$*$-probability space ($*$-ps)}.

The \emph{(joint) distribution} of a family of random variables $\mbf{a} = (a_i)_{i \in I} \subset (\mcal{A}, \varphi)$ is the linear functional
\[
  \mu_{\mbf{a}} : \C\langle \mbf{x} \rangle \to \C, \qquad P \mapsto \varphi(P(\mbf{a})),
\]
where $\mbf{x} = (x_i)_{i \in I}$ is a family of non-commuting indeterminates and $P(\mbf{a})$ is the usual evaluation of NC polynomials. We say that a sequence of families $\mbf{a}_n = (a_i^{(n)})_{i \in I} \subset (\mcal{A}_n, \varphi_n)$ \emph{converges in distribution} if the sequence of functionals $\mu_{\mbf{a}_n}$ converges pointwise. If $(\mcal{A}, \varphi)$ has the additional structure of a $*$-algebra, then we define the \emph{(joint) $*$-distribution} of $\mbf{a}$ as the linear functional
\[
  \nu_{\mbf{a}} : \C\langle \mbf{x}, \mbf{x}^* \rangle \to \C, \qquad Q \mapsto \varphi(Q(\mbf{a})),
\]
where $Q(\mbf{a})$ is the usual evaluation of NC $*$-polynomials. We say that a sequence of families $\mbf{a}_n = (a_i^{(n)})_{i \in I} \subset (\mcal{A}_n, \varphi_n)$ \emph{converges in $*$-distribution} if the sequence of functionals $\nu_{\mbf{a}_n}$ converges pointwise.

Unital subalgebras $(\mcal{A}_i)_{i \in I}$ of a ncps $(\mcal{A}, \varphi)$ are said to be \emph{classically independent} if they commute $[\mcal{A}_i, \mcal{A}_{i'}] = 0$ for $i \neq i'$ and the expectation on a product factors $\varphi(\prod_{j = 1}^n a_j) = \prod_{j = 1}^n \varphi(a_j)$ whenever the elements $a_j \in \mcal{A}_{i(j)}$ belong to distinct subalgebras $i(j) \neq i(k)$ for $j \neq k$. In contrast, the $(\mcal{A}_i)_{i \in I}$ are said to be \emph{freely independent} (or simply \emph{free}) if $\varphi(\prod_{j = 1}^n a_j) = 0$ for centered elements $a_j \in \interior{\mcal{A}}_{i(j)} = \{a \in \mcal{A}_{i(j)} : \varphi(a) = 0\}$ belonging to consecutively distinct subalgebras $i(1) \neq i(2) \neq \cdots \neq i(n)$. Subsets $(\mcal{S}_i)_{i \in I}$ of $(\mcal{A}, \varphi)$ are said to be classically independent (resp., free) if the unital subalgebras that they generate are classically independent (resp., free). If $\mcal{A}$ is a $*$-algebra, then we say that the subsets $(\mcal{S}_i)_{i \in I}$ are $*$-free if the unital $*$-subalgebras that they generate are free.

Let $(\mcal{NC}(n), \leq)$ denote the poset of non-crossing partitions of $[n]$ with the reversed refinement order and $\mu$ the corresponding M\"{o}bius function. We write $0_n$ for the minimal element consisting of singletons and $1_n$ for the maximal element consisting of a single block. For a non-crossing partition $\pi \in \mcal{NC}(n)$, we define the multilinear functional
\[
  \varphi_\pi : \mcal{A}^n \to \C, \qquad \varphi_\pi[a_1, \ldots, a_n] = \prod_{B \in \pi} \varphi(B)[a_1, \ldots, a_n],
\] 
where a block $B = (i_1 < \cdots < i_m) \in \pi$ defines a partial product
\[
  \varphi(B)[a_1, \ldots, a_n] = \varphi\Big(\prod_{j = 1}^m a_{i_j}\Big).
\]
The \emph{free cumulant} $\kappa_\pi$ is the multilinear functional $\kappa_\pi : \mcal{A}^n \to \C$ given by the M\"{o}bius convolution
\begin{equation}\label{eq:free_cumulants_mobius}
  \kappa_\pi[a_1, \ldots, a_n] = \sum_{\substack{\sigma \in \mcal{NC}(n) \\ \text{s.t. } \sigma \leq \pi}} \varphi_\sigma[a_1, \ldots, a_n] \mu(\sigma, \pi).
\end{equation}
One can also define the free cumulants recursively using the equivalent relationship 
\begin{equation}\label{eq:free_cumulants_sum}
  \varphi_\pi[a_1, \ldots, a_n] = \sum_{\substack{\sigma \in \mcal{NC}(n) \\ \text{s.t. } \sigma \leq \pi}} \kappa_\sigma[a_1, \ldots, a_n].
\end{equation}
We distinguish the free cumulant on the maximal partition $1_n$ with the notation $\kappa_n = \kappa_{1_n}$. This allows us to formulate the \emph{multiplicative} property of the free cumulants:
\[
  \kappa_\pi[a_1, \ldots, a_n] = \prod_{B \in \pi} \kappa(B)[a_1, \ldots, a_n],
\] 
where $B = (i_1 < \cdots < i_m)$ is a block as before and
\[
  \kappa(B)[a_1, \ldots, a_n] = \kappa_m[a_{i_1}, \ldots, a_{i_m}].
\]
Thus, the full set of free cumulants $(\kappa_\pi)_{\pi \in \mcal{NC}(n), n \in \N}$ can be recovered from $(\kappa_n)_{n \in \N}$. Furthermore, the vanishing of mixed cumulants characterizes free independence: subsets $(\mcal{S}_i)_{i \in I}$ are freely independent iff for any $n \geq 2$ and $a_1, \ldots, a_n$ such that $a_j \in \mcal{S}_{i(j)}$,
\[
  \exists i(j) \neq i(k) \implies \kappa_n[a_1, \ldots, a_n] = 0.
\]
\end{defn}

\begin{defn}[Conditional expectation]\label{defn:conditional_expectation}
Let $\mcal{B} \subset \mcal{A}$ be a unital $*$-subalgebra of a $*$-ps $(\mcal{A}, \varphi)$. A \emph{conditional expectation onto $\mcal{B}$} is a unital linear map $\CE: \mcal{A} \to \mcal{B}$ satisfying
\begin{enumerate}[label=(\roman*)]
\item \label{ce_expectation} $\varphi(\CE(a)) = \varphi(a)$ for any $a \in \mcal{A}$;
\item \label{ce_star} $\CE(a^*) = \CE(a)^*$ for any $a \in \mcal{A}$;
\item \label{ce_bimodule} $\CE(b_1ab_2) = b_1\CE(a)b_2$ for any $b_1, b_2 \in \mcal{B}$ and $a \in \mcal{A}$.
\end{enumerate}
\end{defn}

\begin{defn}[Equivalence mod $\varphi$]\label{defn:equivalence_mod_phi}
Let $(\mcal{A}, \varphi)$ be a tracial $*$-ps. The traciality of $\varphi$ implies that the subspace of degenerate elements
\[
\mcal{D} = \{a \in \mcal{A} : \varphi(ab) = 0 \text{ for every } b \in \mcal{A}\}
\]
further has the structure of a two-sided $*$-ideal. We say that two random variables $a, b \in \mcal{A}$ are \emph{equal up to degeneracy} if $a - b \in \mcal{D}$, for which we use the notation $a \equiv b \nas$. This equivalence conforms with our intuition from the classical setting: for example, if $a \equiv b \nas$, then one can interchange $a$ and $b$ in a joint $*$-distribution or free cumulant without consequence.
\end{defn}

The traffic framework requires some preliminary definitions.

\begin{defn}[Graphs]\label{defn:graphs}
In this article, a \emph{multidigraph} $G = (V, E, \source, \target)$ consists of a finite non-empty set of vertices $V$, a finite set of edges $E$, and functions $\source, \target: E \to V$ indicating the \emph{source} and \emph{target} of each edge. We omit these functions from the notation and simply write $G = (V, E)$ when convenient. We say that $G$ is \emph{bi-rooted} if it has a pair of distinguished (not necessarily distinct) vertices $(\vin, \vout) \in V^2$, the coordinates of which we call the \emph{input} and the \emph{output}.

A \emph{graph operation} is a connected, bi-rooted multidigraph
\[
  g = (V, E, \source, \target, \vin, \vout, o)
\]
together with an ordering of its edges $o: E \leftrightarrow [\#(E)]$. We interpret $g = g(\cdot_1, \ldots, \cdot_K)$ as a function of $K = \#(E)$ arguments, one for each edge $e \in E$, with coordinates specified by the ordering. In particular, we call such a graph a \emph{$K$-graph operation}. We write $\mcal{G}_K$ for the set of all $K$-graph operations and $\mcal{G} = \cup_{K \geq 0} \mcal{G}_K$ for the set of all graph operations.
\end{defn}

\begin{eg}[Graph operations]\label{eg:graph_operations}
We introduce some conventions for depicting graph operations that the reader will hopefully discern from the examples below. In particular, we enumerate
\[
\mcal{G}_0 = \bigg\{\underset{\txio}{\cdot}\bigg\}
\]
and
\[
\mcal{G}_1 = \bigg\{ \hspace{3.25pt} \underset{\txout}{\cdot} \overset{1}{\leftarrow} \underset{\txin}{\cdot} \quad , \quad \underset{\txout}{\cdot} \overset{1}{\rightarrow} \underset{\txin}{\cdot} \quad , \quad \underset{\txio}{\phantom{\overset{1}{\phantom{\ccdot}}}\vvuparrow \overset{1}{\phantom{\ccdot}}} \quad , \quad \underset{\txio}{\phantom{\overset{1}{\phantom{\ccdot}}}\vvdownarrow \overset{1}{\phantom{\ccdot}}} \quad , \quad \underset{\txio}{\overset{1}{\dcirclearrowleft}} \bigg\}.
\]
When there is little ambiguity, we omit the ordering of the edges in the figure. For instance, this can be done in the examples above. For a slightly less trivial example, consider the graph operation $g(\cdot_1, \cdot_2) = \underset{\txout}{\cdot} \leftleftarrows \underset{\txin}{\cdot}$\ . Nevertheless, we emphasize the importance of the ordering in distinguishing distinct graph operations: for example,
\[
\underset{\txout}{\cdot} \overset{1}{\leftarrow} \cdot \overset{2}{\leftarrow} \underset{\txin}{\cdot} \quad \neq \quad \underset{\txout}{\cdot} \overset{2}{\leftarrow} \cdot \overset{1}{\leftarrow} \underset{\txin}{\cdot}
\]
\end{eg}

One can define an action of the symmetric group on the graph operations by permuting the ordering of the edges. Formally, for a permutation $\sigma \in \mfk{S}_K$ and a graph operation $g \in \mcal{G}_K$ as before, the permuted graph operation
\[
  g_\sigma = (V, E, \source, \target, \vin, \vout, \sigma \circ o).
\]
The set of graph operations $\mcal{G}$ then has the structure of a symmetric operad \cite{Yau18}.

\begin{defn}[Operad of graph operations]\label{defn:operad_graph_operations}
Let $g = (V, E, \source, \target, \vin, \vout, o)$ be a $K$-graph operation. For a $K$-tuple of graph operations $(g_1, \ldots, g_K)$ with
\[
g_i = (V_i, E_i, \source_i, \target_i, \vin^{(i)}, \vout^{(i)}, o_i) \in \mcal{G}_{L_i},
\]
we define the \emph{composite graph operation}
\[
g(g_1, \ldots, g_K) \in \mcal{G}_{\sum_{i=1}^K L_i}
\]
by substitution. Formally, one removes each edge $e \in E$ and installs a copy of $g_{o(e)}$ in its place by identifying the vertices $\source(e) \sim \vin^{(o(e))}$ and $\target(e) \sim \vout^{(o(e))}$. The composite $g(g_1, \ldots, g_K)$ then inherits the natural ordering of its edges. 

The reader can easily verify that this composition is \emph{associative}, i.e.,
\begin{align*}
  &g(g_1(g_{1, 1}, \ldots, g_{1, L_1}), \ldots, g_K(g_{K,1}, \ldots, g_{K, L_K})) \\
  &= (g(g_1, \ldots, g_K))(g_{1, 1}, \ldots, g_{1, L_1}, \ldots, g_{K,1}, \ldots, g_{K, L_K});
\end{align*}
and \emph{equivariant}, i.e.,
\begin{alignat*}{2}
g_\sigma(g_1, \ldots, g_K) &= (g(g_{\sigma(1)}, \ldots, g_{\sigma(K)}))_{\pi(\sigma)}, \qquad &&\forall \sigma \in \mfk{S}_K; \\
g((g_1)_{\sigma_1}, \ldots, (g_K)_{\sigma_K}) &= g(g_1, \ldots, g_K)_{\sigma_1 \oplus \cdots \oplus \sigma_K}, \qquad &&\forall \sigma_i \in \mfk{S}_{L_i},
\end{alignat*}
where $\oplus$ denotes the direct sum of permutations and\footnotesize
\[
\pi(\sigma) = \prod_{i = 1}^K
\begin{pmatrix}
\sum_{j = 1}^{\sigma^{-1}(i) - 1} L_{\sigma(j)} + 1 & \sum_{j = 1}^{\sigma^{-1}(i) - 1} L_{\sigma(j)} + 2  & \cdots & \sum_{j = 1}^{\sigma^{-1}(i)} L_{\sigma(j)} \\
\sum_{j = 1}^{i - 1} L_j + 1 & \sum_{j = 1}^{i - 1} L_j + 2 & \cdots & \sum_{j = 1}^{i} L_j
\end{pmatrix}
\in \mfk{S}_{\sum_{i = 1}^K L_i}.
\]\normalsize
The graph operation $\text{id}_{\mcal{G}} = \underset{\txout}{\cdot} \leftarrow \underset{\txin}{\cdot} \in \mcal{G}_1$ is the \emph{unit} for this composition, namely,
\[
\text{id}_{\mcal{G}}(g) = g(\text{id}_{\mcal{G}}, \ldots, \text{id}_{\mcal{G}}) = g, \qquad \forall g \in \mcal{G}.
\]
\end{defn}

\begin{eg}[Composite graph operation]\label{eg:substitution}
If
\[
  g = \underset{\txout}{\cdot} \ \underset{2}{\overset{1}{\sleftrightarrows}} \ \cdot \underset{3}{\rightarrow}\ \underset{\txin}{\overset{4}{\dcirclearrowleft}}
\]
and
\[
  g_1 = \underset{\txio}{\overset{1}{\dcirclearrowleft}} \quad , \quad  g_2 = \underset{\txout}{\cdot} \overset{2}{\rightarrow} \cdot \overset{1}{\rightarrow} \underset{\txin}{\cdot} \quad , \quad
  g_3 = \underset{\txout}{\cdot} \overset{1}{\leftarrow} \underset{\txin}{\cdot} \quad , \quad g_4 = \underset{\txio}{\phantom{\overset{1}{\phantom{\ccdot}}}\vvuparrow \overset{1}{\phantom{\ccdot}}} \quad ,
\]
then
\[
g(g_1, g_2, g_3, g_4) \hspace{2.5pt} = \hspace{2.5pt} \ \hack \ \overset{1}{\loopdown} \hackk \hspace{-5pt}\overset{4}{\rightarrow} \ \underset{\txin}{\overset{5}{\dcirclearrowleft}} 
\]
\end{eg}

\begin{defn}[$\mcal{G}$-algebra]\label{defn:G_algebra}
A \emph{$\mcal{G}$-algebra} is a complex vector space $\mcal{A}$ together with an action $(Z_g)_{g \in \mcal{G}}$ of the operad of graph operations. By this, we mean that each graph operation $g \in \mcal{G}_K \subset \mcal{G}$ defines a multilinear map
\[
Z_g : \mcal{A}^K \to \mcal{A}
\]
satisfying the following properties:
\begin{enumerate}[label=(\roman*)]
\item (Associativity) The action of the composite graph operation $Z_{g(g_1, \ldots, g_K)}$ factors through the action of the graph operations $Z_g$ and $(Z_{g_i})_{i \in [K]}$:
\[
Z_{g(g_1, \ldots, g_K)} = Z_g \circ (Z_{g_1} \times \cdots \times Z_{g_K});
\]
\item (Equivariance) The ordering of the edges only plays a formal role in defining the action of a graph operation by assigning the location of each argument to a specific edge. Any equivalent assignment of edge locations produces the same action:
\[
Z_{g_\sigma} = Z_g \circ P_\sigma, \qquad \forall \sigma \in \mfk{S}_K,
\]
where
\[
P_\sigma : \mcal{A}^{K} \to \mcal{A}^{K}, \qquad (a_1, \ldots, a_K) \mapsto (a_{\sigma(1)}, \ldots, a_{\sigma(K)});
\]
\item \label{identity} (Identity) The operad unit $\text{id}_{\mcal{G}} = \underset{\txout}{\cdot} \leftarrow \underset{\txin}{\cdot}$ defines the stable action
\[
Z_{\text{id}_{\mcal{G}}}: \mcal{A} \to \mcal{A}, \qquad a \mapsto a.
\]
\end{enumerate}
A \emph{sub-$\mcal{G}$-algebra} $\mcal{B}$ is a subspace $\mcal{B} \subset \mcal{A}$ that is closed under the action of the graph operations. The \emph{$\mcal{G}$-algebra generated by a subset $\mcal{S} \subset \mcal{A}$} is the smallest $\mcal{G}$-algebra containing $\mcal{S}$, which can be characterized as the span of $\bigcup_{K \geq 0} \bigcup_{g \in \mcal{G}_K} Z_g(\mcal{S}^{K})$. A \emph{morphism of $\mcal{G}$-algebras} (or \emph{$\mcal{G}$-morphism}) is a linear map $f: \mcal{A} \to \mcal{B}$ between $\mcal{G}$-algebras that respects the action of the graph operations, namely,
\[
f \circ Z_g = Z_g \circ (f \times \cdots \times f), \qquad \forall g \in \mcal{G}.
\]
\end{defn}

Symbolically, we represent the action $Z_g$ on a $K$-tuple $(a_1, \ldots, a_K)$ by placing each argument $a_i$ in the location prescribed by the ordering. For example, we can formulate the identity axiom \ref{identity} as
\[
\foutput \overset{a}{\leftarrow} \finput = a, \qquad \forall a \in \mcal{A}.
\]
A $\mcal{G}$-algebra structure on $\mcal{A}$ defines an algebra structure on $\mcal{A}$ via the product
\begin{equation}\label{eq:G_algebra_multiplication}
a \cdot_{\mcal{G}} b := \foutput \overset{a}{\leftarrow} \cdot \overset{b}{\leftarrow} \finput \ .
\end{equation}
The identity $1_{\mcal{A}} = Z_{\underset{\txio}{\cdot}}(1)$ comes from the action $Z_{\underset{\txio}{\cdot}} : \C \to \mcal{A}$ of the trivial graph operation $\underset{\txio}{\cdot} \in \mcal{G}_0$. The reader should verify that the axioms of a $\mcal{G}$-algebra ensure the well-definedness of this algebra structure. In particular, a $\mcal{G}$-morphism $f: \mcal{A} \to \mcal{B}$ is also a morphism of unital algebras.

One can define a natural pair of involutions on the operad of graph operations. For a graph operation $g = (V, E, \source, \target, \vin, \vout, o)$, one obtains the \emph{transpose} $g^\intercal = (V, E, \source, \target, \vout, \vin, o)$ by interchanging the input and output. One obtains the \emph{flip} $g_{\rightarrow} = (V, E, \target, \source, \vin, \vout, o)$ by interchanging the maps $\source$ and $\target$, reversing the direction of each edge $e \in E$. If our $\mcal{G}$-algebra also comes equipped with an involution $\cdot^*$, then we can further ask that these operations obey a natural adjoint relation. This leads to the following definition.

\begin{defn}[$\mcal{G}^*$-algebra]\label{defn:G*_algebra}
Let $\mcal{A}$ be a $\mcal{G}$-algebra with a conjugate linear involution $*: \mcal{A} \to \mcal{A}$. We say that $\mcal{A}$ is a \emph{$\mcal{G}^*$-algebra} if
\[
Z_{g_{\rightarrow}^\intercal} \circ (* \times \cdots \times *) = * \circ Z_g, \qquad \forall g \in \mcal{G}.
\]
A \emph{sub-$\mcal{G}^*$-algebra} $\mcal{B}$ is a $*$-subspace $\mcal{B} \subset \mcal{A}$ that is closed under the action of the graph operations. The \emph{$\mcal{G}^*$-algebra generated by a subset $\mcal{S} \subset \mcal{A}$} is the smallest $\mcal{G}^*$-algebra containing $\mcal{S}$, which can characterized as the span of $\bigcup_{K \geq 0} \bigcup_{g \in \mcal{G}_K} Z_g((\mcal{S}\cup\mcal{S}^*)^{K})$. A \emph{morphism of $\mcal{G}^*$-algebras} (or \emph{$\mcal{G}^*$-morphism}) is a $\mcal{G}$-morphism $f: \mcal{A} \to \mcal{B}$ between $\mcal{G}^*$-algebras that further respects the involution operations $*_{\mcal{B}} \circ f = f \circ *_{\mcal{A}}$.
\end{defn}

\begin{eg}[$*$-graph polynomial]\label{eg:*-graph_polynomial}
A \emph{$*$-graph monomial $t = (G, \gamma, \varepsilon)$ in $\mbf{x} = (x_i)_{i \in I}$} is a bi-rooted multidigraph $G = (V, E, \source, \target, \vin, \vout)$ with edge labels $\gamma: E \to I$ and $\varepsilon: E \to \{1, *\}$ in $\langle\mbf{x}, \mbf{x}^*\rangle$. We define the \emph{transpose} $t^\intercal$ as before. We also define the \emph{conjugate} $\overline{t} = (V, E, \target, \source, \vin, \vout, \gamma, \varepsilon^*)$ as the $*$-flip , which flips both the direction and the $*$-label of each edge. Finally, we define the \emph{adjoint} $t^*$ as the conjugate transpose $\overline{t}^\intercal$. We write $\ssgm{x}$ for the set of all $*$-graph monomials and $\ssgp{x}$ for the complex vector space spanned by $\ssgm{x}$, the so-called \emph{$*$-graph polynomials}. We extend the adjoint operation to a conjugate linear involution $*: \ssgp{x} \to \ssgp{x}$.

The reader should verify that the $*$-graph polynomials $\ssgp{x}$ form a $\mcal{G}^*$-algebra under the action of composition: for $*$-graph monomials $t_1, \ldots, t_K$, we define $Z_g(t_1, \ldots, t_K)$ as the $*$-graph monomial obtained by concatenating the $t_i$ according to $g$ as in the composite graph construction $g(g_1, \ldots, g_K)$. The $*$-graph polynomials generalize the usual $*$-polynomials. In particular, one obtains an embedding of unital $*$-algebras $\eta: \C\langle\mbf{x}, \mbf{x}^*\rangle \hookrightarrow \ssgp{x}$,
\begin{equation}\label{eq:polynomial_embed}
x_i \mapsto \foutput \xleftarrow{x_i} \finput \qquad \text{and} \qquad 1 \mapsto \underset{\txio}{\cdot} \ .
\end{equation}
\end{eg}

\begin{eg}[Graph of matrices \cite{Jon99,BS10,MS12}]\label{eg:graph_of_matrices}
Let $\matN(\C)$ denote the $*$-algebra of $N \times N$ matrices over $\C$. For $g \in \mcal{G}_K$, we define the \emph{graph of matrices}
\[
Z_g: \matN(\C)^{K} \to \matN(\C)
\]
by the coordinate formula\small
\begin{equation}\label{eq:coordinate_formula}
Z_g(\mbf{A}_N^{(1)}, \ldots, \mbf{A}_N^{(K)})(i, j) = \sum_{\substack{\phi: V \to [N] \text{ s.t.}\\ \phi(\vout) = i, \ \phi(\vin) = j}} \prod_{e \in E} \mbf{A}_N^{(o(e))}(\phi(\target(e)), \phi(\source(e))).
\end{equation}\normalsize
For notational convenience, we abbreviate $(\phi(\target(e)), \phi(\source(e))) = \phi(e)$. Note that the action \eqref{eq:coordinate_formula} defines a $\mcal{G}^*$-algebra structure on $\matN(\C)$ that recovers the usual matrix multiplication:
\[
\mbf{A}_N \cdot_{\mcal{G}} \mbf{B}_N = \foutput \xleftarrow{\mbf{A}_N} \cdot \xleftarrow{\mbf{B}_N} \finput = \mbf{A}_N\mbf{B}_N.
\]
The action of the graph operations also produces matrices of additional linear algebraic structure: for example,
\begin{enumerate}
\item (Transpose) For any $g \in \mcal{G}$,
\[
Z_{g^\intercal} = \intercal \circ Z_g,
\]
where on the right-hand side of the equality we have used the same notation $\intercal$ for the usual matrix transpose. In particular,
\[
\foutput \xrightarrow{\mbf{A}_N} \finput = \mbf{A}_N^\intercal;
\]
\item (Hadamard-Schur product) Parallel edges correspond to entrywise products. In particular,  
\[
\foutput \overset{\mbf{A}_N}{\underset{\mbf{B}_N}{\leftleftarrows}} \finput = \mbf{A}_N \circ \mbf{B}_N = (\mbf{A}_N(i,j)\mbf{B}_N(i,j))_{1 \leq i, j \leq N}; 
\]
\item (Diagonal) The action of a graph operation with $\vin = \vout$ (a so-called \emph{diagonal graph operation}) produces a diagonal matrix. In particular,
\[
\underset{\txio}{\overset{\mbf{A}_N}{\dcirclearrowleft}} \hspace{3pt} = \Delta(\mbf{A}_N) = (\mbf{A}_N(i, i))_{1 \leq i \leq N};
\]
\item (Degree) Similarly, one can obtain the diagonal matrix of row sums as
\[
\underset{\txio}{\phantom{\overset{\mbf{A}_N}{\phantom{\ccdot}}}\vvdownarrow \overset{\mbf{A}_N}{\phantom{\ccdot}}} = \op{rDeg}(\mbf{A}_N) = \Big(\sum_{j = 1}^N \mbf{A}_N(i,j)\Big)_{1 \leq i \leq N}.
\]
Reversing the direction of the lone edge in this graph yields the column sums $\op{cDeg}(\mbf{A}_N) $ instead.
\end{enumerate}
\end{eg}

Of course, the example above applies equally well to \emph{random matrices} with the appropriate modifications, namely, $L^{\infty-}(\Omega, \mcal{F}, \pr) \otimes \matN(\C)$. We adapt the notation from the matricial setting to general $\mcal{G}$-algebras: for example,
\[
a^\intercal = \foutput \overset{a}{\rightarrow} \finput \quad , \quad a \circ b = \foutput \overset{a}{\underset{b}{\leftleftarrows}} \finput \quad , \quad \Delta(a) = \hspace{3pt} \underset{\txio}{\overset{a}{\dcirclearrowleft}} \quad ,
\]
and so forth. The map $\Delta$ then defines a projection $\Delta = \Delta^2$ on a $\mcal{G}$-algebra $\mcal{A}$ whose image $\Delta(\mcal{A})$ is a commutative sub-$\mcal{G}$-algebra, the so-called \emph{diagonal algebra of $\mcal{A}$}.

Note that the trace of a graph of matrices $\Tr[Z_g(\mbf{A}_N^{(1)}, \ldots, \mbf{A}_N^{(K)})]$ only depends on the graph operation $g = (V, E, \source, \target, \vin, \vout, o)$ up to the \emph{unrooted} graph
\[
T = \wtilde{\Delta}(g) := (\wtilde{V}, E, \source, \target, o)
\]
obtained from $g$ by identifying the input and the output $\vin \sim \vout$ and forgetting their distinguished roles. Indeed,
\begin{align*}
\Tr[Z_g(\mbf{A}_N^{(1)}, \ldots, \mbf{A}_N^{(K)})] &= \sum_{i = 1}^N Z_g(\mbf{A}_N^{(1)}, \ldots, \mbf{A}_N^{(K)})(i, i) \\
&= \sum_{i = 1}^N  \sum_{\substack{\phi: V \to [N] \text{ s.t.}\\ \phi(\vout) = \phi(\vin) = i}} \prod_{e \in E} \mbf{A}_N^{(o(e))}(\phi(e)) \\
&= \sum_{\phi: \wtilde{V} \to [N]} \prod_{e \in E} \mbf{A}_N^{(o(e))}(\phi(e)) \\
&=: \Tr[T(\mbf{A}_N^{(1)}, \ldots, \mbf{A}_N^{(K)})],
\end{align*}
where in the last equality we define $\Tr[T(\mbf{A}_N^{(1)}, \ldots, \mbf{A}_N^{(K)})]$ as the appropriate sum. This observation motivates the notion of a distribution in the traffic framework, which requires the following definition.

\begin{defn}[Test graph]\label{defn:test_graph}
A \emph{test graph $T = (G, \gamma)$ in $\mcal{S}$} is a connected multidigraph $G = (V, E)$ with edge labels $\gamma: E \to \mcal{S}$. A \emph{$*$-test graph} $T = (G, \gamma, \varepsilon)$ in $\mcal{S}$ is a test graph with the additional information of $*$-labels $\varepsilon: E \to \{1, *\}$. We write $\mcal{T}\langle\mcal{S}\rangle$ (resp., $\mcal{T}\langle\mcal{S}, \mcal{S}^*\rangle$) for the set of all test graphs (resp., $*$-test graphs) in $\mcal{S}$. We write $\C\mcal{T}\langle\mcal{S}\rangle$ (resp., $\C\mcal{T}\langle\mcal{S}, \mcal{S}^*\rangle$) for the complex vector space spanned by $\mcal{T}\langle\mcal{S}\rangle$ (resp., $\mcal{T}\langle\mcal{S},\mcal{S}^*\rangle$).
\end{defn}

We are now prepared to discuss the traffic probability framework.

\begin{defn}[Traffic probability]\label{defn:traffic_probability}  
An \emph{algebraic traffic space} is a pair $(\mcal{A}, \tau)$ with $\mcal{A}$ a $\mcal{G}$-algebra and $\tau: \C\mcal{T}\langle\mcal{A}\rangle \to \C$ a $\mcal{G}$-compatible linear functional in the following sense:
\begin{enumerate}[label=(\roman*)]
\item \label{unity} (Unity) The trivial test graph consisting of a single isolated vertex evaluates to 1, i.e.,
\[
\tau[\ \cdot \ ] = 1;
\] 
\item \label{substitution} (Substitution) The functional $\tau$ respects the $\mcal{G}$-action: in particular, if a test graph $T = (V, E, \source, \target, \gamma) \in \mcal{T}\langle\mcal{A}\rangle$ has an edge $e \in E$ with label
\[
\gamma(e) = a = Z_g(a_1, \ldots, a_K),
\]
then $\tau$ returns the same value on the test graph $T_{e,Z_g(a_1, \ldots, a_K)}$ obtained from $T$ by substituting the graph represented by the action $Z_g(a_1, \ldots, a_K)$ in for the edge $e$. Formally, one removes the edge $e \in E$ and installs a copy of the graph $Z_g(a_1, \ldots, a_K)$ in its place by identifying the vertices $\source(e) \sim \vin$ and $\target(e) \sim \vout$, in which case
\[
\tau[T] = \tau[T_{e, Z_g(a_1, \ldots, a_K)}].
\]
For example, if
\[
a = \underset{\txout}{\cdot} \overset{\raise.1em\hbox{$\scriptstyle a_1$}}{\rightarrow} \ \overset{a_2}{\dcirclearrowleft} \ \overset{\raise.1em\hbox{$\scriptstyle a_3$}}{\leftarrow} \underset{\txin}{\cdot} \qquad \text{and} \qquad b = \op{rDeg}(b_1) = \underset{\txio}{\phantom{\overset{b_1}{\phantom{\ccdot}}}\vvdownarrow \overset{b_1}{\phantom{\ccdot}}}\ ,
\]
then
\[
\tau\bigg[ \ \overset{a}{\dcirclearrowleft}\ \underset{c}{\leftarrow} \cdot \overset{b}{\underset{d}{\rightrightarrows}} \cdot \ \bigg] = \tau\bigg[ \ \overset{a_2}{\phantom{.}}\ \ssidecircle\ \ \overset{a_1}{\underset{a_3}{\leftleftarrows}} \cdot \underset{c}{\leftarrow}\ \dlcirclearrowleft \underset{d}{\overset{b_1}{\phantom{\ccdot}}} \ \bigg];
\]
\item \label{multilinearity} (Multilinearity) The functional $\tau$ is multilinear with respect to the edge labels. Formally, fixing the underlying multidigraph $G = (V, E)$ of a test graph $T = (G, \gamma) \in \mcal{T}\langle\mcal{A}\rangle$ defines a $\#(E)$-linear function
\[
\tau[T(\bigtimes_{e \in E} \cdot_e)] : \mcal{A}^E \to \C
\]
of the edges $E$ via the labels $\gamma: E \to \mcal{A}$.
\end{enumerate}
We refer to $\tau$ as the \emph{traffic state}. If $\mcal{A}$ has the additional structure of a $\mcal{G}^*$-algebra and $\tau$ is positive in the sense of \cite{CDM16}, then we say that $(\mcal{A}, \tau)$ is a \emph{traffic space}.

The \emph{traffic distribution} of a family of random variables $\mbf{a} = (a_i)_{i \in I} \subset (\mcal{A}, \tau)$ is the linear functional
\[
  \tau_{\mbf{a}} : \C\mcal{T}\langle \mbf{x} \rangle \to \C, \qquad T \mapsto \tau[T(\mbf{a})],
\]
where $T(\mbf{a}) \in \mcal{T}\langle\mcal{A}\rangle$ is the test graph obtained from $T \in \mcal{T}\langle\mbf{x}\rangle$ by replacing the edge labels $x_{\gamma(e)}$ with edge labels $a_{\gamma(e)}$. If $\mcal{A}$ has the additional structure of a $\mcal{G}^*$-algebra, then we define the \emph{$*$-traffic distribution} of $\mbf{a}$ as the linear functional
\[
  \upsilon_{\mbf{a}} : \C\mcal{T}\langle \mbf{x}, \mbf{x}^* \rangle \to \C, \qquad T \mapsto \tau[T(\mbf{a})],
\]
where $T(\mbf{a}) \in \mcal{T}\langle\mcal{A}\rangle$ is the test graph obtained from $T \in \mcal{T}\langle\mbf{x}, \mbf{x}^*\rangle$ by replacing the edge labels $x_{\gamma(e)}^{\varepsilon(e)}$ with edge labels $a_{\gamma(e)}^{\varepsilon(e)}$.

We say that a sequence of families $\mbf{a}_n = (a_i^{(n)})_{i \in I} \subset (\mcal{A}_n, \tau_n)$ \emph{converges in traffic distribution} (resp., in \emph{$*$-traffic distribution}) if the sequence of functionals $\tau_{\mbf{a}_n}$ (resp., $\upsilon_{\mbf{a}_n}$) converges pointwise. Note that one can always find a realization of the limit of a traffic convergent sequence $\mbf{a}_n$ via the induced traffic space $(\C\mcal{G}\langle\mbf{x}\rangle, \lim_{n \to \infty} \tau_{\mbf{a}_n})$.
\end{defn}

The traffic state $\tau$ of an algebraic traffic space $(\mcal{A}, \tau)$ induces a unital linear functional $\varphi_\tau: \mcal{A} \to \C$,
\begin{equation}\label{eq:trace_traffic_state}
  \varphi_\tau(a) = \varphi_\tau\Big( \ \foutput \overset{a}{\leftarrow} \finput \ \Big) = \tau\Big[\ \, \raisebox{-3pt}{$\overset{a}{\dcirclearrowleft}$} \, \ \Big].
\end{equation}
In particular, the substitution axiom implies that the expectation $\varphi_\tau$ of a random variable $a = Z_g(a_1, \ldots, a_K)$ only depends on the test graph $T$ obtained from $g$ by identifying the input of $g$ with the output of $g$ and forgetting their distinguished roles. We use the notation $T = \wtilde{\Delta}(a)$ for this test graph, which allows us to write
\[
\varphi_\tau(a) = \tau[\wtilde{\Delta}(a)] = \varphi_\tau(\Delta(a)).
\]
For example, this implies that the expectation $\varphi_\tau$ is necessarily tracial. Indeed,
\begin{align*}
  \varphi_\tau(ab) = \varphi_\tau\Big(  \ \foutput \overset{a}{\leftarrow} \cdot \overset{b}{\leftarrow} \finput \ \Big) &= \tau\Big[\ \cdot \ \overset{a}{\underset{b}{\srightleftarrows}} \ \cdot \ \Big] \\
  &= \tau\Big[\ \cdot \ \overset{b}{\underset{a}{\srightleftarrows}} \ \cdot \ \Big] = \varphi_\tau\Big( \ \foutput \overset{b}{\leftarrow} \cdot \overset{a}{\leftarrow} \finput \ \Big) = \varphi_\tau(ba).
\end{align*}
Thus, when referring to an algebraic traffic space $(\mcal{A}, \tau)$, we implicitly assume the induced tracial ncps structure $(\mcal{A}, \varphi_\tau)$ defined above. When there is little ambiguity, we omit the subscript $\tau$ and simply write $\varphi$. We often refer to elements $a \in \mcal{A}$ as \emph{traffic random variables} (or simply \emph{traffics}) to emphasize the (algebraic) traffic space structure. In the case of a traffic space $(\mcal{A}, \tau)$, the positivity condition for the traffic state $\tau$ ensures that the induced trace $\varphi_\tau$ is positive \cite{CDM16}. Consequently, when referring to a traffic space, we implicitly assume the additional tracial $*$-ps structure.

We also work with a transform of the traffic state called the \emph{injective traffic state}:
\begin{equation}\label{eq:injective_mobius}
\tau^0: \C\mcal{T}\langle\mcal{A}\rangle \to \C, \qquad T \mapsto \sum_{\pi \in \mcal{P}(V)} \tau[T^\pi] \mu(0_V, \pi),
\end{equation}
where $(\mcal{P}(V), \leq)$ is the poset of partitions of $V$ with the reversed refinement order, $\mu$ is the corresponding M\"{o}bius function, and $T^\pi$ is the test graph obtained from $T$ by identifying the vertices within each block $B \in \pi$. One recovers the traffic state via the inversion
\begin{equation}\label{eq:injective_sum}
\tau[T] = \sum_{\pi \in \mcal{P}(V)} \tau^0[T^\pi].
\end{equation}
For example,
\[
  \varphi_\tau(ab) =  \varphi_\tau\Big(  \ \foutput \overset{a}{\leftarrow} \cdot \overset{b}{\leftarrow} \finput \ \Big) = \tau\Big[ \ \cdot \ \overset{a}{\underset{b}{\srightleftarrows}} \ \cdot \ \Big] = \tau^0\Big[ \ \cdot \ \overset{a}{\underset{b}{\srightleftarrows}} \ \cdot \ \Big] +  \tau^0\bigg[ \ \overset{a}{\underset{b}{\ffgeight}} \ \bigg].
\]
The injective traffic state is unital \ref{unity} and edge-multilinear \ref{multilinearity}, but in general it fails the substitution axiom \ref{substitution} (but do see \cite[Lemma 4.17]{Mal11}). Of course, the relationships \eqref{eq:injective_mobius} and \eqref{eq:injective_sum} imply that convergence in traffic distribution is equivalent to convergence in injective traffic distribution, where the latter notion is defined in the obvious way.

\begin{eg}[Graph of matrices, revisited]\label{eg:random_matrices}
The $\mcal{G}^*$-algebra $\mcal{A} = L^{\infty-}(\Omega, \mcal{F}, \pr) \otimes \matN(\C)$ admits a traffic state
\[
  \tau[T] = \E\bigg[\frac{1}{N}\Tr[T]\bigg] = \E\bigg[\frac{1}{N}\sum_{\phi: V \to [N]} \prod_{e \in E} \gamma(e)(\phi(e))\bigg], \qquad \forall T \in \mcal{T}\langle\mcal{A}\rangle
\]
that recovers the trace
\[
\varphi_\tau(\mbf{A}_N) = \E\bigg[\frac{1}{N}\Tr(\mbf{A}_N)\bigg], \qquad \forall \mbf{A}_N \in \mcal{A}.
\]
The injective traffic state $\tau^0$ admits an explicit formula without reference to the M\"{o}bius function in the matricial setting, namely,
\[
\tau^0[T] = \E\bigg[\frac{1}{N}\sum_{\phi: V \hookrightarrow [N]} \prod_{e \in E} \gamma(e)(\phi(e))\bigg],
\]
where the notation $\phi: V \hookrightarrow [N]$ indicates an injective function (whence the name \emph{injective} traffic state).
\end{eg}

So far, we have seen that every algebraic traffic space defines a tracial ncps. In the opposite direction, one can ask if every tracial ncps can be given the structure of an algebraic traffic space and, if so, whether or not there is a canonical choice amongst the possibilities. We review the construction in \cite{CDM16} answering this question in the affirmative.

\begin{defn}[Free $\mcal{G}$-algebra]\label{defn:free_G_algebra}
Let $(\mcal{A}, \varphi)$ be a tracial ncps. We write $\C\mcal{G}\langle\mcal{A}\rangle$ for the graph polynomials in $\mcal{A}$ (recall Example \ref{eg:*-graph_polynomial}). Let $\mcal{I} \subset \C\mcal{G}\langle\mcal{A}\rangle$ be the two-sided ideal spanned by elements of the form
\[
Z_g\Big(\ \foutput \xleftarrow{a_1} \finput , \ldots, \foutput \xleftarrow{a_K} \finput \ \Big)
-
Z_g\Big(P\big(\ \foutput \xleftarrow{b_1} \finput \ , \ldots,\ \foutput \xleftarrow{b_n} \finput \ \big) , \ldots, \foutput \xleftarrow{a_K} \finput \ \Big)
\]
for any graph operation $g$ and NC polynomial $P$ such that the random variables $a_1, \ldots, a_K, b_1, \ldots, b_n \in \mcal{A}$ satisfy $a_1 = P(b_1, \ldots, b_n)$. The \emph{free $\mcal{G}$-algebra generated by $\mcal{A}$} is the quotient $\mcal{G}(\mcal{A}) := \C\mcal{G}\langle\mcal{A}\rangle/\mcal{I}$. If $\mcal{A}$ has the additional structure of a $*$-algebra, then $\mcal{G}(\mcal{A})$ has the additional structure of a $\mcal{G}^*$-algebra.
\end{defn}

The free $\mcal{G}$-algebra generated by $\mcal{A}$ satisfies the obvious universal property. In particular, the quotient structure allows us to define a monomorphism of unital algebras
\[
  \digamma^{\mcal{G}}: \mcal{A} \hookrightarrow \mcal{G}(\mcal{A}), \qquad a \mapsto \foutput \xleftarrow{a} \finput\, ,
\]
and so we think of
\[
\mcal{A} = \Big(\ \foutput \overset{a}{\leftarrow} \finput \ \Big| \  a \in \mcal{A}\Big) \subset \mcal{G}(\mcal{A}).
\]
We emphasize that the notation $\foutput \overset{a}{\leftarrow} \finput$ now stands for an equivalence class.

To define the traffic state on $\mcal{G}(\mcal{A})$, we will need the notion of a cactus graph.

\begin{defn}[Cactus graph]\label{defn:cactus_graph}
A connected multidigraph $G = (V, E)$ is a \emph{cactus} if every edge $e \in E$ belongs to a unique simple cycle. If each such cycle is also directed, then we say that $G$ is an \emph{oriented cactus}. We refer to the cycles of a cactus $G$ as \emph{pads} and denote the set of such cycles by $\op{Pads}(G)$. See Figure \ref{fig1:cactus} for an illustration.
\end{defn}

\begin{figure}
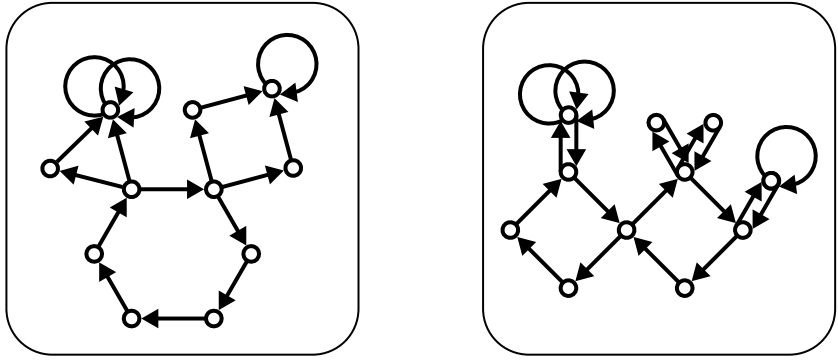
\caption{An example of a cactus and an oriented cactus respectively.}\label{fig1:cactus}
\end{figure}

\begin{defn}[Universal enveloping traffic space]\label{defn:ue_traffic_space}
Let $(\mcal{A}, \varphi)$ be a tracial ncps. The \emph{universal enveloping (UE) traffic space of $(\mcal{A}, \varphi)$} is the algebraic traffic space $(\mcal{G}(\mcal{A}), \tau_\varphi)$ whose injective traffic state $\tau_\varphi^0: \C\mcal{T}\langle\mcal{A}\rangle \to \C$ is defined as follows:
\begin{enumerate}[label=(\roman*)]
\item \label{cycle} If $T \in \mcal{T}\langle\mcal{A}\rangle$ is a directed cycle of length $n$ in the clockwise orientation with edges labeled by $a_1, \ldots, a_n$ \emph{counterclockwise}, then
\[
\tau_\varphi^0[T] = \kappa_n[a_1, \ldots, a_n],
\]
where $\kappa_n$ is the $n$th free cumulant of $(\mcal{A}, \varphi)$. For example,
\[
  \begin{tikzpicture}[shorten > = 1.5pt]
    \node at (-1.5, 0) {$\tau_\varphi^0\Bigg[$};
    \node at (3.325, 0) {$\Bigg] = \kappa_6[a_1, a_2, a_3, a_4, a_5, a_6]$;};
    \draw[fill=black] (.6, 0) circle (1pt);
    \draw[fill=black] (-.6, 0) circle (1pt);
    \draw[fill=black] (.3, .5196) circle (1pt);
    \draw[fill=black] (.3, -.5196) circle (1pt);
    \draw[fill=black] (-.3, .5196) circle (1pt);
    \draw[fill=black] (-.3, -.5196) circle (1pt);
    \draw[semithick, ->] (.6,0) to node[pos=.625, right] {${\scriptstyle a_5}$} (.3,-.5196);
    \draw[semithick, ->] (.3,-.5196) to node[midway, below] {${\scriptstyle a_4}$} (-.3,-.5196);
    \draw[semithick, ->] (-.3,-.5196) to node[pos=.375, left] {${\scriptstyle a_3}$} (-.6, 0);
    \draw[semithick, ->] (-.6, 0) to node[pos=.625, left] {${\scriptstyle a_2}$} (-.3, .5196);
    \draw[semithick, ->] (-.3, .5196) to node[midway, above] {${\scriptstyle a_1}$} (.3, .5196);
    \draw[semithick, ->] (.3, .5196) to node[pos=.375, right] {${\scriptstyle a_6}$} (.6, 0);
  \end{tikzpicture}
\]
\item \label{oriented_cactus} If $T \in \mcal{T}\langle\mcal{A}\rangle$ is an oriented cactus, then
\[
\tau_\varphi^0[T] = \prod_{C \in \op{Pads}(T)} \tau_\varphi^0[C].
\]
For example,
\[
  \tau_\varphi^0\Big[ \ {\scriptstyle a_1} \ \ssidecircle \ \ \overset{a_2}{\underset{a_3}{\srightleftarrows}} \ \ \srsidecircle \ {\scriptstyle a_4} \ \Big] = \kappa_1[a_1]\kappa_2[a_2, a_3]\kappa_1[a_4];
\]
\item \label{non_cacti} Otherwise,
\[
\tau_\varphi^0[T] = 0.
\]
\end{enumerate}
\end{defn}

Since $\varphi$ is a trace, the free cumulants $(\kappa_n)_{n \in \N}$ are cyclically invariant, i.e.,
\[
\kappa_n[a_1, \ldots, a_n] = \kappa_n[a_2, \ldots, a_n, a_1] = \cdots = \kappa_n[a_n, a_1, \ldots, a_{n-1}].
\]
This implies that the definition in \ref{cycle} does not depend on where we choose to start reading off the elements $a_i$ in the cycle as long as we proceed in a counterclockwise fashion. Strictly speaking, our injective traffic state should be defined on $\C\mcal{T}\langle\mcal{G}(\mcal{A})\rangle$. One can show that the definition of $\tau_\varphi^0$ respects the quotient structure of $\mcal{G}(\mcal{A}) = \C\mcal{G}\langle\mcal{A}\rangle/\mcal{I}$, ensuring that the construction is well-defined \cite[Proposition 4.4]{CDM16}.

Equation \eqref{eq:injective_sum} defines the corresponding traffic state $\tau_\varphi: \C\mcal{T}\langle\mcal{A}\rangle \to \C$. As before, this gives rise to a linear functional $\varphi_{\tau_\varphi}: \mcal{G}(\mcal{A}) \to \C$ through the now familiar route of identifying the roots and forgetting their distinguished roles. In particular,
\[
  \varphi_{\tau_\varphi} \circ \digamma^{\mcal{G}}(a) = \varphi_{\tau_\varphi}\Big(\ \foutput \overset{a}{\leftarrow} \finput \ \Big) = \tau_{\varphi}\Big[\ \, \raisebox{-3pt}{$\overset{a}{\dcirclearrowleft}$} \, \ \Big] = \tau_{\varphi}^0\Big[\ \, \raisebox{-3pt}{$\overset{a}{\dcirclearrowleft}$} \, \ \Big] = \kappa_1[a] = \varphi(a).
\]
In other words, $\varphi_{\tau_\varphi}|_{\mcal{A}} = \varphi$ extends the original trace. For notational convenience, we simply write $\psi = \varphi_{\tau_\varphi}$ for the induced trace on $\mcal{G}(\mcal{A})$.

Recall that if $\mcal{A}$ is a $*$-algebra, then $\mcal{G}(\mcal{A})$ has the additional structure of a $\mcal{G}^*$-algebra. The positivity of the traffic state in the UE traffic space will not play a significant role in this article, and so we do not discuss it much further. Instead, our interest in this construction comes from the following results in \cite{CDM16}.

\begin{prop}[{\cite[Theorem 1.1]{CDM16}}]
Let $\mcal{M}_N = (\mbf{M}_N^{(i)})_{i \in I} \subset L^{\infty-}(\Omega, \mcal{F}, \pr) \otimes \matN(\C)$ be a family of random matrices satisfying the following properties:
\begin{enumerate}[label=(\roman*)]
\item \label{unitary_inv} (Unitary invariance) For any unitary matrix $\mbf{U}_N \in \mcal{U}(N)$,
\[
\mbf{U}_N\mcal{M}_N\mbf{U}_N^* = (\mbf{U}_N\mbf{M}_N^{(i)}\mbf{U}_N^*)_{i \in I} \deq (\mbf{M}_N^{(i)})_{i \in I} = \mcal{M}_N;
\]
\item \label{convergence_sd} (Convergence in $*$-distribution) For any $*$-polynomial $P \in \C\langle\mbf{x}, \mbf{x}^*\rangle$, the limit $\lim_{N \to \infty} \E\big[\frac{1}{N}\Tr(P(\mcal{M}_N))\big]$ exists;
\item \label{ue_factor} (Factorization) For any $*$-polynomials $P_1, \ldots, P_\ell \in \C\langle\mbf{x}, \mbf{x}^*\rangle$,
\[
\lim_{N \to \infty} \E\bigg[\prod_{k = 1}^\ell \frac{1}{N}\Tr(P_k(\mcal{M}_N))\bigg] = \prod_{k = 1}^\ell \lim_{N \to \infty} \E\bigg[\frac{1}{N}\Tr(P_k(\mcal{M}_N))\bigg],
\]
where the product on the right-hand side exists by \ref{convergence_sd}.
\end{enumerate}
In particular, \ref{convergence_sd} implies that $\mcal{M}_N$ converges in $*$-distribution to a family of random variables $\mbf{a} = (a_i)_{i \in I}$ in a tracial $*$-ps $(\mcal{A}, \varphi)$. Under these assumptions, $\mcal{M}_N$ further converges in $*$-traffic distribution to $\mbf{a}$ in the UE traffic space $(\mcal{G}(\mcal{A}), \tau_\varphi)$.
\end{prop}

Based on the matrix heuristic, the transposed algebra
\[
  \mcal{A}^\intercal = \Big(a^\intercal = \foutput \overset{a}{\rightarrow} \finput \ \Big| \ a \in \mcal{A} \Big)
\]
emerges as a natural subalgebra of $\mcal{G}(\mcal{A})$. Similarly, one can also consider the degree algebra
\[
\op{Deg}(\mcal{A}) = \C\langle\op{rDeg}(\mcal{A}), \op{cDeg}(\mcal{A})\rangle \subset \Delta(\mcal{G}(\mcal{A})).  
\]
The cactus structure of the injective traffic state $\tau_\varphi^0$ imposes a rigid probabilistic structure on $\mcal{G}(\mcal{A})$. In particular, the following result both explains and extends a result of Mingo and Popa on freeness from the transpose for unitarily invariant random matrices \cite{MP16}.

\begin{prop}[{\cite[Corollary 4.7]{CDM16}}]\label{prop:free_independence_ue_traffic_space}
The unital subalgebras $\mcal{A}, \mcal{A}^\intercal$, and $\op{Deg}(\mcal{A})$ are freely independent in the UE traffic space $(\mcal{G}(\mcal{A}), \psi)$.
\end{prop}

The cactus structure of $\tau_\varphi^0$ is further justified by its intertwining of free independence and traffic independence.

\begin{prop}[{\cite[Proposition 4.8]{CDM16}}]
For unital subalgebras $(\mcal{A}_i)_{i \in I}$ of $\mcal{A}$, the following two conditions are equivalent:
\begin{enumerate}[label=(\roman*)]
\item \label{free} The subalgebras $(\mcal{A}_i)_{i \in I}$ are freely independent in $(\mcal{A}, \varphi)$;
\item \label{traffic} The sub-traffic spaces $(\mcal{G}(\mcal{A}_i))_{i \in I}$ are traffic independent in $(\mcal{G}(\mcal{A}), \tau_\varphi)$.
\end{enumerate}
\end{prop}

The careful reader will notice that we have not defined traffic independence. In particular, we will study the structure of the UE traffic space in this article without appealing to this concept. 

\subsection{Statement of results}\label{sec:results}

Our first result extends the inherent free independence structure in Proposition \ref{prop:free_independence_ue_traffic_space} to general graph operations. For notational convenience, we write $(\mcal{B}, \psi) = (\mcal{G}(\mcal{A}), \varphi_{\tau_\varphi})$ to emphasize the underlying ncps structure of the UE traffic space. We distinguish two special classes of graph operations:
\[
  \Delta(\mcal{G}_K) = \{g \in \mcal{G}_K : \vin = \vout\}
\]
and
\begin{align*}
  \Theta(\mcal{G}_K) = \{g \in \mcal{G}_K : &\text{ there exists a simple (undirected) cycle} \\
  &\text{ that visits both $\vin$ and $\vout$}\}.
\end{align*}
In this notation, the diagonal algebra becomes
\[
  \Delta(\mcal{B}) = \op{span}\Big(\bigcup_{K \geq 0} \bigcup_{g \in \Delta(\mcal{G}_K)} Z_g(\mcal{A}^K)\Big).
\]
Similarly, we define the unital subalgebra
\[
  \Theta(\mcal{B}) = \op{span}\Big(\bigcup_{K \geq 0} \bigcup_{g \in \Theta(\mcal{G}_K)} Z_g(\mcal{A}^K)\Big).
\]
Note that
\[
  \op{rDeg}, \op{cDeg} \in \Delta(\mcal{G}_K) \subset \Theta(\mcal{G}_K),
\]
and so
\[
  \op{Deg}(\mcal{A}) \subset \Delta(\mcal{B}) \subset \Theta(\mcal{B}).
\]
\begin{thm}\label{thm:free_product}
Let $(\mcal{A}, \varphi)$ be a tracial $*$-ps. Then the unital $*$-subalgebras $\mcal{A}$, $\mcal{A}^\intercal$, and $\Delta(\mcal{B})$ are freely independent in $(\mcal{B}, \psi)$. Moreover, there exists a homomorphic conditional expectation $\mathscr{E}: \mcal{B} \to \mcal{A} * \mcal{A}^\intercal * \Delta(\mcal{B})$ such that $\mathscr{E}^{-1}(\Delta(\mcal{B})) = \Theta(\mcal{B})$. Altogether, this implies the free product decomposition
\begin{equation}\label{eq:free_product}
\mcal{B} = \mcal{A} *\mcal{A}^\intercal * \Theta(\mcal{B}).
\end{equation}
\end{thm}

Theorem \ref{thm:free_product} holds regardless of the choice of tracial $*$-ps $(\mcal{A}, \varphi)$. Naturally, one can ask how an existing free product structure $\mcal{A} = *_{i \in I} \mcal{A}_i$ behaves in this construction. We already know that the free independence of the subalgebras $(\mcal{A}_i)_{i \in I}$ is equivalent to the traffic independence of the sub-traffic spaces $(\mcal{G}(\mcal{A}_i))_{i \in I}$ \cite{CDM16}. From a different perspective, we can study these sub-traffic spaces as subalgebras $(\mcal{B}_i)_{i \in I} = (\mcal{G}(\mcal{A}_i))_{i \in I}$.
Of course, the free product decomposition \eqref{eq:free_product} still applies, and so we know the behavior of the cross-terms in the decomposition
\[
  (\mcal{B}_i)_{i \in I} = (\mcal{A}_i * \mcal{A}_i^\intercal * \Theta(\mcal{B}_i))_{i \in I}
\]
grace of the inclusions
\[
  \mcal{A}_i \subset \mcal{A}, \quad \mcal{A}_i^\intercal \subset \mcal{A}^\intercal, \quad \text{and} \quad \Theta(\mcal{B}_i) \subset \Theta(\mcal{B}).
\]
Thus, it remains to understand the relationship within each of the three ``columns''
\[
  (\mcal{A}_i)_{i \in I}, \quad (\mcal{A}_i^\intercal)_{i \in I}, \quad \text{and} \quad (\Theta(\mcal{B}_i))_{i \in I}.
\]
By assumption, the $(\mcal{A}_i)_{i \in I}$ are freely independent, which is preserved in the UE traffic space since $\psi|_{\mcal{A}} = \varphi$. One can show that the transpose $\intercal: \mcal{A} \to \mcal{A}^\intercal$ defines an involutive anti-isomorphism, transporting the free product structure $\mcal{A} = *_{i \in I} \mcal{A}_i$ to the transposed algebra $\mcal{A}^\intercal = *_{i \in I} \mcal{A}_i^\intercal$. Finally, we come to the last column $(\Theta(\mcal{B}_i))_{i \in I}$. Applying our conditional expectation $\mathscr{E}$, we can further reduce this problem to understanding the relationship between the diagonal algebras $(\Delta(\mcal{B}_i))_{i \in I}$. Since the diagonal algebra $\Delta(\mcal{B}_i) \subset \Delta(\mcal{B})$ is commutative, the trend of free independence cannot possibly continue. Instead, we find an interesting connection to the classical framework.
\begin{prop}\label{prop:commutification}
Let $(\mcal{A}, \varphi)$ be a tracial $*$-ps with freely independent unital $*$-subalgebras $(\mcal{A}_i)_{i \in I}$. Then the commutative subalgebras $(\Delta(\mcal{B}_i))_{i \in I}$ are classically independent in $(\mcal{B}, \psi)$.
\end{prop}

We formulate this result in the heuristic equation
\begin{equation}\label{eq:commutification}
\text{classical independence} \quad \overset{\Delta}{\impliedby} \quad \text{free independence}
\end{equation}
and compare it to our earlier equation \eqref{eq:non_commutification}. Before, starting with classical independence, we obtain free independence through a natural process of ``non-commutification,'' namely, passing to a matrix algebra and taking a limit. In the opposite direction, Proposition \ref{prop:commutification} starts with free independence; however, the route back to the commutative world becomes unclear. We could hope to make use of the traffic framework, where the diagonal algebra emerges as a natural ``commutification'' of our space. In particular, by pushing $(\mcal{A}, \varphi)$ up to the UE traffic space $(\mcal{G}(\mcal{A}), \tau_\varphi)$, we can project down to the diagonal algebra $\Delta(\mcal{G}(\mcal{A}))$. As it turns out, the shadow cast by free independence in this projection is precisely classical independence.

Note that the $(\Theta(\mcal{B}_i))_{i \in I}$ are \emph{not} classically independent as they do not commute; however, the conditional expectation $\mathscr{E}$ allows us to compute the trace on the algebra $\op{alg}(\bigcup_{i \in I} \Theta(\mcal{B}_i))$ generated by $\bigcup_{i \in I} \Theta(\mcal{B}_i)$ as if they did. In other words, for any $t_1, \ldots, t_n \in \op{alg}(\bigcup_{i \in I} \Theta(\mcal{B}_i)) \subset \Theta(\mcal{B})$,
\begin{align*}
  \psi(t_1 \cdots t_n) &= \psi(\mathscr{E}(t_1 \cdots t_n)) \\
                       &= \psi(\mathscr{E}(t_1) \cdots \mathscr{E}(t_n)) \\
                       &= \psi(\mathscr{E}(t_{\pi(1)}) \cdots \mathscr{E}(t_{\pi(n)})) \\
                       &= \psi(\mathscr{E}(t_{\pi(1)} \cdots t_{\pi(n)})) = \psi(t_{\pi(1)} \cdots t_{\pi(n)}), \qquad \forall \pi \in \mfk{S}_n.
\end{align*}
In fact, the conditional expectation $\mathscr{E}$ satisfies the stronger property
\[
  \mathscr{E}(t) \equiv t \text{ (mod $\psi$)}, \qquad \forall t \in \mcal{B}.
\]
Thus, from the distributional point of view, the reduction from $\Theta(\mcal{B})$ to $\Delta(\mcal{B})$ comes without loss of generality. This redundancy in the action of the graph operations extends a great deal further. To make this precise, we introduce a subclass of the diagonal graph operations
\[
  \Delta_{\text{tree}}(\mcal{G}_K) = \{g \in \Delta(\mcal{G}_K) : \text{ $g$ is a tree}\}
\]
and the corresponding subalgebra
\[
  \Delta_{\text{tree}}(\mcal{B}) = \op{span}\Big(\bigcup_{K \geq 0} \bigcup_{g \in \Delta_{\text{tree}}(\mcal{G}_K)} Z_g(\mcal{A}^K)\Big).
\]
\begin{thm}\label{thm:tree_reduction}
For any $t \in \Delta(\mcal{B})$, there exists a $\mbf{T}(t) \in \Delta_{\text{\emph{tree}}}(\mcal{B})$ such that
\[
  \mbf{T}(t) \equiv t \text{ \emph{(mod}  $\psi\emph{)}$}.
\]
\end{thm}

We can apply this result to each of the diagonal components in the free product decomposition $\mcal{B} \equiv \mcal{A} * \mcal{A}^\intercal * \Delta(\mcal{B}) \text{ (mod $\psi$)}$ to show that a similar statement holds for general traffics $t \in \mcal{B}$. Accordingly, we think of the UE traffic space as being spanned by tree-like graph operations.

Finally, we complete the characterization of the distribution of the UE traffic space $(\mcal{B}, \psi)$ by showing that the diagonal algebra $\Delta(\mcal{B})$ is generated by complex Gaussian random variables with a covariance structure diagonalized by the graph operations (see \S\ref{sec:gaussianity} for a precise statement).

\begin{thm}\label{thm:gaussianity}
Up to degeneracy, the diagonal algebra $\Delta(\mcal{B})$ is generated by complex Gaussian random variables $Q(t)$ for a special class of tree-like diagonal graph monomials $t \in \Delta(\mcal{B})$. Here, $Q$ is a transform of $t$ that outputs a linear combination of quotients of $t$. In particular, $\psi(Q(t)Q(t')) = 0$ unless the underlying rooted digraphs of $t$ and $t'$ are anti-isomorphic.
\end{thm}

Our analysis of the cactus structure of $\tau_\varphi^0$ reveals a more general cactus-cumulant correspondence, allowing us to extend our results to a larger class of traffic distributions. We give a precise statement of this correspondence in Section \ref{sec:applications} with applications to random multi-matrix models beyond the unitarily invariant case. For example, we state a consequence for \emph{dependent} random matrices generalizing results in \cite{BDJ06,MP16}.

\begin{cor}\label{cor:wigner_family}
Let $\mbf{W}_N$ be a Wigner matrix with pseudo-variance $N \E[\mbf{W}_N(j, k)^2] = \beta \in \R$, and define $\mbf{D}_N = \op{rDeg}(\mbf{W}_N)$. Then $(\mbf{W}_N, \mbf{W}_N^\intercal, \mbf{D}_N)$ converges in $*$-distribution to a triple $(s_1, s_2, d)$, where $(s_1, s_2)$ and $d$ are $*$-free, $(s_1, s_2)$ is a semicircular family with covariance $\begin{pmatrix}
    1 & \beta \\
    \beta & 1
  \end{pmatrix}$,
and $d$ is a complex Gaussian random variable with variance $1$ and pseudo-variance $\beta$.
\end{cor}

\section{Rigid structures in the universal enveloping traffic space}\label{sec:rigid_structures}

\subsection{The combinatorics of cactus graphs}\label{sec:cacti}
We start with a review of some basic notions from graph theory (see, for example, \cite{Bol98,GR01}).

\begin{defn}[Connectivity]\label{defn:connectivity}
Let $G = (V, E)$ be a multigraph. An \emph{edge cutset} in $G$ is a subset of edges $E' \subset E$ whose deletion increases the number of connected components. In particular, a \emph{cut-edge} is an edge $e \in E$ such that the singleton $\{e\}$ is an edge cutset. Similarly, a \emph{vertex cutset} in $G$ is a subset of vertices $V' \subset V$ whose deletion (along with all edges adjacent to $V'$) increases the number of connected components. A \emph{cut-vertex} is defined in the obvious way. A \emph{block} of $G$ is a maximal cut-vertex-free connected subgraph $H \subset G$. Note that any two distinct blocks of $G$ have at most one vertex in common (necessarily a cut-vertex of $G$). Conversely, every cut-vertex of $G$ belongs to at least two distinct blocks.

Now suppose that $G$ is connected. The \emph{edge connectivity} of $G$ is the size of the smallest edge cutset in $G$, which we denote by $\lambda(G) \geq 1$. We say that $G$ is \emph{$k$-edge-connected} if $\lambda(G) \geq k$. In other words, deleting any $\ell < k$ edges of $G$ does not affect its connectivity. In particular, we say that $G$ is \emph{two-edge-connected} (or \emph{t.e.c.} for short) if it has no cut-edges $\lambda(G) \geq 2$.

We can define a similar notion for vertices even if $G$ is not connected. The \emph{edge connectivity} of two distinct vertices $v, w \in V$ is the size of the smallest subset of edges whose deletion disconnects $v$ and $w$, which we denote by $\lambda(v, w) \geq 0$. We say that $v$ and $w$ are \emph{$k$-edge-connected} if $\lambda(v, w) \geq k$. 
\end{defn}

We recall Menger's theorem for edge connectivity \cite{Men27}.

\begin{thm}\label{thm:menger}
Let $v$ and $w$ be distinct vertices of a multigraph $G$. Then $\lambda(v, w)$ is equal to the maximum number of edge-disjoint paths from $v$ to $w$.
\end{thm}

In view of Menger's theorem, we also say that $k$-edge-connected vertices $v$ and $w$ form a \emph{$k$-connection}, particularly when we want to emphasize the number of edge-disjoint paths connecting $v$ and $w$. We apply this to obtain a simple characterization of cactus graphs.

\begin{prop}\label{prop:cactus_characterization}
A multigraph $G = (V, E)$ is a cactus iff
\[
\lambda(v, w) = 2, \qquad \forall v \neq w \in V.
\]
\end{prop}

\begin{proof}
First, suppose that $G$ is a cactus. Note that a cactus can be reconstructed from its pads by ``growing'' the cactus as follows: start at level 0 by choosing an arbitrary pad $C^{(0)} \in \op{Pad}(G)$ to be the base. At level 1, attach all of the remaining pads $C_1^{(1)}, \ldots, C_{\ell_1}^{(1)}$ that share a vertex with $C^{(0)}$. Of course, each pad can only share at most one vertex with another pad, so we imagine each $C_i^{(1)}$ as growing from a vertex in $C^{(0)}$. Furthermore, note that if any two pads at this level share a vertex, then it has to be the same vertex that they each share with $C^{(0)}$. Otherwise, the cactus has grown in on itself and one can easily find an edge that belongs to more than one simple cycle. At level $n$, we attach all of the remaining pads that share a vertex with a pad at level $n-1$. In particular, we can think of each pad at level $n-1$ as a new base and growing the remaining pads on each base. As before, we note that a pad at level $n$ can only be attached to a single pad at level $n-1$; otherwise, one can again find an edge that belongs to more than one simple cycle.

Now, since every edge belongs to a unique simple cycle, $G$ is necessarily t.e.c. This implies that
\[
\lambda(v, w) \geq \lambda(G) \geq 2, \qquad \forall v \neq w \in V.
\]
If $v$ and $w$ belong to a common pad, then we can assume that this pad is the base $C^{(0)}$. In this case, deleting the two edges adjacent to $v$ in $C^{(0)}$ clearly disconnects $v$ and $w$ in $G$. If $v$ and $w$ do not belong to a common pad, then we can again take a pad containing $v$ to be the base $C^{(0)}$. In this case, $w$ belongs to a pad that was successively grown from an ancestor $C_1^{(1)}$ on level 1, say attached to a vertex $u \in C^{(0)}$. Deleting the two edges adjacent to $u$ in $C_1^{(1)}$ then clearly disconnects $v$ and $w$ in $G$ (see Figure \ref{fig2.1:disconnect}). It follows that
\[
\lambda(v, w) = \lambda(G) = 2, \qquad \forall v \neq w \in V.
\]

\begin{figure}
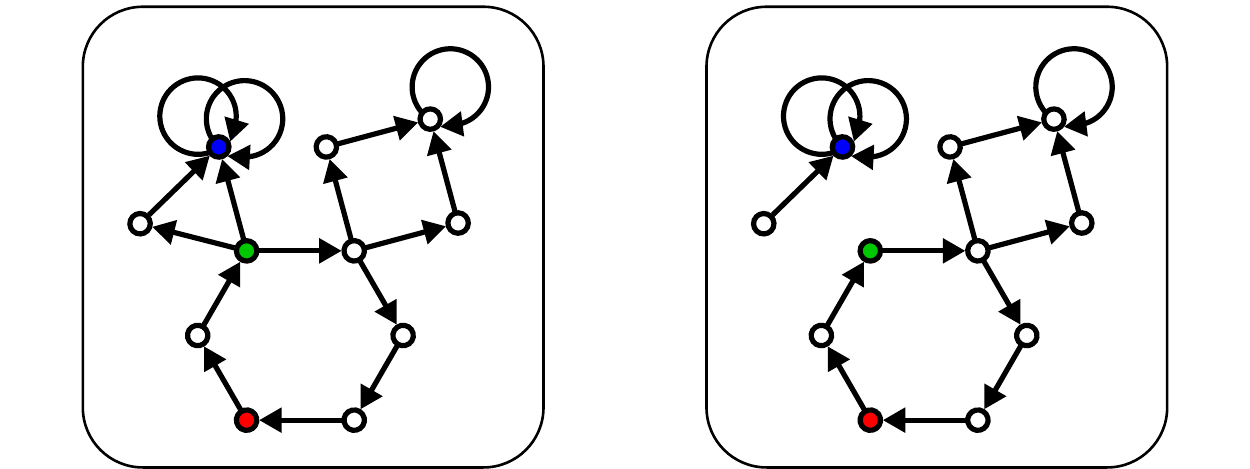
\caption{An example of the edge connectivity of two distinct vertices in the cactus from Figure \ref{fig1:cactus}. Here, we disconnect the red vertex $v \in C^{(0)}$ and the blue vertex $w \in C_1^{(1)}$. The green vertex denotes the ancestor vertex $u \in C^{(0)}$.}\label{fig2.1:disconnect}
\end{figure}

In the opposite direction, assume that the edge connectivity of every pair of vertices in $G$ is equal to two. This implies that $G$ is connected with $\lambda(G) = 2$, and so every edge belongs to a simple cycle. For a contradiction, assume that there exists an edge $e \in E$ that belongs to two distinct simple cycles $C_1$ and $C_2$. Then there must be at least one edge $e_1 \in C_1 \setminus C_2$ (and one edge $e_2 \in C_2 \setminus C_1$). Let $E'$ be a separating pair of edges for the distinct vertices $v$ and $w$ adjacent to $e$ (a loop belongs to a unique simple cycle). Of course, it must be that $e \in E'$. In fact, the second edge in $E'$ must be another edge in $C_1 \cap C_2$ since otherwise $v$ and $w$ are not separated. Thus, deleting $e$ and $e_1$ from $G$ does not separate $v$ and $w$. Let $P$ be a simple path from $v$ to $w$ in the $(e, e_1)$-deleted graph $\wtilde{G}$. By construction, $P$ cannot stay in $C_1$. Let $e_P$ be the first edge along the path $P$ that leaves $C_1$, and define $v_P$ to be the vertex in $C_1$ adjacent to $e_P$. Similarly, let $f_P$ be the first edge along this path that returns to $C_1$, and define $w_P$ to be the vertex in $C_1$ adjacent to $f_P$. By construction, $v_P \neq w_P$ form a 3-connection in $G$: two edge-disjoint paths come from inside the cycle $C_1$, while a third comes from the truncation of $P$ outside of $C_1$ (see Figure \ref{fig2.2:menger}). Menger's theorem then implies that $\lambda(v_P, w_P) \geq 3$, a contradiction.

\begin{figure}
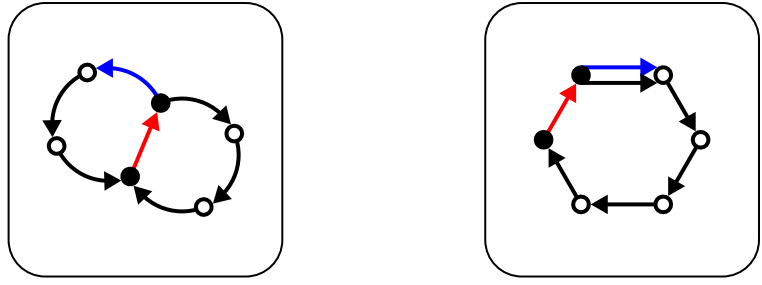
\caption{Two representative cases of the contradiction argument in the proof. Here, we have colored the edge $e$ belonging to two distinct simple cycles red, its adjacent vertices $v$ and $w$ black, and the edge $e_1 \in C_1 \setminus C_2$ blue. The reader should identify the vertices $v_P$ and $w_P$ as well as the associated $3$-connection.}\label{fig2.2:menger}
\end{figure}
\end{proof}

We use Proposition \ref{prop:cactus_characterization} to formalize \cite[Remark 4.5]{CDM16}. Consider a cycle graph $C = (V, E)$ of length $n$ with vertices $V = \{v_1, \ldots, v_n\}$ in counterclockwise order and edges $E = (e_1, \ldots, e_n)$ connecting $v_i \overset{e_i}{\sim} v_{i+1}$. The mapping $[n] \ni i \mapsto v_i \in V$ induces a one-to-one correspondence
\[
f: \mcal{NC}(n) \to \mcal{NC}(V)
\] 
between non-crossing partitions of $[n]$ and non-crossing partitions of $V$, where the latter notion comes from drawing $C = (V, E)$ as a circle. Similarly, the mapping $[\overline{n}] \ni \overline{i} \mapsto e_i \in E$ induces a one-to-one correspondence 
\[
g: \mcal{NC}(\overline{n}) \to \mcal{NC}(E)
\]
between non-crossing partitions of $[\overline{n}] = \{\overline{1} < \cdots < \overline{n}\}$ and non-crossing partitions of $E$. Furthermore, the Kreweras complement $\mcal{K}: \mcal{NC}(n) \to \mcal{NC}(\overline{n})$, defined as a non-crossing partition of $[\overline{n}] = \{\overline{1}, \ldots, \overline{n}\}$ via the interlacing
\[
1 < \overline{1} < \cdots < n < \overline{n},
\]
corresponds to the Kreweras complement $\mcal{K}: \mcal{NC}(V) \to \mcal{NC}(E)$ defined as a non-crossing partition of the edges $E$ of $C$ via the interlacing
\[
v_1 \overset{e_1}{\sim} v_2 \overset{e_2}{\sim} \cdots \overset{e_{n-1}}{\sim} v_n \overset{e_n}{\sim} v_1.
\]
See Figure \ref{fig2.3:non_crossing_correspondence} for an illustration. Formally, this amounts to the equality
\[
g \circ \mcal{K} = \mcal{K} \circ f.
\]

\begin{figure}
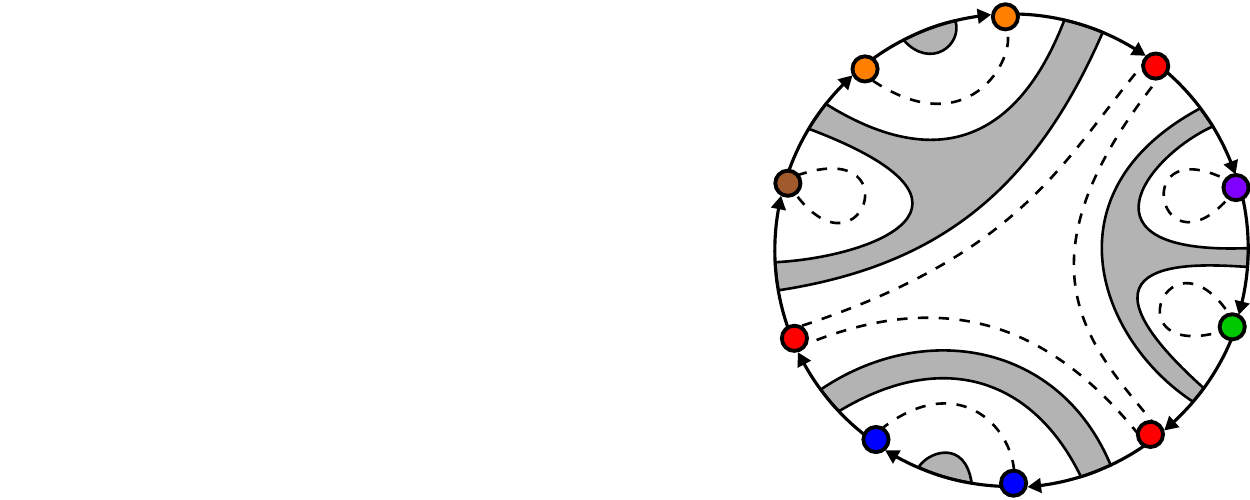
\caption{An example of the correspondence between non-crossing partitions on the line and non-crossing partitions on the circle.}\label{fig2.3:non_crossing_correspondence}
\end{figure}

\begin{figure}
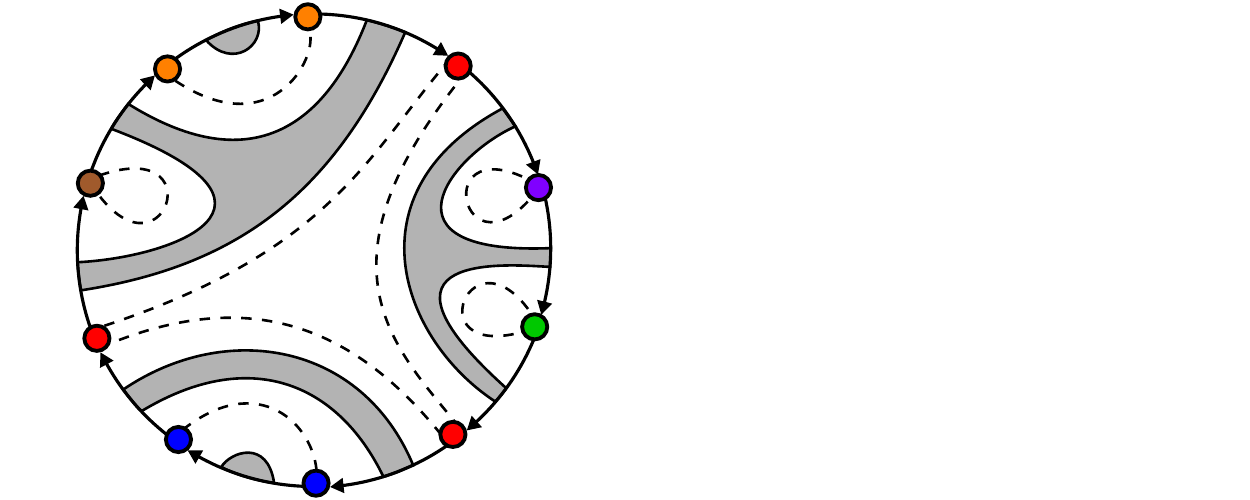
\caption{An example of Proposition \ref{prop:cactus_non_crossing_partition} in action.}\label{fig2.6:complement}
\end{figure}

\begin{prop}\label{prop:cactus_non_crossing_partition}
Let $C = (V, E)$ be a cycle graph as above. For a partition $\pi \in \mcal{P}(V)$, the following two conditions are equivalent:
\begin{enumerate}[label=(\roman*)]
\item \label{equivalence_cactus} The quotient graph $C^\pi$ is a cactus;
\item \label{equivalence_non_crossing} The partition $\pi$ is non-crossing.
\end{enumerate}
Furthermore, if $\pi \in \mcal{NC}(V)$, then the pads of the cactus $C^\pi$ correspond to the blocks of the Kreweras complement $\mcal{K}(\pi) \in \mcal{NC}(E)$ via the map
\begin{align*}
  \mcal{K}(\pi) \ni B &= (e_{i_1}, \ldots, e_{i_k}) \\
  &\mapsto (v_{i_1}, v_{i_1 +1} \overset{\pi}{\sim} v_{i_2}, \ldots, v_{i_{k-1} + 1} \overset{\pi}{\sim} v_{i_k}, v_{i_k + 1} \overset{\pi}{\sim} v_{i_1}) \in \op{Pads}(C^\pi), 
\end{align*}
where $(i_1 < \cdots < i_k)$.

Suppose, in addition, that $C = (V, E, \source, \target)$ is a directed graph (though not necessarily a directed cycle). We say that an edge $e_i$ is oriented counterclockwise if $\source(e_i) = v_i$; otherwise, $\source(e_i) = v_{i+1}$ and we say that $e_i$ is oriented clockwise. For a partition $\pi \in \mcal{P}(V)$, the following two conditions are equivalent:
\begin{enumerate}[label=(\Roman*)]
\item \label{equivalence_oriented_cactus} The quotient graph $C^\pi$ is an oriented cactus;
\item \label{equivalence_oriented_non_crossing} The partition $\pi$ is non-crossing and each block of $\mcal{K}(\pi) \in \mcal{NC}(E)$ only contains edges of a uniform orientation.
\end{enumerate}  
\end{prop}
\begin{proof}
We will only prove the equivalence of \ref{equivalence_cactus} and \ref{equivalence_non_crossing} as the rest of the proposition follows almost immediately. First, suppose that $C^\pi$ is a cactus. For a contradiction, assume that $\pi \not\in \mcal{NC}(V)$. After a suitable rotation of the cycle, this implies that there exist $i_1 < i_2 < i_3 < i_4 \in [n]$ such that $v_{i_1} \overset{\pi}{\sim} v_{i_3}$ and $v_{i_2} \overset{\pi}{\sim} v_{i_4}$ belong to different blocks of $\pi$. In that case, $v_{i_2}$ and $v_{i_3}$ form a 3-connection in $C^\pi$, which is absurd (see Figure \ref{fig2.4:butterfly}).

In the opposite direction, suppose that $\pi \in \mcal{NC}(V)$. Naturally, we can identify the blocks $B_1, \ldots, B_{\#(\pi)} \in \pi$ with the vertices of $C^\pi$. Since the original graph $C$ is a cactus, we know that
\[
\lambda_C(v_{i_1}, v_{i_2}) = 2, \qquad \forall i_1 \neq i_2 \in [n].
\]
At the same time, edge-disjoint paths in $C$ induce edge-disjoint paths in the quotient $C^\pi$, which implies that
\[
\lambda_{C^\pi}(B_{j_1}, B_{j_2}) \geq 2, \qquad \forall j_1 \neq j_2 \in [\#(\pi)].
\]
Now, because $\pi$ is non-crossing, two distinct blocks $B_{j_1}, B_{j_2} \in \pi$ can only take one of two relative positions up to symmetry: after a suitable rotation of the cycle, either
\[
v_1 \in B_{j_1} \quad \text{and} \quad \max\{i \in [n] : v_i \in B_{j_1}\} < \min\{i \in [n] : v_i \in B_{j_2}\}
\]
or
\[
\exists v_{i_1}, v_{i_2} \in B_{j_2}: \forall v_i \in B_{j_1} \qquad i_1 < i < i_2.
\]
The first case corresponds to when $B_{j_1}$ and $B_{j_2}$ lie on two non-intersecting arcs of the cycle $C$ drawn as a circle, whereas the second case corresponds to when $B_{j_1}$ lies on an arc trapped between two vertices of $B_{j_2}$. In either case, deleting the edges
\[
e_{\min\{i \in [n]: v_i \in B_{j_1}\}-1} \quad \text{and} \quad e_{\max\{i \in [n]: v_i \in B_{j_1}\}}
\]
disconnects $B_{j_1}$ and $B_{j_2}$ in $C^\pi$. We think of the identification
\[
v_{\min\{i \in [n]: v_i \in B_{j_1}\}} \overset{\pi}{\sim} v_{\max\{i \in [n]: v_i \in B_{j_1}\}}
\]
as pinching off the arc supporting $B_{j_1}$. By removing the edges at the boundary of this arc, we have separated the vertices lying on this arc from the rest of the vertices, even in the quotient $C^\pi$ (any identification of vertices across the two arcs would be crossing, see Figure \ref{fig2.5:pinching}). It follows that
\[
\lambda_{C^\pi}(B_{j_1}, B_{j_2}) = 2, \qquad \forall j_1 \neq j_2 \in [\#(\pi)],
\]
and so $C^\pi$ is a cactus.
\end{proof}

\begin{figure}
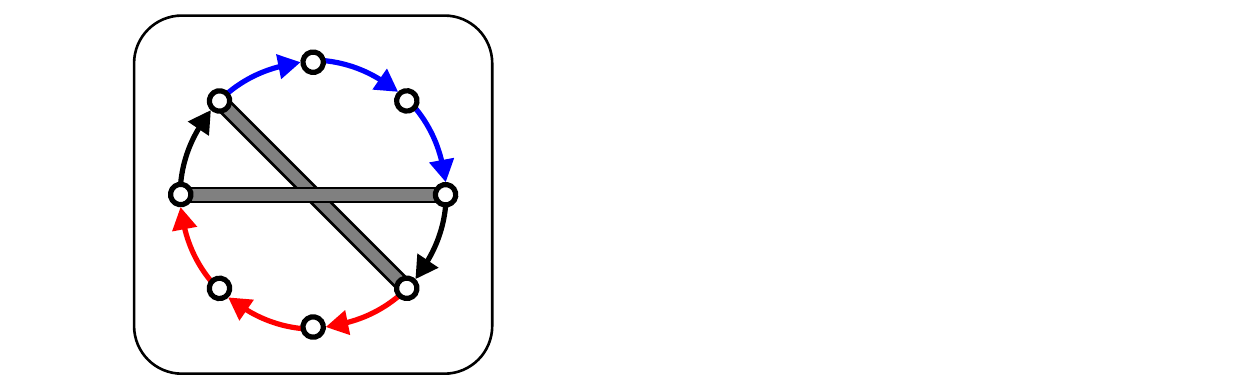
\caption{An example of a crossing leading to a 3-connection.}\label{fig2.4:butterfly}
\end{figure}

\begin{figure}
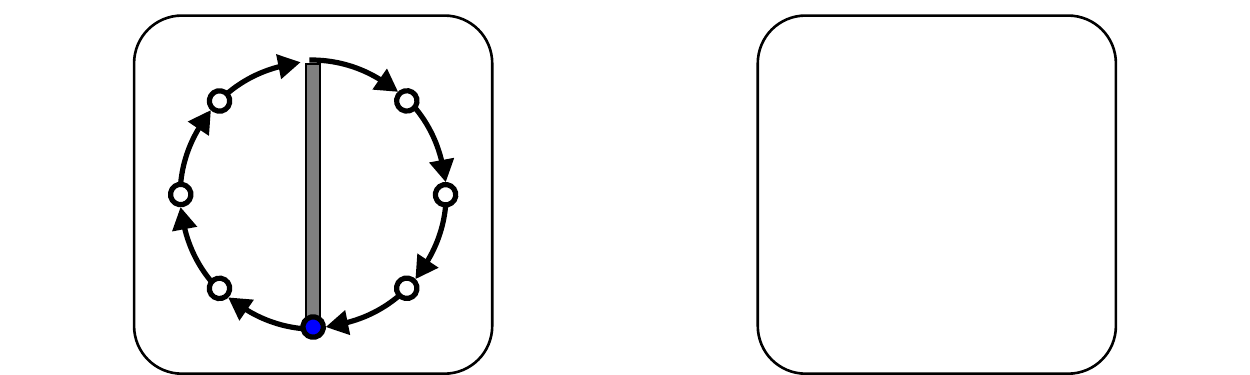
\caption{A depiction of the pinching argument and the corresponding edge removals that disconnect the arc (in dashed lines). We color the vertices $v_{\min\{i \in [n]: v_i \in B_{j_1}\}}$ and $v_{\max\{i \in [n]: v_i \in B_{j_1}\}}$ red and blue respectively.}\label{fig2.5:pinching}
\end{figure}

We emphasize an important point in the proof above: edge-disjoint paths in a graph $G$ induce edge-disjoint paths in a quotient $G^\pi$. Thus, if two vertices $v \neq w$ in $G$ are not identified $v \overset{\pi}{\not\sim} w$ in $G^\pi$, then their edge connectivity (weakly) increases $\lambda_{G^\pi}(v, w) \geq \lambda_G(v, w)$. This immediately implies the following useful lemma for identifying cactus quotients.

\begin{lemma}[A 3-connection criteria for cacti]\label{lem:3_connection}
Let $G = (V, E)$ be a finite multigraph, and suppose that the vertices $v \neq w \in V$ form a 3-connection in $G$. If a partition $\pi \in \mcal{P}(V)$ induces a cactus $G^\pi$, then $v \overset{\pi}{\sim} w$.
\end{lemma}

We give a number of applications of the 3-connection lemma to computations in the UE traffic space $(\mcal{G}(\mcal{A}), \tau_\varphi)$.

\begin{prop}\label{prop:test_graph_3_connection}
Let $T = (V, E, \gamma) \in \mcal{T}\langle\mcal{A}\rangle$. If the vertices $v \neq w \in V$ form a 3-connection, then
\[
\tau_\varphi[T] = \tau_\varphi[T_{v \sim w}],
\]
where $T_{v \sim w}$ is the test graph obtained from $T$ by identifying $v$ and $w$.
\end{prop}
\begin{proof}
Since $\tau_\varphi^0$ is supported on cacti, the 3-connection lemma implies that
\[
\tau_\varphi[T] = \sum_{\pi \in \mcal{P}(V)} \tau_\varphi^0[T^\pi] = \sum_{\substack{\pi \in \mcal{P}(V) \\ \text{s.t. } v \overset{\pi}{\sim} w}} \tau_\varphi^0[T^\pi] = \tau_\varphi[T_{v \sim w}].
\]
\end{proof}

\begin{cor}\label{cor:graph_monomial_3_connection}
Let $t \in \mcal{G}(\mcal{A})$ be a graph monomial with vertices $v \neq w$ that form a 3-connection. Then
\[
t_{v \sim w} \equiv t \text{ \emph{(mod $\psi$)}},
\]
where $t_{v \sim w}$ is the graph monomial obtained from $t$ by identifying $v$ and $w$.
\end{cor}
\begin{proof}
This amounts to proving that
\[
\psi(t t') = \psi(t_{v \sim w} t'), \qquad \forall t' \in \mcal{G}(\mcal{A}).
\]
Without loss of generality, we can also assume that $t'$ is a graph monomial. For a graph monomial $s$, we recall that $\wtilde{\Delta}(s)$ denotes the test graph obtained from $s$ by identifying the input and the output and forgetting their distinguished roles.
Then there are only two possibilities: either $v$ and $w$ are identified in $\wtilde{\Delta}(t t')$, in which case $\wtilde{\Delta}(t t') = \wtilde{\Delta}(t_{v \sim w} t')$; or, $v$ and $w$ are not identified in $\wtilde{\Delta}(t t')$, in which case $v \neq w$ still form a 3-connection in $\wtilde{\Delta}(t t')$ and $\wtilde{\Delta}(t t')_{v \sim w} = \wtilde{\Delta}(t_{v \sim w} t')$. Proposition \ref{prop:test_graph_3_connection} then implies that
\[
\psi(t t') = \tau_\varphi[\wtilde{\Delta}(t t')] = \tau_\varphi[\wtilde{\Delta}(t_{v \sim w} t')] = \psi(t_{v \sim w} t'). 
\] 
\end{proof}

Iterating the 3-connection lemma, we arrive at the following definition.

\begin{defn}[Quasi-cactus]\label{defn:quasi_cactus}
A multigraph $G = (V, E)$ is said to be a \emph{quasi-cactus} if
\[
\lambda(v, w) \in \{1, 2\}, \qquad \forall v \neq w \in V.
\]
Equivalently, every edge $e \in E$ belongs to at most one simple cycle. See Figure \ref{fig2.7:quasi} for an illustration.
\end{defn}

\begin{figure}
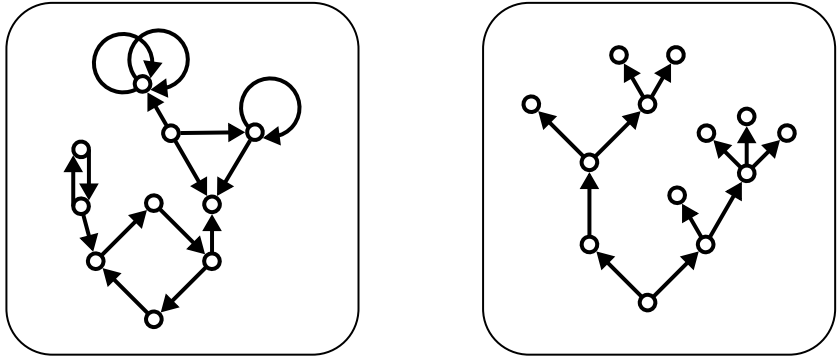
\caption{Examples of quasi-cacti. We can think of a quasi-cactus as a wiring of cacti, as on the left. Of course, a quasi-cactus could consist entirely of the wires, as on the right.}\label{fig2.7:quasi}
\end{figure}

\begin{cor}\label{cor:quasi_cactus}
For any graph monomial $t \in \mcal{G}(\mcal{A})$, there exists a quasi-cactus quotient $t^\pi$ such that
\[
t^\pi \equiv t \text{ \emph{(mod $\psi$)}}.
\] 
\end{cor}

Graph monomials with a cycle that visits both the input and the output play a special role in the free product decomposition of $\mcal{G}(\mcal{A})$. In particular, the construction of the conditional expectation in Theorem \ref{thm:free_product} crucially relies on the following equivalence.

\begin{cor}\label{cor:delta_equivalence}
Let $t \in \Theta(\mcal{B})$ be a graph monomial. Then
\[
t \equiv \Delta(t) \text{ \emph{(mod $\psi$)}}.
\]
\end{cor}
\begin{proof}
If $t \in \Delta(\mcal{B})$, then $\Delta(t) = t$ and we are done. Otherwise, the cycle condition ensures a 2-connection between $v: = \vin \neq \vout =: w$ in $t$. Let $t' \in \mcal{G}(\mcal{A})$ be a graph monomial. If $t' \in \Delta(\mcal{B})$, then $\Delta(t t') = \Delta(t) t'$. Otherwise, $v$ and $w$ form a 3-connection in $\wtilde{\Delta}(t t')$, a third edge-disjoint path coming from the edges of $t'$ (``going out the back door''). Furthermore, note that $\wtilde{\Delta}(t t')_{v \sim w} = \wtilde{\Delta}(\Delta(t) t')$. As before, we use Proposition \ref{prop:test_graph_3_connection} to conclude that
\[
\psi(t t') = \psi(\Delta(t t')) = \tau_\varphi[\wtilde{\Delta}(t t')]= \tau_\varphi[\wtilde{\Delta}(\Delta(t) t')] =  \psi(\Delta(t) t').
\]
\end{proof}

\begin{eg}\label{eg:delta_equivalence}
Let $a, b \in \mcal{A}$. Then $a \circ b \equiv \Delta(a)\Delta(b) \text{ (mod $\psi$)}$. Pictorially,
\[
\foutput \overset{a}{\underset{b}{\leftleftarrows}} \finput \hspace{5pt} \equiv \hspace{5pt} {\scriptstyle{\text{out}}} \overset{a}{\underset{b}{\fffgeight}} {\scriptstyle{\text{in}}} \hspace{5pt} \text{ (mod $\psi$)}.
\]
\end{eg}

We can even use the 3-connection lemma to outright prune t.e.c. subgraphs that are attached at a single vertex. We will generalize this idea to obtain a generic cycle pruning algorithm in Section \ref{sec:cycle_pruning}.

\begin{lemma}\label{lem:prune_tec}
Let $T_1 = (V_1, E_1, \gamma_1)$ and $T_2 = (V_2, E_2, \gamma_2)$ be test graphs in $\mcal{A}$, and suppose that $T_2$ is t.e.c. We write $T_1 \sharp T_2 = (V_1 \sharp V_2, E_1 \sharp E_2, \gamma_1 \sharp \gamma_2)$ for the test graph obtained from $T_1$ and $T_2$ by identifying an arbitrary vertex $v_1$ of $T_1$ with an arbitrary vertex $v_2$ of $T_2$, in which case
\[
\tau_\varphi[T_1 \sharp T_2] = \tau_\varphi[T_1] \tau_\varphi[T_2].
\]
In particular, this factorization is independent of the choice of vertices $v_1$ and $v_2$.
\end{lemma}
\begin{proof}
To begin, note that $V_1 \sharp V_2 = (V_1 \sqcup V_2)/(v_1 \sim v_2)$, $E_1 \sharp E_2 = E_1 \sqcup E_2$, and
\[
\gamma_1 \sharp \gamma_2: E_1 \sharp E_2 \to \mcal{A}, \qquad (\gamma_1 \sharp \gamma_2)|_{E_i} = \gamma_i.
\]
In particular, we denote the amalgamated vertex $v_1 \sim v_2 \in V_1 \sharp V_2$ by $\rho$. For any pair of partitions $(\pi_1, \pi_2) \in \mcal{P}(V_1) \times \mcal{P}(V_2)$, we define the class of partitions
\[
  \mcal{P}_\rho(\pi_1, \pi_2) = \{\pi \in \mcal{P}(V_1 \sharp V_2): \pi|_{V_i} = \pi_i\}.
\]
In other words, $\mcal{P}_\rho(\pi_1, \pi_2)$ consists of the partitions $\pi \in \mcal{P}(V_1 \sharp V_2)$ obtained from $(\pi_1, \pi_2)$ by either keeping a block $B_i \in \pi_i$ (so $B_i \in \pi$) or merging it with at most one other block $B_j \in \pi_j$, where $j \neq i$ (so $B_i \cup B_j \in \pi$). Of course, the block in $\pi_1$ containing $v_1$ and the block in $\pi_2$ containing $v_2$ are necessarily merged. Indeed, the minimal element $\pi_\rho(\pi_1, \pi_2) \in \mcal{P}_\rho(\pi_1, \pi_2)$ for the usual reversed refinement order keeps every other block of $\pi_1$ and $\pi_2$ separate.

By construction,
\[
\bigsqcup_{(\pi_1, \pi_2) \in \mcal{P}(V_1) \times \mcal{P}(V_2)} \mcal{P}_\rho(\pi_1, \pi_2) = \mcal{P}(V_1 \sharp V_2),
\]
and so we can compute
\[
\tau_\varphi[T_1 \sharp T_2] = \sum_{\pi \in \mcal{P}(V_1 \sharp V_2)} \tau_\varphi^0[(T_1 \sharp T_2)^\pi] = \sum_{(\pi_1, \pi_2) \in \mcal{P}(V_1) \times \mcal{P}(V_2)} \sum_{\pi \in \mcal{P}_\rho(\pi_1, \pi_2)} \tau_\varphi^0[(T_1 \sharp T_2)^\pi].
\]
Suppose that $\pi \in \mcal{P}_\rho(\pi_1, \pi_2)\setminus\{\pi_\rho(\pi_1, \pi_2)\}$. Then there exists a vertex $v$ in $T_1$ and a vertex $w$ in $T_2$ such that
\[
v \overset{\pi}{\sim} w \overset{\pi}{\not\sim} \rho.
\]
Since $T_2$ is t.e.c., $w$ and $\rho$ form a 2-connection in $T_2$. But then $v \overset{\pi}{\sim} w$ and $\rho$ form a 3-connection in $(T_1 \sharp T_2)^\pi$, a third edge-disjoint path coming from the edges of $T_1$. The 3-connection lemma then implies that $\tau_\varphi^0[(T_1 \sharp T_2)^\pi] = 0$, and so
\[
\sum_{\pi \in \mcal{P}_\rho(\pi_1, \pi_2)} \tau_\varphi^0[(T_1 \sharp T_2)^\pi] = \tau_\varphi^0[(T_1 \sharp T_2)^{\pi_\rho(\pi_1, \pi_2)}].
\]
The cactus structure of $\tau_\varphi^0$ further implies that if $S$ is a test graph composed of otherwise disjoint test graphs $S_1, \ldots, S_n$ all attached at a single vertex, then
\begin{equation}\label{eq:vertex_factorization}
\tau_\varphi^0[S] = \prod_{i = 1}^n \tau_\varphi^0[S_i].
\end{equation}
In particular, $\tau_\varphi^0[(T_1 \sharp T_2)^{\pi_\rho(\pi_1, \pi_2)}] = \tau_\varphi^0[T_1^{\pi_1}]\tau_\varphi^0[T_2^{\pi_2}]$, and so
\begin{align*}
\tau_\varphi[T_1 \sharp T_2] &= \sum_{(\pi_1, \pi_2) \in \mcal{P}(V_1) \times \mcal{P}(V_2)} \tau_\varphi^0[T_1^{\pi_1}] \tau_\varphi^0[T_2^{\pi_2}] \\
&= \Big(\sum_{\pi_1 \in \mcal{P}(V_1)} \tau_\varphi^0[T_1^{\pi_1}]\Big) \Big(\sum_{\pi_2 \in \mcal{P}(V_2)} \tau_\varphi^0[T_2^{\pi_2}]\Big) = \tau_\varphi[T_1]\tau_\varphi[T_2].
\end{align*}
\end{proof}

\begin{cor}\label{cor:prune_tec}
Let $t \in \mcal{G}(\mcal{A})$ be a graph monomial, and suppose that $T$ is a t.e.c. test graph in $\mcal{A}$. We write $t \sharp T$ for the graph monomial obtained from $t$ and $T$ by identifying an arbitrary vertex $v$ of $t$ with an arbitrary vertex $w$ of $T$, in which case
\[
t \sharp T \equiv \tau[T]t \text{ \emph{(mod $\psi$)}}.
\]
In particular, this equivalence is independent of the choice of vertices $v$ and $w$.
\end{cor}
\begin{proof}
Let $t' \in \mcal{G}(\mcal{A})$ be a graph monomial, and consider the test graphs $T_1 = \wtilde{\Delta}(t t')$ and $T_2 = \wtilde{\Delta}((t \sharp T) t')$. In the notation of Lemma \ref{lem:prune_tec},
\[
T_2 = T_1 \sharp T,
\]
where we attach $T$ to $T_1$ by identifying the vertex $w$ of $T$ with the (image of the) vertex $v \in t$ in $T_1$. Since $T$ is t.e.c., this implies that
\[
\psi((t \sharp T)t') = \tau_\varphi[T_2] = \tau_\varphi[T_1 \sharp T] = \tau_\varphi[T_1] \tau_\varphi[T] = \psi(t t')\tau_\varphi[T] = \psi((\tau_\varphi[T]t)t').
\]
\end{proof}

\begin{eg}\label{eg:prune_tec}
Suppose that $t \in \Delta(\mcal{B})$ is a graph monomial such that the underlying graph of $t$ is a cactus. Then, up to degeneracy, $t$ is a constant:
\[
  t \equiv \prod_{C \in \op{Pads}(t)} \tau_\varphi[C] \text{ (mod $\psi$)}.
\]
\end{eg}

All of these results follow more or less from the same basic idea captured in the 3-connection lemma, namely, that certain identifications must be made in order to contribute to the calculation of the injective traffic state $\tau_\varphi^0$, and so it makes no difference if we make these identifications beforehand. In a slightly different direction, we can also use the edge connectivity characterization of cactus graphs to track the image of t.e.c. subgraphs in a cactus quotient. Doing so, we obtain the following simple but useful corollary. 

\begin{cor}\label{cor:subgraph_cactus}
Let $G$ be a multigraph $(V, E)$ with a t.e.c. subgraph $H = (W, F)$. If a partition $\pi \in \mcal{P}(V)$ induces a cactus $G^\pi$, then the sub-quotient $H^{\pi|_W}$ is also a cactus. Similarly, if $G$ is a multidigraph and $G^\pi$ is an oriented cactus, then the sub-quotient $H^{\pi|_W}$ is also an oriented cactus. 
\end{cor}
\begin{proof}
We will only prove the first statement as the second statement follows almost immediately. Since $H$ is t.e.c., so too is the quotient $H^{\pi|_W} = (\wtilde{W}, F)$. At the same time, $H^{\pi|_W}$ is also a subgraph of the cactus $G^\pi = (\wtilde{V}, E)$, which implies that 
\[
  2 \leq \lambda_{H^{\pi|_W}}(v, w) \leq \lambda_{G^\pi}(v, w) = 2, \qquad \forall v \neq w \in \wtilde{W} \subset \wtilde{V}.
\]
\end{proof}

\subsection{Free products and conditional expectations}\label{sec:free_product}

We divide Theorem \ref{thm:free_product} into parts. First, let $(\kappa_n^{\mcal{B}})_{n \in \N}$ denote the free cumulants of $(\mcal{B}, \psi)$. Note that
\[
  \psi|_{\mcal{A}} = \varphi \quad \implies \quad \kappa_n^{\mcal{B}}|_{\mcal{A}^n} = \kappa_n, \qquad \forall n \in \N. 
\]
We use this to compute the free cumulants of the transposed algebra
\[
  \mcal{A}^\intercal = \Big(a^\intercal = \foutput \overset{a}{\rightarrow} \finput \ \Big| \ a \in \mcal{A} \Big) \subset \mcal{G}(\mcal{A}).
\]
In particular,
\[
  \begin{tikzpicture}[shorten > = 1.5pt]
    \node at (-4.5, 0) {$\psi(a_1^\intercal \cdots a_n^\intercal) = \psi\Big( \ \foutput \xrightarrow{a_1} \cdot \cdots \cdot \xrightarrow{a_n} \finput\ \Big) = \tau_\varphi\Bigg[$};
    \node at (2.5, 0) {$\Bigg] = \varphi(a_n \cdots a_1)$.};
    \draw[fill=black] (.6, 0) circle (1pt);
    \draw[fill=black] (-.6, 0) circle (1pt);
    \draw[fill=black] (.3, .5196) circle (1pt);
    \draw[fill=black] (.3, -.5196) circle (1pt);
    \draw[fill=black] (-.3, .5196) circle (1pt);
    \draw[fill=black] (-.3, -.5196) circle (1pt);
    \draw[semithick, ->] (.3,-.5196) to node[pos=.375, right] {${\scriptstyle a_{n-1}}$} (.6,0);
    \draw[semithick, ->] (-.3,-.5196) to node[midway, below] {${\scriptstyle \cdots}$} (.3,-.5196);
    \draw[semithick, ->] (-.6, 0) to node[pos=.625, left] {${\scriptstyle a_3}$} (-.3,-.5196);
    \draw[semithick, ->] (-.3, .5196) to node[pos=.375, left] {${\scriptstyle a_2}$} (-.6, 0);
    \draw[semithick, ->] (.3, .5196) to node[midway, above] {${\scriptstyle a_1}$} (-.3, .5196);
    \draw[semithick, ->] (.6, 0) to node[pos=.625, right] {${\scriptstyle a_n}$} (.3, .5196);
  \end{tikzpicture}
\]
It follows that $\intercal: (\mcal{A}, \varphi) \to (\mcal{A}^\intercal, \psi|_{\mcal{A}^\intercal})$ defines an involutive anti-isomorphism of $*$-probability spaces, which implies
\begin{equation}\label{eq:transposed_cumulants}
\kappa_n^{\mcal{B}}[a_1^\intercal, \ldots, a_n^\intercal] = \kappa_n[a_n, \ldots, a_1].
\end{equation}
In fact, the same argument shows that $\intercal: \mcal{G}(\mcal{A}) \to \mcal{G}(\mcal{A})$ defines an involutive anti-automorphism of $*$-probability spaces, which implies
\[
  \kappa_n^{\mcal{B}}[t_1^\intercal, \ldots, t_n^\intercal] = \kappa_n^{\mcal{B}}[t_n, \ldots, t_1], \qquad \forall t_i \in \mcal{G}(\mcal{A}).
\]

\begin{lemma}\label{lem:free_independence}
Let $(\mcal{A}, \varphi)$ be a tracial $*$-ps. Then the unital $*$-subalgebras  $\mcal{A}$, $\mcal{A}^\intercal$, and $\Delta(\mcal{B})$ are freely independent in the UE traffic space $(\mcal{B}, \psi)$.
\end{lemma}
\begin{proof}
We will show that any mixed cumulant in $\mcal{A}$, $\mcal{A}^\intercal$, and $\Delta(\mcal{B})$ vanishes. To begin, let $t_1 \cdots t_{2n}$ be an alternating product in the subalgebras $\mcal{A}$, $\mcal{A}^\intercal$, and $\Delta(\mcal{B})$:
\[
t_{2i+1} \in \Delta(\mcal{B}) \quad \text{and} \quad t_{2i} \in \mcal{A} \cup \mcal{A}^\intercal.
\]
Note that we can always write a mixed product in $\mcal{A}$, $\mcal{A}^\intercal$, and $\Delta(\mcal{B})$ in this form by strategically inserting the identity element
\[
\ \foutput \xleftarrow{1_{\mcal{A}}} \finput \ = \ \foutput \xrightarrow{1_{\mcal{A}}} \finput \ = \ \hspace{3pt} \underset{\txio}{\overset{1_{\mcal{A}}}{\dcirclearrowleft}} \hspace{3pt} \ = \  \underset{\txio}{\cdot} \ = 1_{\mcal{B}} \in \mcal{A} \cap \mcal{A}^\intercal \cap \Delta(\mcal{B}).
\]
The multilinearity of the free cumulants allows us to further assume that each $t_i$ is a graph monomial, and so we write
\[
d_i = t_{2i+1} \in \Delta(\mcal{B}) \quad \text{and} \quad a_{\overline{i}}^{\hat{\intercal}(\overline{i})} = t_{2i} \in \mcal{A}^{\hat{\intercal}(\overline{i})},
\]
where $\hat{\intercal}: [\overline{n}] \to \{1, \intercal\}$ indicates the transpose label and the indices $i \in [n]$ and $\overline{i} \in [\overline{n}]$ interlace as before. In particular, we write $[n+\overline{n}] = \{1 < \overline{1} < \cdots < n < \overline{n}\}$. The Kreweras complement then defines a function
\[
\mcal{K}: \mcal{NC}(n) \cup \mcal{NC}(\overline{n}) \to \mcal{NC}(n) \cup \mcal{NC}(\overline{n})
\]
such that
\[
\mcal{K}(\mcal{NC}(n)) = \mcal{NC}(\overline{n}), \quad \mcal{K}(\mcal{NC}(\overline{n})) = \mcal{NC}(n), \quad \text{and} \quad \mcal{K}^2 = \text{id}. 
\]

We compute the trace of our alternating product
\begin{equation}\label{eq:trace_expansion}
\psi(d_1a_{\overline{1}}^{\hat{\intercal}(\overline{1})} \cdots d_na_{\overline{n}}^{\hat{\intercal}(\overline{n})}) = \tau_\varphi[T] = \sum_{\pi \in \mcal{P}(V)} \tau_\varphi^0[T^\pi],
\end{equation}
where $T = \wtilde{\Delta}(d_1a_{\overline{1}}^{\hat{\intercal}(\overline{1})} \cdots d_na_{\overline{n}}^{\hat{\intercal}(\overline{n})}) \in \mcal{T}(\mcal{A})$. In other words, $T$ consists of a cycle $C$ of length $n$ with vertices $v_1, \ldots, v_n$ in counterclockwise order; edges $v_i \overset{e_i}{\sim} v_{i+1}$ in the clockwise orientation if $\hat{\intercal}(\overline{i}) = 1$ (resp., in the counterclockwise orientation if $\hat{\intercal}(\overline{i}) = \intercal$); labels $\gamma(e_i) = a_{\overline{i}}$; and each monomial $d_i$ attached at the vertex $v_i$ via the identification
\[
\op{input}(d_i) = \op{output}(d_i) \sim v_i.
\]

One can think of the notation $T$ as a shortening of $T(d_1, a_{\overline{1}}, \ldots, d_n, a_{\overline{n}})$. In particular, we visualize the test graph $T$ as a cycle with a single loop (or ``petal'') attached at each vertex and the arguments $d_i, a_{\overline{i}}$ as indicating the edge labels: the loops are labeled by the $d_i$ and the edges of the cycle are labeled by the $a_{\overline{i}}$. In the case of a loop, the edge label stands in for the graph monomial $d_i$ that is to be rooted at that location by substitution. We draw the loops as undirected since the orientation plays no role in the substitution; the edges of the cycle are oriented according the transpose label $\hat{\intercal}$. See Figure \ref{fig2.8:flower} for an illustration.

\begin{figure}
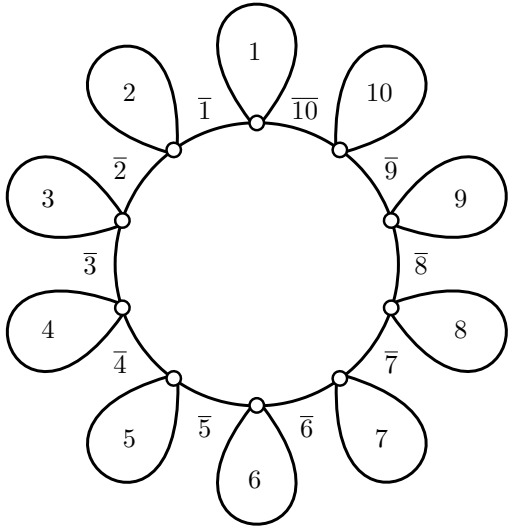
\caption{A visualization of the test graph $T$ for $n = 10$. To highlight the relevant features, we omit the direction of the edges and their labels.}\label{fig2.8:flower}
\end{figure}

Of course, not every partition contributes in the calculation of \eqref{eq:trace_expansion}. For starters, $T^\pi$ must be an oriented cactus. In that case, Corollary \ref{cor:subgraph_cactus} implies that the sub-quotient $C^{\pi|_{C}}$ is also an oriented cactus. Proposition \ref{prop:cactus_non_crossing_partition} then tells us that $\pi$ restricts to a non-crossing partition of $C$ that further satisfies condition \ref{equivalence_oriented_non_crossing}. The enumeration of the vertices $v_1, \ldots, v_n$ (resp., edges $v_i \overset{e_i}{\sim} v_{i+1}$) of the cycle $C$ allows us to consider $\pi|_{C} \in \mcal{NC}([n])$ (resp., $\mcal{K}(\pi|_{C}) \in \mcal{NC}(\overline{n})$) as convenient. The blocks
\[
B = (i_1 < \cdots < i_{\#(B)}) \in \pi|_{C}
\]
then group the petals $d_i$ into bunches (or ``flowers'')
\[
\Big(\prod_{i \in B} d_i\Big)_{B \in \pi|_{C}},
\]
each attached at a single vertex $B$ in $C^{\pi|_{C}}$. Suppose that $\pi$ makes an identification across different flowers (``cross-pollinates''), i.e., there exist vertices $u \in d_i$ and $w \in d_j$ such that $u \overset{\pi}{\sim}{w}$ and $i \overset{\pi|_{C}}{\not\sim} j$. Then the vertices $v_i$ and $v_j$ of the cycle form a 3-connection in $T^\pi$: two edge-disjoint paths come from the cactus $C^{\pi|_{C}}$, and a third comes from the edges of $d_i$ and $d_j$ (see Figure \ref{fig2.9:pollinate}). But then $T^\pi$ cannot possibly be a cactus, and so $\tau_\varphi^0[T_n^\pi] = 0$.

\begin{figure}
\centerfloat
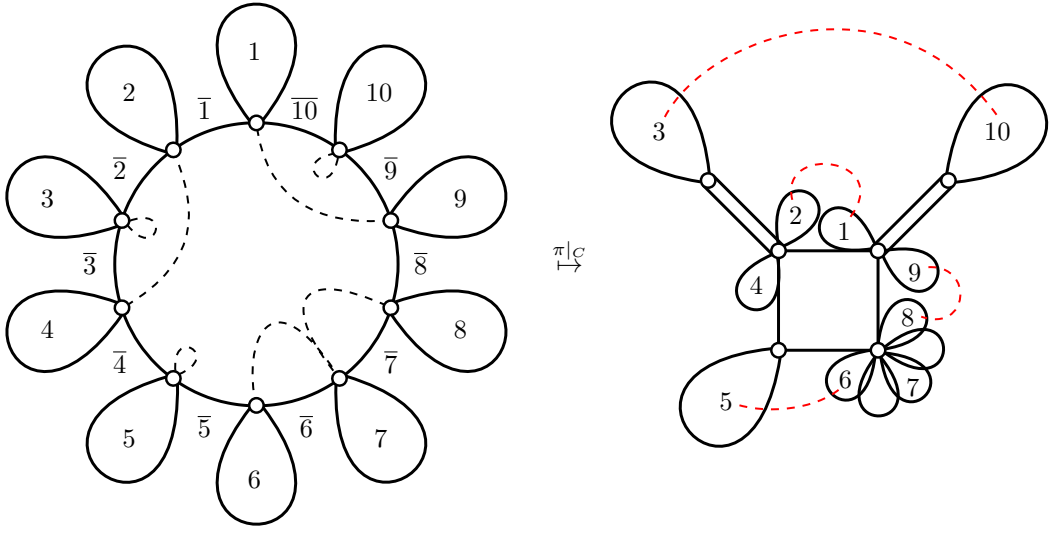
\caption{A visualization of the flowering process. The red dotted lines on the right indicate identifications across different flowers. Such cross-pollination will not produce a cactus.}\label{fig2.9:pollinate}
\end{figure}

Thus, we are left to consider partitions $\pi \in \mcal{P}(V)$ such that $\pi|_{C} \in \mcal{NC}(n)$ and $\pi$ does not cross-pollinate. Iterating the vertex factorization property \eqref{eq:vertex_factorization} of the injective traffic state at each vertex $B \in \pi|_{C}$ in $T^\pi$, we can rewrite \eqref{eq:trace_expansion} as\small
\[
  \sum_{\pi \in \mcal{P}(V)} \tau_\varphi^0[T^\pi] = \sum_{\pi \in \mcal{NC}(n)}\Big(\prod_{B \in \pi}\sum_{\eta_B \in \mcal{P}(V_B)} \tau_\varphi^0\Big[\wtilde{\Delta}\Big(\prod_{i \in B} d_i\Big)^{\eta_B}\Big]\Big)\Big(\tau_\varphi^0[C(a_{\overline{1}}, \ldots, a_{\overline{n}})^\pi]\Big),
\]\normalsize
where $V_B$ is the vertex set of the test graph $\wtilde{\Delta}(\prod_{i \in B} d_i)$. Moreover, by definition,
\[
\sum_{\eta_B \in \mcal{P}(V_B)} \tau_\varphi^0\Big[\wtilde{\Delta}\Big(\prod_{i \in B} d_i\Big)^{\eta_B}\Big] = \tau_\varphi\Big[\wtilde{\Delta}\Big(\prod_{i \in B} d_i\Big)\Big] = \psi\Big(\prod_{i \in B} d_i\Big),
\]
which implies
\begin{align*}
  \prod_{B \in \pi} \sum_{\eta_B \in \mcal{P}(V_B)} \tau_\varphi^0\Big[\wtilde{\Delta}\Big(\prod_{i \in B} d_i\Big)^{\eta_B}\Big] &= \prod_{B \in \pi} \psi\Big(\prod_{i \in B} d_i\Big) = \psi_\pi[d_1, \ldots, d_n] \\
                                                                                                                                      &= \sum_{\substack{\omega \in \mcal{NC}(n) \\ \text{s.t. } \omega \leq \pi}} \kappa_\omega^{\mcal{B}}[d_1, \ldots, d_n].
\end{align*}
Recall that condition \ref{equivalence_oriented_non_crossing} of Proposition \ref{prop:cactus_non_crossing_partition} imposes an additional constraint on $\pi$. Altogether, our expression for the trace becomes
\begin{align*}
  &\psi(d_1a_{\overline{1}}^{\hat{\intercal}(\overline{1})}\cdots d_na_{\overline{n}}^{\hat{\intercal}(\overline{n})}) \\
  = &\sum_{\substack{\pi \in \mcal{NC}(n) \text{ s.t. } \\ \mcal{K}(\pi) = \sigma \cup \rho \in \mcal{NC}(\overline{n}) \\ \text{for some } \sigma \in \mcal{NC}(\hat{\intercal}^{-1}(1)) \\ \text{and } \rho \in \mcal{NC}(\hat{\intercal}^{-1}(\intercal))}} \Big(\sum_{\substack{\omega \in \mcal{NC}(n) \\ \text{s.t. } \omega \leq \pi}} \kappa_\omega^{\mcal{B}}[d_1, \ldots, d_n]\Big)\Big(\kappa_{\mcal{K}(\pi)}^{\mcal{B}}[a_{\overline{1}}^{\hat{\intercal}(\overline{1})}, \ldots, a_{\overline{n}}^{\hat{\intercal}(\overline{n})}]\Big) \\
= &\sum_{\substack{\pi \in \mcal{NC}(n+\overline{n}) \text{ s.t. } \\ \pi = \pi_1 \cup \pi_2, \\ \text{where } \pi_2 = \sigma \cup \rho \in \mcal{NC}(\overline{n}) \\ \text{for some } \sigma \in \mcal{NC}(\hat{\intercal}^{-1}(1)) \\ \text{and } \rho \in \mcal{NC}(\hat{\intercal}^{-1}(\intercal)) \\ \text{and } \pi_1 \leq \mcal{K}(\pi_2) \in \mcal{NC}(n)}} \kappa_{\pi}^{\mcal{B}}[d_1, a_{\overline{1}}^{\hat{\intercal}(\overline{1})}, \ldots, d_n, a_{\overline{n}}^{\hat{\intercal}(\overline{n})}],
\end{align*}
where we have applied the cumulant formula \eqref{eq:transposed_cumulants} for the transposed algebra in the first equality and the Kreweras complement $\mcal{K}$ to reindex the sum in the second equality. In particular, we see that there are no contributions from mixed cumulants in $\mcal{A}$, $\mcal{A}^\intercal$, and $\Delta(\mcal{B})$. It follows that mixed cumulants in $\mcal{A}$, $\mcal{A}^\intercal$, and $\Delta(\mcal{B})$ vanish, as was to be shown.
\end{proof}

We move on to the construction of the conditional expectation $\mathscr{E}: \mcal{B} \to \mcal{A} * \mcal{A}^\intercal * \Delta(\mcal{B})$. For a simple connected graph $G = (V, E)$, we recall the construction of the \emph{block-cut tree} $bc(G)$ of $G$. The vertices of $bc(G)$ consist of both the cut-vertices of $G$ and the blocks of $G$ with edges determined by inclusion: we connect a cut-vertex $v$ to a block $H$ if $ v \in H$. As the name suggests, the block-cut tree is indeed a tree. It will be convenient to distinguish between the two different classes of vertices in $bc(G)$. In particular, we use circular vertices for the cut-vertices and square vertices for the blocks.

We will need a simple modification of the block-cut tree construction in the case of a bi-rooted multidigraph $G = (V, E, \source, \target, \vin, \vout)$. Allowing for multiple (directed) edges does not materially affect the construction; however, allowing for loops creates an issue when determining the blocks. In particular, a loop based at a cut-vertex will belong to more than one block of $G$. To account for this, we temporarily remove any loop based at a cut-vertex of $G$, resulting in a graph $\wtilde{G}$. We reintroduce the loops in the block-cut tree $bc(\wtilde{G})$ by adding a single block for each set of loops based at a given cut-vertex with an edge between the two to indicate the inclusion. This process ensures that we have a faithful reconstruction of the original graph $G$ from our modified block-cut tree. Furthermore, if either distinguished vertex $\vin$ or $\vout$ is not a cut-vertex, then we add it to our tree as a circular vertex, colored black, and attach it to its corresponding (unique) block (if $\vin = \vout$ then we only add a single vertex). A moment's thought shows that the resulting graph, which we denote $bcd(G)$, is of course still a tree. See Figure \ref{fig2.10:bcd} for an illustration.

\begin{lemma}\label{lem:conditional_expectation}
There exists a homomorphic conditional expectation $\mathscr{E}: \mcal{B} \to \mcal{A} * \mcal{A}^\intercal * \Delta(\mcal{B})$ such that
\[
\mathscr{E}^{-1}(\mcal{A}) = \mcal{A}, \quad \mathscr{E}^{-1}(\mcal{A}^\intercal) = \mcal{A}^\intercal, \quad \mathscr{E}^{-1}(\Delta(\mcal{B})) = \Theta(\mcal{B}),
\]
and 
\begin{equation}\label{eq:ce_equivalence}
\mathscr{E}(t) \equiv t \text{ \emph{(mod $\psi$)}}, \qquad \forall t \in \mcal{B}.
\end{equation}
\end{lemma}

\begin{figure}
\centerfloat
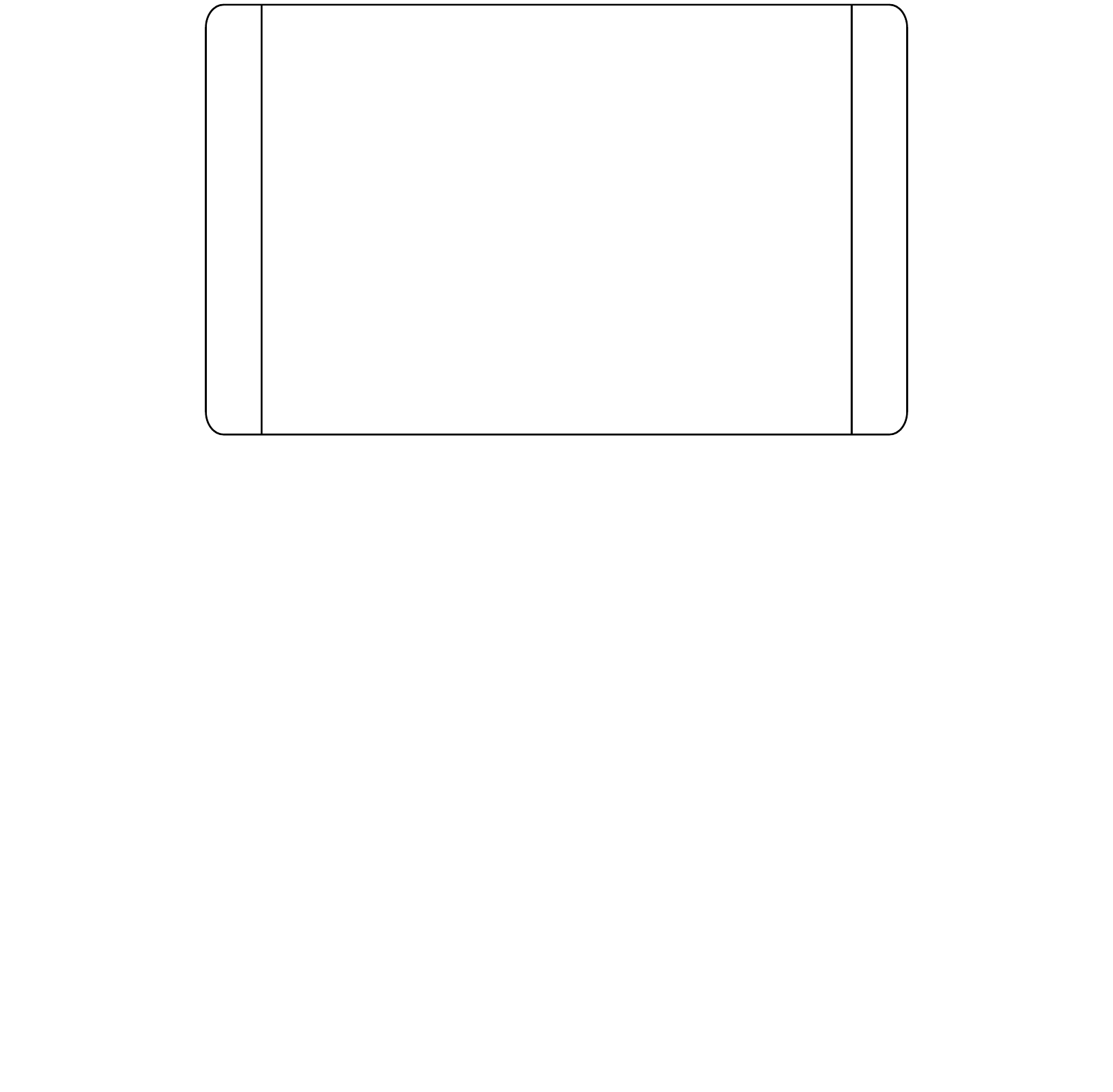
\caption{An example of the modified block-cut tree construction. Starting with a bi-rooted multidigraph $G$, we color each of the non-root cut vertices. Note that in this case the output $\vout$ is not a cut-vertex, and so we append it to the standard block-cut tree.}\label{fig2.10:bcd}
\end{figure}

\begin{proof}
Let $t = (V, E, \vin, \vout, \gamma) \in \mcal{B}$ be a graph monomial with modified block-cut tree $bcd(t)$. By definition, there exists a unique path $P = (v_1, B_1, \ldots, v_{n-1}, B_{n-1}, v_n)$ from $\vin$ to $\vout$ in $bcd(t)$ (in particular, $v_1 = \vin$ and $v_n = \vout$). For each vertex $w$ on this path, we consider the connected component containing $w$ in $bcd(t)$ off of the path $P$. In particular, let $\wtilde{C}(w)$ denote the connected component containing $w$ after removing the edges of the path $P$. Note that each connected component $\wtilde{C}(w)$ corresponds to a connected edge-labeled subgraph $C(w)$ of the original graph monomial $t$. Each of the subgraphs $C(w)$ defines a graph monomial after a natural choice of distinguished vertices:
\[
d_i = (C(v_i), v_i, v_i) \in \Delta(\mcal{B}) \quad \text{and} \quad m_i = (C(B_i), v_i, v_{i+1}) \in \mcal{B}.
\]
For example, if $\vin = \vout \in t$, then
\[
P = (v_1), \quad \wtilde{C}(v_1) = bcd(t), \quad \text{and} \quad t = d_1 \in \Delta(\mcal{B}).
\]
The last equality is a special case of the general factorization 
\begin{equation}\label{eq:factorization}
t = d_n m_{n-1} d_{n-1} \cdots m_1 d_1.
\end{equation}
See Figure \ref{fig2.11:factor} for an illustration.

\begin{figure}
\centerfloat
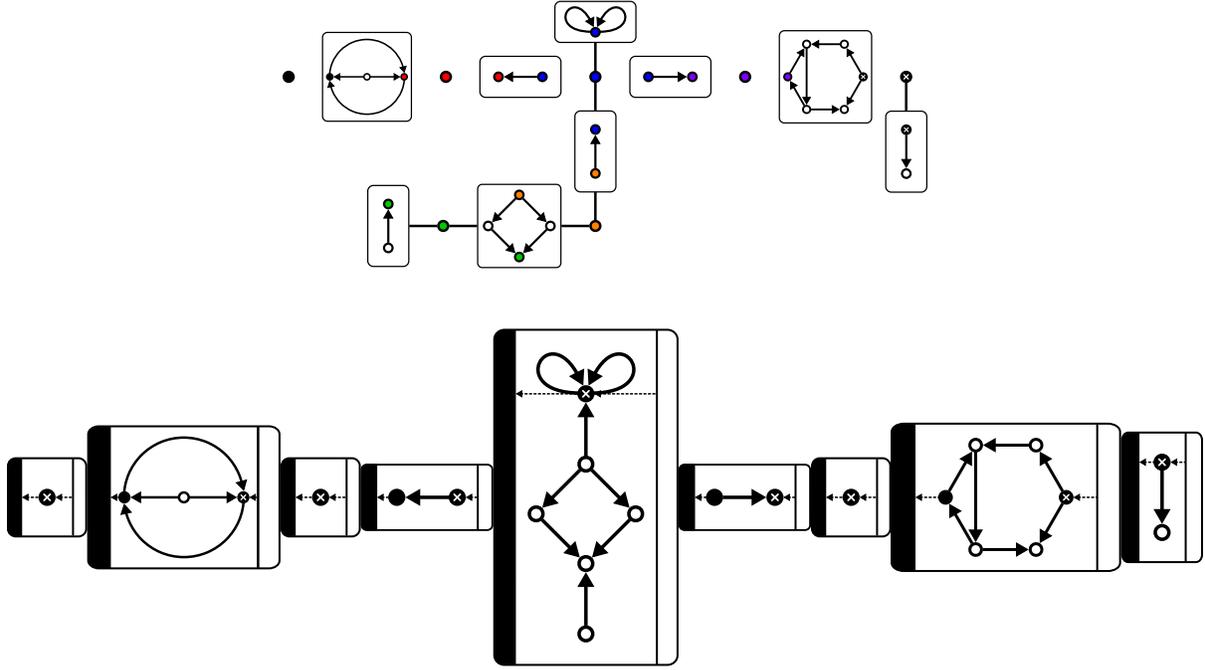
\caption{Removing the edges of the path $P$. We apply this procedure to the block-cut tree from Figure \ref{fig2.10:bcd} to identify the components and produce the factorization \eqref{eq:factorization}.}\label{fig2.11:factor}
\end{figure}

Suppose that a factor $m_i \not\in \mcal{A} \cup \mcal{A}^\intercal$. Then there must be a simple cycle in $m_i$ that visits both the input $v_i$ and the output $v_{i+1}$. Indeed, this follows from the vertex version of Menger's theorem: if $v_i \neq v_{i+1}$ are the only vertices in $m_i$, then $m_i \not\in \mcal{A} \cup \mcal{A}^\intercal$ implies that there are multiple edges between $v_i$ and $v_{i+1}$; if $v_i \neq v_{i+1}$ are not the only vertices in $m_i$, then the lack of cut-vertices in $m_i$ implies the existence of such a simple cycle. Corollary \ref{cor:delta_equivalence} then implies that
\[
m_i \equiv \Delta(m_i) \text{ (mod $\psi$)}.
\] 
To apply this to our factorization \eqref{eq:factorization}, we define a linear operator $\nabla: \mcal{B} \to \mcal{B}$ on graph monomials $m$ by the formula
\[
\nabla(m) =
\begin{dcases}
m &\text{if } m \in \mcal{A} \cup \mcal{A}^\intercal, \\
\Delta(m) &\text{otherwise.}
\end{dcases}
\]
By construction, $\mathscr{E}(t) := d_n\nabla(m_{n-1})d_{n-1}\cdots\nabla(m_1)d_1 \in \mcal{A} * \mcal{A}^\intercal * \Delta(\mcal{B})$ satisfies
\[
\mathscr{E}(t) \equiv t \text{ (mod $\psi$)}.
\]

Extending $\mathscr{E}$ to $\mcal{B}$ by linearity, properties \ref{ce_expectation} and \ref{ce_star} of a conditional expectation (Definition \ref{defn:conditional_expectation}) follow almost immediately. Moreover, property \ref{ce_bimodule} will follow if we can prove that $\mathscr{E}$ is a homomorphism $\mathscr{E}(t_1t_2) = \mathscr{E}(t_1)\mathscr{E}(t_2)$. As before, we can assume that $t_1$ and $t_2$ are graph monomials, say with block-cut tree factorizations
\[
t_1 = d_{n_1}^{(1)}m_{n_1 - 1}^{(1)}d_{n_1 - 1}^{(1)} \cdots m_1^{(1)}d_1^{(1)} \quad \text{and} \quad t_2 = d_{n_2}^{(2)}m_{n_2 - 1}^{(2)}d_{n_2 - 1}^{(2)} \cdots m_1^{(2)}d_1^{(2)}
\]
coming from the paths\small
\[
P_1 = (v_1^{(1)}, B_1^{(1)}, \ldots, v_{n_1-1}^{(1)}, B_{n_1-1}^{(1)}, v_{n_1}^{(1)}) \quad \text{and} \quad P_2 = (v_1^{(2)}, B_1^{(2)}, \ldots, v_{n_2-1}^{(2)}, B_{n_2-1}^{(2)}, v_{n_2}^{(2)})
\]\normalsize
in $bcd(t_1)$ and $bcd(t_2)$ respectively. The reader will now see why we have insisted on including the distinguished vertices in the modified block-cut tree. In particular, the amalgamated vertex $\vout^{(2)} = \vin^{(1)} \in t_1t_2$ is either a cut-vertex of $t_1t_2$ or once again a distinguished vertex $\vin$ or $\vout \in t_1t_2$. The path from $\vin$ to $\vout$ in $bcd(t_1t_2)$ can then be written as
\[
P = (v_1^{(2)}, B_1^{(2)}, \ldots, v_{n_2-1}^{(2)}, B_{n_2-1}^{(2)}, v_{n_2}^{(2)} = v_1^{(1)}, B_1^{(1)}, \ldots, v_{n_1-1}^{(1)}, B_{n_1-1}^{(1)}, v_{n_1}^{(1)}),
\]
which gives rise to the factorization
\[
t_1t_2 = d_{n_1}^{(1)}m_{n_1 - 1}^{(1)}d_{n_1 - 1}^{(1)} \cdots m_1^{(1)}(d_1^{(1)}d_{n_2}^{(2)})m_{n_2 - 1}^{(2)}d_{n_2 - 1}^{(2)} \cdots m_1^{(2)}d_1^{(2)}.
\]
We conclude that\small
\begin{align*}
\mathscr{E}(t_1t_2) &= d_{n_1}^{(1)}\nabla(m_{n_1 - 1}^{(1)})d_{n_1 - 1}^{(1)} \cdots \nabla(m_1^{(1)})(d_1^{(1)}d_{n_2}^{(2)})\nabla(m_{n_2 - 1}^{(2)})d_{n_2 - 1}^{(2)} \cdots \nabla(m_1^{(2)})d_1^{(2)}\\
                    &= (d_{n_1}^{(1)}\nabla(m_{n_1 - 1}^{(1)})d_{n_1 - 1}^{(1)} \cdots \nabla(m_1^{(1)})d_1^{(1)})(d_{n_2}^{(2)}\nabla(m_{n_2 - 1}^{(2)})d_{n_2 - 1}^{(2)} \cdots \nabla(m_1^{(2)})d_1^{(2)}) \\
                    &= \mathscr{E}(t_1)\mathscr{E}(t_2).
\end{align*}\normalsize
Finally, the equalities
\[
\mathscr{E}^{-1}(\mcal{A}) = \mcal{A}, \quad \mathscr{E}^{-1}(\mcal{A}^\intercal) = \mcal{A}^\intercal, \quad \text{and} \quad \mathscr{E}^{-1}(\Delta(\mcal{B})) = \Theta(\mcal{B})
\]
follow virtually by definition. 
\end{proof}

\begin{rem}\label{rem:well_defined_ce}
The construction of our conditional expectation $\mathscr{E}$ relies on taking a particular (graphical) realization of a graph monomial $t \in \mcal{B}$. Of course, one should then verify that this construction is well-defined. We can restrict our attention to the monomials $m_i$ coming from the blocks $B_i$ since $\mathscr{E}$ only acts on these factors. Moreover, since the action of $\mathscr{E}$ is defined on each factor individually, we can further restrict to a single factor $m_i$. Here, there are two cases to consider. First, suppose that a factor $m_i = P(a_1, \ldots, a_k) \in \mcal{A}$. Then
\[
m_i = \foutput \xleftarrow{P(a_1, \ldots, a_k)} \finput = P\big( \ \foutput \xleftarrow{a_1} \finput \ ,  \cdots, \ \foutput \xleftarrow{a_k} \finput \ \big).
\]
Running through the algorithm for $\mathscr{E}$ on $m_i$, we have the equality
\begin{align*}
\mathscr{E}\Big(P\big( \ \foutput \xleftarrow{a_1} \finput \ ,  \cdots, \ \foutput \xleftarrow{a_k} \finput \ \big)\Big) &= P\big( \ \foutput \xleftarrow{a_1} \finput \ ,  \cdots, \ \foutput \xleftarrow{a_k} \finput \ \big) \\
&= \foutput \xleftarrow{P(a_1,\ldots, a_k)} \finput \\
&= \mathscr{E}\Big( \ \foutput \xleftarrow{P(a_1, \ldots, a_k)} \finput \ \Big),
\end{align*}
and similarly for $m_i = P(a_1, \ldots, a_k)^\intercal$.

Next, suppose that $h \in \mcal{B}$ is a graph monomial such that $h = z d \in \Delta(\mcal{B})$ for some $z \in \C$ and graph monomial $d \in \Delta(\mcal{B})$. For example, it could be that $z \in \R_+$ and
\[
h = \ \foutput \xleftarrow{\sqrt{z}} \cdot \xrightarrow{\sqrt{z}} \finput \ = z \big(\underset{\txio}{\cdot}\big).
\]
A factor $m_i$ could then take the form
\[
m_i = \ \foutput \overset{g}{\underset{h}{\leftleftarrows}} \finput \ = z\big(\ {\scriptstyle{\text{out}}} \overset{g}{\underset{d}{\fffgeight}} {\scriptstyle{\text{in}}} \ \big),
\]
where $g$ stands in for an arbitrary graph monomial. Again, running through the algorithm for $\mathscr{E}$ on $m_i$, 
\[
\mathscr{E}\Big( \ \foutput \overset{g}{\underset{h}{\leftleftarrows}} \finput \ \Big) = \hspace{5pt} {\scriptstyle{\text{out}}} \overset{g}{\underset{h}{\fffgeight}} {\scriptstyle{\text{in}}} \hspace{5pt} = z\big( \ {\scriptstyle{\text{out}}} \overset{g}{\underset{d}{\fffgeight}} {\scriptstyle{\text{in}}} \ \big) = \mathscr{E}\Big(z\big(\ {\scriptstyle{\text{out}}} \overset{g}{\underset{d}{\fffgeight}} {\scriptstyle{\text{in}}} \ \big)\Big).
\]

These are the only cases where the identifications defining $\mcal{B} = \C\mcal{G}^{(2)}\langle\mcal{A}\rangle/\mcal{I}$ can affect the path $P$ in $bcd(t)$: the former by expanding the block $b_i$ by introducing cut vertices in $m_i$; the latter by compressing the block $b_i$ by identifying the vertices $v_i \neq v_{i+1} \in m_i$. In any case, we see that the action of $\mathscr{E}$ is well-defined.
\end{rem}

In general, a conditional expectation is only unique up to degeneracy. In particular, if $\mathscr{F}: \mcal{B} \to \mcal{A} * \mcal{A}^\intercal * \Delta(\mcal{B})$ is also a conditional expectation, then
\[
\mathscr{F}(t) \equiv \mathscr{E}(t) \text{ (mod $\psi$)}, \qquad \forall t \in \mcal{B}.
\]
Indeed, even with the additional properties stated in Lemma \ref{lem:conditional_expectation}, one can still find such maps $\mathscr{F} \neq \mathscr{E}$. To see this, note that our algorithm for $\mathscr{E}$ only operates on the cut-vertices of $t$ along the path $P$. The map $\mathscr{E}$ satisfies the equivalence \eqref{eq:ce_equivalence} precisely because it only identifies redundant vertices (i.e., vertices that would need to be identified anyway in order to contribute to the calculation of the injective traffic state). Yet, there can be many such redundant vertices, whereas $\mathscr{E}$ only considers a ``minimal'' subset of them. One can modify the map $\mathscr{E}$ while preserving all of the desired properties by specifying a more vigilant approach to dealing with such redundancies within each component $C(v_i)$ and $C(B_i)$ defined by the path $P$. Doing so clearly defines other maps $\mathscr{F} \neq \mathscr{E}$.

The last piece of Theorem \ref{thm:free_product} now follows almost immediately.

\begin{cor}\label{cor:free_product}
The UE traffic space $(\mcal{B}, \psi) = (\mcal{G}(\mcal{A}), \varphi_{\tau_\varphi})$ admits the free product decomposition
\[
(\mcal{B}, \psi) = (\mcal{A}, \psi|_{\mcal{A}}) * (\mcal{A}^\intercal, \psi|_{\mcal{A}^\intercal}) * (\Theta(\mcal{B}), \psi|_{\Delta(\mcal{B}}). 
\]
\end{cor}
\begin{proof}
Our modified block-cut tree algorithm already proves the \emph{algebraic} free product decomposition $\mcal{B} = \mcal{A} * \mcal{A}^\intercal * \Theta(\mcal{B})$. We can further use our conditional expectation $\mathscr{E}$ to pull back the free independence of $\mcal{A}$, $\mcal{A}^\intercal$, and $\Delta(\mcal{B})$ to $\mathscr{E}^{-1}(\mcal{A}) = \mcal{A}$, $\mathscr{E}^{-1}(\mcal{A}^\intercal) = \mcal{A}^\intercal$, and $\mathscr{E}^{-1}(\Delta(\mcal{B})) = \Theta(\mcal{B})$.
\end{proof}

To prove Proposition \ref{prop:commutification}, it suffices to consider a pair of unital $*$-subalgebras $\mcal{A}_1, \mcal{A}_2 \subset \mcal{A}$.
\begin{lemma}\label{lem:commutification}
Let $\mcal{A}_1$ and $\mcal{A}_2$ be freely independent unital $*$-subalgebras of a tracial $*$-ps $(\mcal{A}, \varphi)$. Then the commutative subalgebras $\Delta(\mcal{B}_1) = \Delta(\mcal{G}(\mcal{A}_1))$ and $\Delta(\mcal{B}_2) = \Delta(\mcal{G}(\mcal{A}_2))$ are classically independent in $(\mcal{B}, \psi)$. 
\end{lemma}
\begin{proof}
Let $t_i \in \Delta(\mcal{B}_i)$ be a graph monomial. Then
\[
  t_1t_2 = \hspace{5pt}{\scriptstyle{\text{out}}} \overset{t_1}{\underset{t_2}{\fffgeight}} {\scriptstyle{\text{in}}} \hspace{5pt}, \qquad \wtilde{\Delta}(t_1t_2) = \hspace{5pt} \overset{t_1}{\underset{t_2}{\fffgeight}} \hspace{5pt},
\]
and we can compute the trace as
\[
  \psi(t_1t_2) = \tau_\varphi[\wtilde{\Delta}(t_1 t_2)] = \sum_{\pi \in \mcal{P}(V)} \tau_\varphi^0[\wtilde{\Delta}(t_1 t_2)^\pi].
\]
We think of the edges of $t_1$ (resp., $t_2$) as being colored black (resp., red) to indicate the edge labels in $\mcal{A}_1$ (resp., $\mcal{A}_2$). Since $\mcal{A}_1$ and $\mcal{A}_2$ are freely independent, $\tau_\varphi^0[\wtilde{\Delta}(t_1 t_2)^\pi] = 0$ unless $\wtilde{\Delta}(t_1 t_2)^\pi$ is an oriented cactus whose pads are each of a uniform color (a \emph{colored oriented cactus}). But this implies that $\wtilde{\Delta}(t_i)^{\pi|_{\wtilde{\Delta}(t_i)}}$ is a sub-cactus of $\wtilde{\Delta}(t_1 t_2)^\pi$. Moreover, note that the sub-cacti $\wtilde{\Delta}(t_1)^{\pi|_{\wtilde{\Delta}(t_1)}}$ and  $\wtilde{\Delta}(t_2)^{\pi|_{\wtilde{\Delta}(t_2)}}$ can only have one vertex in common as two such vertices would form a 4-connection, with two edge-disjoint paths coming from the black edges and two edge-disjoint paths coming from the red edges. Of course, this common vertex must be
\[
  \rho := \op{input}(t_1) = \op{output}(t_1) = \op{input}(t_2) = \op{output}(t_2) \in \wtilde{\Delta}(t_1t_2),
\]
in which case
\[
\tau_\varphi^0[\wtilde{\Delta}(t_1 t_2)^\pi] = \tau_\varphi^0[\wtilde{\Delta}(t_1)^{\pi|_{\wtilde{\Delta}(t_1)}}]\tau_\varphi^0[\wtilde{\Delta}(t_2)^{\pi|_{\wtilde{\Delta}(t_2)}}].
\]
In particular, if $\wtilde{\Delta}(t_1 t_2)^\pi$ is a colored oriented cactus with $v_i \in \wtilde{\Delta}(t_i)$ such that $v_1 \overset{\pi}{\sim} v_2$, then it is necessarily the case that $v_1 \overset{\pi}{\sim} \rho \overset{\pi}{\sim} v_2$. This allows us to factor the trace
\begin{align*}
\psi(t_1t_2) = \sum_{\pi \in \mcal{P}(V)} \tau_\varphi^0[\wtilde{\Delta}(t_1 t_2)^\pi] &= \Big(\sum_{\pi_1 \in \mcal{P}(V_1)} \tau_\varphi^0[\wtilde{\Delta}(t_1)^{\pi_1}]\Big)\Big(\sum_{\pi_2 \in \mcal{P}(V_2)} \tau_\varphi^0[\wtilde{\Delta}(t_2)^{\pi_2}]\Big) \\
&= \tau_\varphi[\wtilde{\Delta}(t_1)] \tau_\varphi[\wtilde{\Delta}(t_2)] = \psi(t_1)\psi(t_2),
\end{align*}
as was to be shown.
\end{proof}
The general case of Proposition \ref{prop:commutification} now follows from the associativity of free independence. Our proof relies on an explicit calculation made possible by the cactus structure of the injective traffic state $\tau_\varphi^0$. One can also prove this by appealing to the relationship between traffic independence and classical/free independence. More precisely, Proposition 4.8 in \cite{CDM16} states that the free independence of the $(\mcal{A}_i)_{i \in I}$ in $(\mcal{A}, \varphi)$ is equivalent to the traffic independence of the $(\mcal{G}(\mcal{A}_i))_{i \in I}$ in $(\mcal{G}(\mcal{A}), \tau_\varphi)$. We can specialize this to the traffic independence of the sub-traffic-spaces $(\Delta(\mcal{G}(\mcal{A}_i)))_{i \in I}$, where $\Delta(\mcal{G}(\mcal{A}_i)) \subset \mcal{G}(\mcal{A}_i)$. Theorem 5.5 of \cite{Mal11} proves that traffic independence and classical independence are equivalent for diagonal traffic random variables $\Delta(t) = t$, and so the result follows.

\subsection{A cycle pruning algorithm}\label{sec:cycle_pruning}

We generalize the ideas of Section \ref{sec:cacti} and \ref{sec:free_product} to formulate a generic cycle pruning algorithm. This generality comes at a cost: in contrast to our earlier results, our equivalence now takes the form of a graph \emph{polynomial}.
\begin{lemma}\label{lem:pruning}
Let $t = \Delta(d_0a_0^{\hat{\intercal}(0)} \cdots d_na_n^{\hat{\intercal}(n)})$ be a graph monomial in $\mcal{B}$, where $d_i \in \Delta(\mcal{B})$ and $a_i \in \mcal{A}$. For any subset $A = \{i_1 < \cdots < i_{\#(A)}\} \subset [n]$, we define the $A$-segmented factors 
\[
m_{A, k} = a_{i_k}^{\hat{\intercal}(i_k)} d_{i_k + 1} \cdots a_{i_{k+1}-2}^{\hat{\intercal}(i_{k+1}-2)}d_{i_{k+1} - 1} a_{i_{k+1} - 1}^{\hat{\intercal}(i_{k+1} - 1)}
\]
for $0 \leq k \leq \#(A)$, where $i_0 = 0$ and $i_{\#(A)+1} = n+1$ (see Figure \ref{fig2.12:segments}). In particular,
\[
d_0a_0^{\hat{\intercal}(0)} \cdots d_na_n^{\hat{\intercal}(n)} = d_0m_{A, 0}d_{i_1}m_{A, 1} \cdots d_{i_{\#(A)}}m_{A, \#(A)}.
\]
Then
\[
t \equiv \sum_{A \subset [n]} \Big(\psi(m_A)(d_0\prod_{i \in A} d_i)\Big) \text{ \emph{ (mod $\psi$)}}, 
\]
where
\[
m_A = \sum_{\substack{B \subset [n] \\ \text{\emph{s.t. }} A \subset B}} \Big((-1)^{\#(B\setminus A)}\prod_{i \in B \setminus A} d_i\prod_{k = 0}^{\#(B)} \Delta(m_{B, k})\Big).
\]
\end{lemma}

\begin{figure}
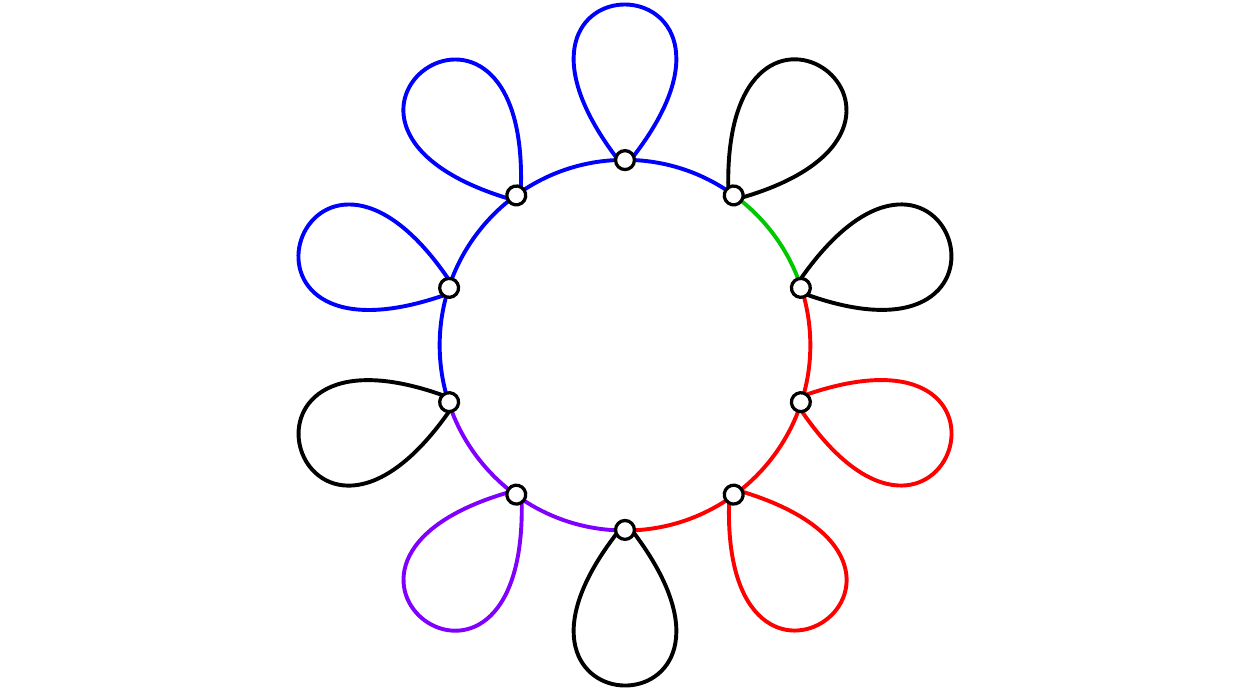
\caption{An example of $A$-segmented factors for $n = 9$ and $A = \{3, 4, 8\}$. The factors $m_{A, k}$ are colored red, green, blue, and purple for $k = 0, 1, 2, 3$ respectively.}\label{fig2.12:segments}
\end{figure}

\begin{proof}
If $n = 0$, then $t = \Delta(d_0a_0^{\hat{\intercal}(0)}) = d_0\Delta(a_0^{\hat{\intercal}(0)})$ has a loop $e$ with edge label $\gamma(e) = a_0$. Corollary \ref{cor:prune_tec} implies that
\[
  t \equiv \psi(\Delta(a_0^{\hat{\intercal}(0)}))d_0 \text{ (mod $\psi$)},
\]
as was to be shown.

Now assume that $n \geq 1$. As before, we think of $t = \Delta(d_0a_0^{\hat{\intercal}(0)} \cdots d_na_n^{\hat{\intercal}(n)})$ as a flower: in this case, a cycle of length $n+1$ with the petals $d_i$ based at each vertex $v_i$. For any subset $A \subset [n]$, we define the graph monomial $t_A$ by the identifying the vertices $v_0 \sim v_i$ for $i \in A$. In other words,
\[
  t_A = d_0 \prod_{i \in A} d_i \prod_{k=0}^{\#(A)} \Delta(m_{A, k}).
\]
We use this to define the graph polynomials
\[
  p_A(t) = \sum_{\substack{B \subset [n] \\ \text{s.t. } A \subset B}} (-1)^{\#(B\setminus A)}t_B,
\]
which satisfy the identity
\begin{align*}
  \sum_{A \subset [n]} p_A(t) &= \sum_{A \subset [n]} \sum_{\substack{B \subset [n] \\ \text{s.t. } A \subset B}} (-1)^{\#(B\setminus A)} t_B \\
  &= \sum_{B \subset [n]} \Big(\sum_{A \subset B} (-1)^{\#(B\setminus A)}\Big) t_B= t_\emptyset = t.
\end{align*}
So, the result will follow if we can show that
\[
p_A(t) \equiv \psi(m_A)d_0\prod_{i \in A} d_i \text{ (mod $\psi$)},
\]
which follows trivially if $A = [n]$. We assume hereafter that $[n]\setminus A \neq \emptyset$.

Let $t' \in \mcal{B}$ be a graph monomial. We will prove that
\[
\psi(p_A(t)t') = \psi(m_A)\psi((d_0\prod_{i \in A} d_i) t').
\]
For starters, note that we can factor 
\[
\Delta(t_B t') = \Big(d_0\prod_{i \in A} d_i \Delta(t')\Big)\Big(\prod_{i \in B \setminus A} d_i \prod_{k = 0}^{\#(B)} \Delta(m_{B, k})\Big) =: l_A(B)r_A(B)
\]
into a left side $l_A = l_A(B)$ and a right side $r_A(B)$ joined at the single vertex
\[
\rho := \op{input}(l_A) = \op{output}(l_A) = \op{input}(r_A(B)) = \op{output}(r_A(B)) \in \Delta(t_B t').
\]
Similarly, we define the test graphs
\begin{align*}
  L_A &= \wtilde{\Delta}(l_A) = (V_{L_A}, E_{L_A}, \gamma_{L_A}); \\
  R_A(B) &= \wtilde{\Delta}(r_A(B)) = (V_{R_A(B)}, E_{R_A(B)}, \gamma_{R_A(B)}),
\end{align*}
in which case 
\[
\wtilde{\Delta}(t_B t') = L_A \sharp R_A(B) = (V_{L_A \sharp R_A(B)}, E_{L_A \sharp R_A(B)}, \gamma_{L_A \sharp R_A(B)}).
\]
Here, we use $L_A \sharp R_A(B)$ to denote the test graph obtained from $L_A$ and $R_A(B)$ by identifying the vertices $\op{input}(l_A) = \op{output}(l_A) \in L_A$ and $\op{input}(r_A(B)) = \op{output}(r_A(B)) \in R_A(B)$. For convenience, we write $r_A = r_A(A)$ and $R_A = R_A(A)$.

By definition,
\begin{align}\label{eq:left_right}
\begin{split}
  \psi(t_B t') = \psi(l_Ar_A(B)) &= \tau_\varphi[L_A \sharp R_A(B)] \\
  &= \sum_{\pi_B \in \mcal{P}(V_{L_A \sharp R_A(B)})} \tau_\varphi^0[(L_A \sharp R_A(B))^{\pi_B}].
\end{split}
\end{align}
It will be convenient to reindex the sum in terms of partitions $\pi \in \mcal{P}(V_{L_A} \sharp V_{R_A})$. For any pair of partitions $\pi_L \in \mcal{P}(V_{L_A})$ and $\pi_R \in \mcal{P}(V_{R_A})$, we define the class of partitions
\[
\mcal{P}_\rho(\pi_L, \pi_R) = \{\pi \in \mcal{P}(V_{L_A \sharp R_A}) : \pi|_{V_{L_A}} = \pi_L \text{ and } \pi|_{V_{R_A}} = \pi_R\}.
\]
We recall the interpretation of $\mcal{P}_\rho(\pi_L, \pi_R)$ from Lemma \ref{lem:prune_tec}: $\mcal{P}_\rho(\pi_L, \pi_R)$ consists of the partitions $\pi \in \mcal{P}(V_{L_A \sharp R_A})$ obtained from $(\pi_L, \pi_R)$ by either keeping a block $V \in \pi_L \cup \pi_R$ (so $V \in \pi$) or merging it with at most a single block from $V'$ from the other side (so $V \cup V' \in \pi$). Of course, the block in $\pi_L$ containing $\op{input}(l_A) = \op{output}(l_A)$ and the block in $\pi_R$ containing $\op{input}(r_A) = \op{output}(r_A)$ are necessarily merged. As before, we write $\pi_\rho(\pi_L, \pi_R)$ for the minimal element in $\mcal{P}_\rho(\pi_L, \pi_R)$ for the reversed refinement order. By construction,
\[
\bigsqcup_{(\pi_L, \pi_R) \in \mcal{P}(V_{L_A}) \times \mcal{P}(V_{R_A})} \mcal{P}_\rho(\pi_L, \pi_R) = \mcal{P}(V_{L_A} \sharp V_{R_A}).
\]
Moreover, for any partition $\pi_B \in \mcal{P}(V_{L_A} \sharp V_{R_A(B)})$, there exists a unique partition $\pi \in \mcal{P}(V_{L_A} \sharp V_{R_A})$ such that
\[
(L_A \sharp R_A(B))^{\pi_B} = (L_A \sharp R_A)^{\pi}.
\]
Indeed, one can construct $\pi$ from $\pi_B$ by simply expanding the amalgamated vertex $v_i \sim v_j$ for $i, j \in B$ in $\pi_B$ into the vertices $v_i \sim v_j$ for $i, j \in A \subset B$ and $v_k$ for $k \in B \setminus A$. For a partition $\pi_R \in \mcal{P}(V_{R_A})$, we then define
\[
B_{\pi_R} = \{i : v_i \overset{\pi_R}{\sim} v_0\} \cup A\subset [n].
\]
This allows us to rewrite \eqref{eq:left_right} as
\begin{align*}
  \psi(t_B t') &= \sum_{\pi_B \in \mcal{P}(V_{L_A \sharp R_A(B)})} \tau_\varphi^0[(L_A \sharp R_A(B))^{\pi_B}] \\
               &= \sum_{\pi_L \in \mcal{P}(V_{L_A})} \sum_{\substack{\pi_R \in \mcal{P}(V_{R_A}) \\ \text{s.t. } B \subset B_{\pi_R}}} \sum_{\pi \in \mcal{P}_\rho(\pi_L, \pi_R)} \tau_\varphi^0[(L_A \sharp R_A)^\pi],
\end{align*}
in which case\small
\begin{align}
\psi(p_A(t)t') &= \sum_{\substack{B \subset [n] \\ \text{s.t. } A \subset B}} (-1)^{\#(B\setminus A)} \sum_{\pi_L \in \mcal{P}(V_{L_A})} \sum_{\substack{\pi_R \in \mcal{P}(V_{R_A}) \\ \text{s.t. } B \subset B_{\pi_R}}} \sum_{\pi \in \mcal{P}_\rho(\pi_L, \pi_R)} \tau_\varphi^0[(L_A \sharp R_A)^\pi] \notag \\
&= \sum_{\pi_L \in \mcal{P}(V_{L_A})} \sum_{\pi_R \in \mcal{P}(V_{R_A})} \sum_{\pi \in \mcal{P}_\rho(\pi_L, \pi_R)} \Big(\sum_{\substack{B \subset [n] \text{ s.t.} \\ A \subset B \subset B_{\pi_R}}} (-1)^{\#(B\setminus A)}\Big)\tau_\varphi^0[(L_A \sharp R_A)^\pi] \notag \\
&= \sum_{\pi_L \in \mcal{P}(V_{L_A})} \sum_{\substack{\pi_R \in \mcal{P}(V_{R_A})\\ \text{s.t. } B_{\pi_R} = A}}  \sum_{\pi \in \mcal{P}_\rho(\pi_L, \pi_R)} \tau_\varphi^0[(L_A \sharp R_A)^\pi] \label{eq:left_right_sum}
\end{align}\normalsize
since
\[
\sum_{\substack{B \subset [n] \text{ s.t.}\\ A \subset B \subset B_{\pi_R}}} (-1)^{\#(B\setminus A)} = 0
\]
unless $A = B_{\pi_R}$.

Continuing the calculation, suppose that $\pi \neq \pi_\rho(\pi_L, \pi_R) \in \mcal{P}_\rho(\pi_L, \pi_R)$ in \eqref{eq:left_right_sum}. Then $(L_A \sharp R_A)^\pi$ cannot possibly be a cactus. Indeed, since $B_{\pi_R} = A$, the vertices $v_i \neq \rho$ form a 2-connection in $(L_A \sharp R_A)^\pi$ for $i \in [n] \setminus A$ (which is non-empty by assumption) via the edges of
\[
\wtilde{\Delta}(a_{i_k}^{\hat{\intercal}(i_k)} \cdots a_{i_{k+1} - 1}^{\hat{\intercal}(i_{k+1} - 1)}) \subset \wtilde{\Delta}(m_{A, k}) \subset R_A,
\]
where $A = \{i_1 < \cdots < i_{\#(A)}\}$ and $i_k < i < i_{k+1}$. Since $\pi \neq \pi_\rho(\pi_L, \pi_R)$, it must be that $\pi$ identifies a vertex $v$ in
\[
\wtilde{\Delta}(d_i) \subset \wtilde{\Delta}(m_{A, k}) = \wtilde{\Delta}(a_{i_k}^{\hat{\intercal}(i_k)} d_{i_k + 1} \cdots a_{i_{k+1}-2}^{\hat{\intercal}(i_{k+1} - 2)}d_{i_{k+1} - 1} a_{i_{k+1} - 1}^{\hat{\intercal}(i_{k+1} - 1)})
\] 
with a vertex $w \neq \rho$ in $L_A \subset L_A \sharp R_A$ for some $i \in [n] \setminus A$. Of course, $w$ is already connected to $\rho$ in $L_A$, so this identification creates a path from $v_i$ to $\rho$ using only the edges of $L_A$ and $\wtilde{\Delta}(d_i)$, which implies that $v_i$ and $\rho$ form a 3-connection in $(L_A \sharp R_A)^\pi$ (see Figure \ref{fig2.13:three_connection}). The lone contribution in our sum over $\mcal{P}_\rho(\pi_L, \pi_R)$ then comes from the minimal element $\pi_\rho(\pi_L, \pi_R)$:
\begin{align*}
  &\sum_{\pi_L \in \mcal{P}(V_{L_A})} \sum_{\substack{\pi_R \in \mcal{P}(V_{R_A})\\ \text{s.t. } B_{\pi_R} = A}}  \sum_{\pi \in \mcal{P}_\rho(\pi_L, \pi_R)} \tau_\varphi^0[(L_A \sharp R_A)^\pi] \\
  = &\sum_{\pi_L \in \mcal{P}(V_{L_A})} \sum_{\substack{\pi_R \in \mcal{P}(V_{R_A})\\ \text{s.t. } B_{\pi_R} = A}} \tau_\varphi^0[(L_A \sharp R_A)^{\pi_\rho(\pi_L, \pi_R)}] \\
= &\sum_{\pi_L \in \mcal{P}(V_{L_A})} \tau_\varphi^0[L_A^{\pi_L}]\sum_{\substack{\pi_R \in \mcal{P}(V_{R_A})\\ \text{s.t. } B_{\pi_R} = A}} \tau_\varphi^0[R_A^{\pi_R}].
\end{align*}

\begin{figure}
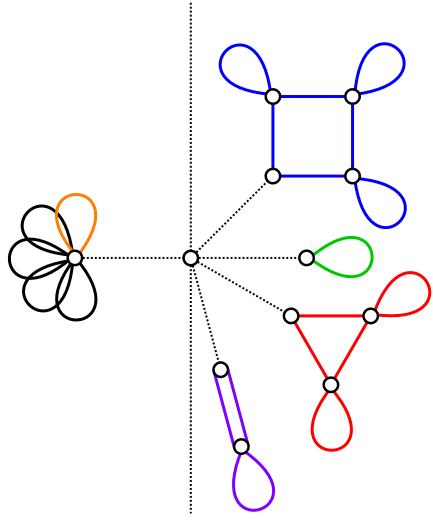
\caption{An example of $L_A \sharp R_A$ for the $A$-segmented factors in Figure \ref{fig2.12:segments}. The vertices connected by dotted lines are all equal to $\rho$. The additional orange loop represents the test graph $\wtilde{\Delta}(t')$. A partition $\pi \neq \pi_\rho(\pi_L, \pi_R)$ makes an identification across the center line separating left and right, leading to a 3-connection.}\label{fig2.13:three_connection}
\end{figure}

By definition,
\[
\sum_{\pi_L \in \mcal{P}(V_{L_A})} \tau_\varphi^0[L_A^{\pi_L}] = \psi(l_A) = \psi((d_0\prod_{i \in A} d_i) t'), 
\]
whereas
\begin{align*}
\sum_{\substack{\pi_R \in \mcal{P}(V_{R_A})\\ \text{s.t. } B_{\pi_R} = A}} \tau_\varphi^0[R_A^{\pi_R}] &= \sum_{\substack{B \subset [n] \\ \text{s.t. } A \subset B}} \Big((-1)^{\#(B\setminus A)} \sum_{\substack{\pi_R \in \mcal{P}(V_{R_A})\\ \text{s.t. } B \subset B_{\pi_R}}} \tau_\varphi^0[R_A^{\pi_R}]\Big) \\
&= \sum_{\substack{B \subset [n] \\ \text{s.t. } A \subset B}} (-1)^{\#(B\setminus A)} \psi(r_A(B)) \\
&= \sum_{\substack{B \subset [n] \\ \text{s.t. } A \subset B}} (-1)^{\#(B\setminus A)} \psi\Big(\prod_{i \in B \setminus A} d_i \prod_{k=0}^{\#(B)} \Delta(m_{B, k})\Big) = \psi(m_A).
\end{align*}
We conclude that 
\[
\psi(p_A(t)t') = \psi(m_A)\psi((d_0\prod_{i \in A} d_i) t'),
\]
as was to be shown.
\end{proof}

Informally, Lemma \ref{lem:pruning} says that the flower $t = \Delta(d_0a_0^{\hat{\intercal}(0)} \cdots d_na_n^{\hat{\intercal}(n)})$ is equivalent mod $\psi$ to a polynomial in the petals $d_i$. In fact, the result still holds even if $\op{input}(t) = \op{output}(t)$ is not located in the cycle $\Delta(a_0^{\hat{\intercal}(0)} \cdots a_n^{\hat{\intercal}(n)}) \subset t$. Indeed, suppose that $t \in \Delta(\mcal{B})$ has a simple cycle $C \subset t$. Without loss of generality, we may assume that $t$ is a quasi-cactus (Corollary \ref{cor:quasi_cactus}). For any vertex $v \in C$, we define $t_v$ to be the graph monomial obtained from $t$ by changing both the input and the output to $v$. Since $t$ is a quasi-cactus, we can write $t_v = \Delta(d_0a_0^{\hat{\intercal}(0)} \cdots d_na_n^{\hat{\intercal}(n)})$ as a flower. Moreover, by construction,
\[
\wtilde{\Delta}(t) = \wtilde{\Delta}(t_v) \in \mcal{T}\langle\mcal{A}\rangle, \qquad \forall v \in C.
\]
More specifically, we choose the unique vertex $v \in C$ such that the petal $d_0 \subset t_v$ rooted at $v$ contains the original input/output of $t$. In the notation of Lemma \ref{lem:pruning}, the vertex $v$ now becomes $v_0$. We can then apply our cycle pruning algorithm to obtain a graph polynomial
\[
p(t_v): = \sum_{A \subset [n]} \Big(\psi(m_A)(d_0\prod_{i \in A} d_i)\Big) \equiv t_v \text{ (mod $\psi$)}.
\]

To translate this back to our original graph monomial $t$, let $\hat{d}_A$ be the unique graph monomial such that
\[
\wtilde{\Delta}(\hat{d}_A) = \wtilde{\Delta}\Big(d_0\prod_{i \in A} d_i\Big)
\]
with $\op{input}(\hat{d}_A) = \op{output}(\hat{d}_A) = \op{input}(t) = \op{output}(t) \in d_0$. We claim that
\begin{equation}\label{eq:relocated_pruning}
p(t): = \sum_{A \subset [n]} \Big(\psi(m_A)\hat{d}_A\Big) \equiv t \text{ (mod $\psi$)}.
\end{equation}
To see this, let $t' \in \mcal{B}$ be a graph monomial. Note that the construction of the polynomial $p(t_v)$ leaves the initial petal $d_0$ intact throughout. In particular, any modification to this petal does not affect the coefficients $\psi(m_A)$, and so we only need to account for the change in the terms $d_0 \prod_{i \in A} d_i$. This implies that
\begin{align*}
\psi(tt') = \tau_\varphi\big[\wtilde{\Delta}(tt')\big] = \tau_\varphi\big[\wtilde{\Delta}(t_v)\sharp\wtilde{\Delta}(t')\big] &= \psi(t_v \sharp \wtilde{\Delta}(t')) \\
&= \sum_{A \subset [n]} \psi(m_A) \psi\Big((d_0 \sharp \wtilde{\Delta}(t')) \prod_{i \in A} d_i\Big) \\
&= \sum_{A \subset [n]} \psi(m_A) \tau_\varphi\Big[\wtilde{\Delta}(d_0\prod_{i \in A} d_i) \sharp \wtilde{\Delta}(t') \Big] \\ 
&= \sum_{A \subset [n]} \psi(m_A) \tau_\varphi\Big[\wtilde{\Delta}(\hat{d}_A) \sharp \wtilde{\Delta}(t') \Big] \\ 
&= \sum_{A \subset [n]} \psi(m_A) \psi(\hat{d}_A t') = \psi(p(t) t'),
\end{align*}
where, in every case, $\sharp$ denotes the appropriate object (test graph or graph monomial) obtained by identifying the vertices $\op{input}(t) = \op{output}(t)$ (seen as vertices of $t_v$, $d_0$, or $\hat{d}_A$) and $\op{input}(t') = \op{output}(t') \in \wtilde{\Delta}(t')$. The equivalence \eqref{eq:relocated_pruning} now follows.

While the formal details of the proof are quite involved, the result has a simple interpretation. Originally, we thought of every diagonal element $d_i$ as a petal of the flower $\Delta(d_0a_0^{\hat{\intercal}(0)} \cdots d_na_n^{\hat{\intercal}(n)})$; however, this neglects the fact that $d_0$ plays a special role in the construction. Instead, we should think of $d_0$ as the \emph{stem} of the flower. Lemma \ref{lem:pruning} tells us how to prune the flower before reattaching it to the stem. If $\op{input}(t) = \op{output}(t)$ is not located in the cycle, then we simply need to orient ourselves properly to apply the algorithm. So, we designate the stem according to the location of the distinguished vertex, in which case everything goes through as before. Iterating the algorithm, we can gradually remove every cycle of $t \in \Delta(\mcal{B})$. Of course, the diagonality assumption greatly simplifies the analysis, but we can always reduce to this case. Indeed, recall that $\mcal{B} \equiv \mcal{A} * \mcal{A}^\intercal * \Delta(\mcal{B}) \text{ (mod $\psi$})$ via the conditional expectation $\mathscr{E}$. In the notation of Lemma \ref{lem:conditional_expectation}, if $t$ is a graph monomial with block-cut tree factorization $t = d_n m_{n-1} d_{n-1} \cdots m_1 d_1$, then
\[
\mathscr{E}(t) = d_n \nabla(m_{n-1}) d_{n-1} \cdots \nabla(m_1) d_1;
\]
however, note that if $\nabla(m_i) \in \mcal{A} \cup \mcal{A}^\intercal$, then $\nabla(m_i)$ is a cut-edge. So, it must be that every cycle in $\mathscr{E}(t)$ belongs to a factor $\nabla(m_i) \in \Delta(\mcal{B})$ or $d_i \in \Delta(\mcal{B})$. Passing to a quasi-cactus equivalent, we can then apply our cycle pruning algorithm to each such factor. Moreover, the quasi-cactus property is preserved by our algorithm (remove the flower, attach petals); so, even though passing to a quasi-cactus equivalent may create more cycles, this is a one-time cost that then allows us to iterate our algorithm to eventually prune \emph{every} cycle (recall that Corollary \ref{cor:prune_tec} already takes care of loops). Theorem \ref{thm:tree_reduction} now follows, and so too does its extension to $\mcal{B}$.
\begin{thm}\label{thm:tree}
For any $t \in \mcal{B}$, there exists a graph polynomial $\mbf{T}(t)$ of trees such that
\[
\mbf{T}(t) \equiv t \text{ \emph{(mod $\psi$)}}.
\]
\end{thm}

\subsection{Gaussianity of the diagonal algebra}\label{sec:gaussianity}

It remains to characterize the distribution of $\Delta_{\op{tree}}(\mcal{B})$. For concreteness, we choose a special class of representatives for our graph monomials. In particular, we consider the subset $\mcal{D}$ of diagonal graph monomials $t = (G, \vin, \vout, \gamma) \in \C\mcal{G}\langle\mcal{A}\rangle$ such that
\begin{enumerate}[label=(A\arabic*)]
\item \label{D_tree} $G$ is a tree; 
\item \label{D_root} $\deg(\vin) = 1$;
\item \label{D_constant} No edge is labeled by a constant;
\item \label{D_degree} No vertex $v$ has both indegree $\deg^+(v) = \#(e \in E: \target(e) = v)$ and outdegree $\deg^-(v) = \#(e \in E: \source(e) = v)$ equal to one. 
\end{enumerate}
Let $[\mcal{D}]$ denote the image of $\mcal{D}$ in $\mcal{G}(\mcal{A}) = \C\mcal{G}\langle\mcal{A}\rangle/\mcal{I}$. The unital algebra generated by $[\mcal{D}]$ is clearly equal to $\Delta_{\op{tree}}(\mcal{B})$. The well-definedness of the trace then allows us to restrict our attention to the distribution of $\mcal{D}$.

For a graph monomial $t \in \mcal{D}$, we define
\[
  Q(t) = \sum_{\pi \in \mcal{P}(V)} \mu(0_V, \pi)t^\pi.
\]
The operator $Q$ captures the notion of a graph with no self-identifications by employing the inclusion-exclusion principle in the context of graph operations. Indeed,
\begin{align*}
  \psi(Q(t)) &= \sum_{\pi \in \mcal{P}(V)} \mu(0_V, \pi)\psi(t^\pi) \\
             &= \sum_{\pi \in \mcal{P}(V)} \mu(0_V, \pi)\tau_\varphi[\wtilde{\Delta}(t^\pi)] \\
             &= \sum_{\pi \in \mcal{P}(V)} \mu(0_V, \pi)\tau_\varphi[\wtilde{\Delta}(t)^\pi] = \tau_\varphi^0[\wtilde{\Delta}(t)],
\end{align*}
where the last equality follows from the definition of the injective traffic state \eqref{eq:injective_mobius}. Moreover, since $t \neq \underset{\txio}{\cdot}$ is a non-trivial tree, $Q(t)$ is centered:
\[
  \psi(Q(t)) = \tau_\varphi^0[\wtilde{\Delta}(t)] = 0.
\]
The operator $Q$ further satisfies an abstract Wick-type formula.
\begin{lemma}\label{lem:wick_formula}
Let $t_1, \ldots, t_n \in \mcal{D}$. Then
\[
  \psi(Q(t_1)\cdots Q(t_n)) = \sum_{P \in \mcal{P}_2(n)} \prod_{(i, j) \in P} \psi(Q(t_i)Q(t_j)),
\]
where $\mcal{P}_2(n)$ denotes the set of pair partitions of $[n]$. In particular, if $n$ is odd, then $\mcal{P}_2(n) = \emptyset$ and
\[
  \psi(Q(t_1)\cdots Q(t_n)) = 0.
\]
\end{lemma}
\begin{proof}
We specify the graph monomials $t_i = (V_i, E_i, \vin^{(i)}, \vout^{(i)}, \gamma_i)$. For notational convenience, we define $\rho_i = \vin^{(i)} = \vout^{(i)}$. We can then write the trace of our $Q$-product as
\begin{align*}
\psi(Q(t_1) \cdots Q(t_n)) &= \sum_{(\pi_1, \ldots, \pi_n) \in \bigtimes_{i \in [n]} \mcal{P}(V_i)} \Big(\prod_{i \in [n]} \mu_{V_i}(0_{V_i}, \pi_i)\Big)\psi(t_1^{\pi_1} \cdots t_n^{\pi_n}) \\
&= \sum_{(\pi_1, \ldots, \pi_n) \in \bigtimes_{i \in [n]} \mcal{P}(V_i)} \Big(\prod_{i \in [n]} \mu_{V_i}(0_{V_i}, \pi_i)\Big)\tau_\varphi[\wtilde{\Delta}(t_1^{\pi_1} \cdots t_n^{\pi_n})].
\end{align*}
We define
\begin{align*}
  T &= \wtilde{\Delta}(t_1 \cdots t_n) = (V, E); \\
  T_{(\pi_1, \ldots, \pi_n)} &= \wtilde{\Delta}(t_1^{\pi_1} \cdots t_n^{\pi_n}) = (V_{(\pi_1, \ldots, \pi_n)}, E_{(\pi_1, \ldots, \pi_n)}),
\end{align*}
which allows us to write
\[
\tau_\varphi[\wtilde{\Delta}(t_1^{\pi_1} \cdots t_n^{\pi_n})] = \sum_{\sigma \in \mcal{P}(V_{(\pi_1, \ldots, \pi_n)})} \tau_\varphi^0[T_{(\pi_1, \ldots, \pi_n)}^\sigma].
\]
For any $n$-tuple of partitions $(\pi_1, \ldots, \pi_n) \in \bigtimes_{i \in [n]} \mcal{P}(V_i)$, we define the class of partitions
\[
\mcal{P}_\rho(\pi_1, \ldots, \pi_n) = \{\pi \in \mcal{P}(V): \pi|_{V_i} = \pi_i \text{ for each } i \in [n]\}.
\]
This construction generalizes our earlier work with $2$-tuples: here, $\mcal{P}_\rho(\pi_1, \ldots, \pi_n)$ consists of the partitions $\pi \in \mcal{P}(V)$ obtained from $(\pi_1, \ldots, \pi_n)$ by either keeping a block $B_{i(1)} \in \pi_{i(1)}$ (so $B_{i(1)} \in \pi)$ or merging it with at most a single block $B_{i(j)}$ from each of the other partitions $\pi_{i(j)}$ (so $\cup_{j=1}^k B_{i(j)} \in \pi$ for some $k \leq n$). As before, the blocks $B_{\rho_i} \in \pi_i$ containing the roots $\rho_i$ are necessarily merged. We also have the equality
\[
\bigsqcup_{(\pi_1, \ldots, \pi_n) \in \bigtimes_{i \in [n]} \mcal{P}(V_i)} \mcal{P}_\rho(\pi_1, \ldots, \pi_n) = \mcal{P}(V).
\]
Moreover, for every partition $\sigma \in \mcal{P}(V_{(\pi_1, \ldots, \pi_n)})$, there exist a unique partition $\pi_\sigma \in \mcal{P}(V)$ such that
\[
T_{(\pi_1, \ldots, \pi_n)}^\sigma = T^{\pi_\sigma}.
\]
Indeed, one can construct $\pi_\sigma$ from $\sigma$ by simply expanding the amalgamated vertices $v_i \overset{\pi_i}{\sim} w_i$ in $\sigma$ into the vertices $v_i$ and $w_i$ in $\pi_\sigma$ for each $i \in [n]$. This observation allows us to rewrite the trace of our $Q$-product as 
\begin{align}
  \notag &\sum_{(\pi_1, \ldots, \pi_n) \in \bigtimes_{i \in [n]} \mcal{P}(V_i)} \Big(\prod_{i \in [n]} \mu_{V_i}(0_{V_i}, \pi_i) \sum_{\sigma \in \mcal{P}(V_{(\pi_1, \ldots, \pi_n)})} \tau_\varphi^0[T^{\pi_\sigma}]\Big) \\
  \notag = &\sum_{(\pi_1, \ldots, \pi_n) \in \bigtimes_{i \in [n]} \mcal{P}(V_i)} \Big(\prod_{i \in [n]} \mu_{V_i}(0_{V_i}, \pi_i) \sum_{\substack{\pi \in \mcal{P}(V) \text{ s.t.} \\ \pi_i \leq \pi|_{V_i} \text{ for each } i \in [n]}} \tau_\varphi^0[T^\pi]\Big) \\
  \notag = &\sum_{\pi \in \mcal{P}(V)} \Big(\sum_{\substack{(\pi_1, \ldots, \pi_n) \in \bigtimes_{i \in [n]} \mcal{P}(V_i) \\ \text{s.t. } \pi_i \leq \pi|_{V_i} \text{ for each } i \in [n]}} \prod_{i \in [n]} \mu_{V_i}(0_{V_i}, \pi_i)\Big)\tau_\varphi^0[T^\pi] \\
  \notag = &\sum_{\substack{\pi \in \mcal{P}(V) \text{ s.t.} \\ \pi|_{V_i} = 0_{V_i} \text{ for each }  i \in [n]}} \tau_\varphi^0[T^\pi] \\
         = &\sum_{\pi \in \mcal{P}_\rho(0_{V_1}, \ldots, 0_{V_n})} \tau_\varphi^0[T^\pi] \label{eq:Q_product}
\end{align}
since
\[
\sum_{\substack{\pi_i \in \mcal{P}(V_i) \\ \text{s.t. } \pi_i \leq \pi|_{V_i}}} \mu_{V_i}(0_{V_i}, \pi_i) = 0  
\]
unless $\pi|_{V_i} = 0_{V_i}$.

We say that $\pi \in \mcal{P}_\rho(0_{V_1}, \ldots, 0_{V_n})$ \emph{associates} the trees $t_i$ and $t_j$ if there exist vertices $v_i \neq \rho_i \in t_i$ and $v_j \neq \rho_j \in t_j$ such that $v_i \overset{\pi}{\sim} v_j$. We will show that a contributing partition only has pair associations. First, suppose that a tree $t_i$ is not $\pi$-associated to any other tree $t_j$. If we write $T_i = \wtilde{\Delta}(\prod_{j \neq i} t_j)$, then
\[
\tau_\varphi^0[T^\pi] = \tau_\varphi^0[\wtilde{\Delta}(t_i)]\tau_\varphi^0[T_i^{\pi|_{T_i}}] = 0
\]
since $\pi|_{V_i} = 0_{V_i}$ and $\wtilde{\Delta}(t_i)$ is only attached to $T = T_i \sharp \wtilde{\Delta}(t_i)$ at the root $\rho = \rho_i$.

Next, suppose that $t_i$ is associated to two different trees $t_j$ and $t_k$. Let $v_i, w_i \in t_i$, $v_j \in t_j$, and $w_k \in t_k$ be vertices that witness this association:
\[
v_i \overset{\pi}{\sim} v_j \quad \text{and} \quad w_i \overset{\pi}{\sim} w_k.
\]
Since $\op{input}(t_i) = \op{output}(t_i) = \rho_i$, we can think of $t_i$ as a rooted tree. Accordingly, let $v = v_i \wedge w_i$ be the closest ancestor of $v_i$ and $w_i$. Note that $v \neq \rho_i$ since $\deg(\rho_i) = 1$. Moreover, $v \overset{\pi}{\not\sim} \rho \in T$ since $\pi|_{V_i} = 0_{V_i}$. But then $v \neq \rho \in T^\pi$ form a 3-connection in $T^\pi$: one path comes from going down the edges from $v$ to $\rho_i$ in $t_i$; a second path comes from going up the edges from $v$ to $v_i$ in $t_i$ and then down the edges from $v_j$ to $\rho_j$ in $t_j$; and a third path comes from going up the edges from $v$ to $w_i$ in $t_i$ and then down the edges from $w_k$ to $\rho_k$ in $t_k$. See Figure \ref{fig2.14:ancestor} for an illustration. Once again, this implies that
\[
\tau_\varphi^0[T^\pi] = 0.
\]

\begin{figure}
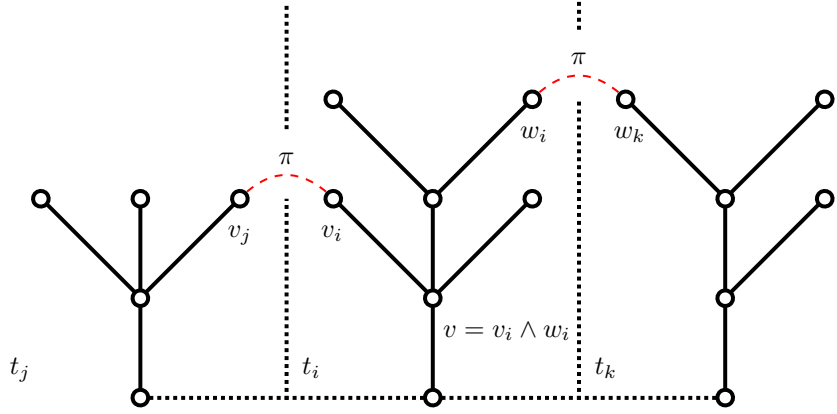
\caption{An example of a partition $\pi$ that associates $t_i$ to two different trees $t_j$ and $t_k$ and the corresponding closest ancestor $v = v_i \wedge w_i$. As before, the vertices connected by black dotted lines are all equal to $\rho$.}\label{fig2.14:ancestor}
\end{figure}

We can then restrict the sum in \eqref{eq:Q_product} to partitions $\mcal{P}_\rho^{(2)}(0_{V_1}, \ldots, 0_{V_n})$ that only make pair associations between the trees $(t_i)_{i = 1}^n$. For such a partition $\pi$, we write $P_\pi \in \mcal{P}_2(n)$ for the pair partition of $[n]$ induced by the pair associations of $\pi$: if $t_i$ and $t_j$ are $\pi$-associated, then $(i, j) \in P_\pi$. Moreover, note that each pair of $\pi$-associated trees is only attached to another pair of $\pi$-associated trees at the vertex $\rho$ in $T^\pi$. Altogether, this implies that
\begin{align*}
\psi(Q(t_1) \cdots Q(t_n)) &= \sum_{\pi \in \mcal{P}_\rho^{(2)}(0_{V_1}, \ldots, 0_{V_n})} \tau_\varphi^0[T^\pi] \\
&= \sum_{\pi \in \mcal{P}_\rho^{(2)}(0_{V_1}, \ldots, 0_{V_n})} \prod_{(i, j) \in P_\pi} \tau_\varphi^0[\wtilde{\Delta}(t_it_j)^{\pi|_{\wtilde{\Delta}(t_i t_j)}}] \\
&= \sum_{P \in \mcal{P}_2(n)} \prod_{(i, j) \in P} \Big(\sum_{\pi_{(i, j)} \in \mcal{P}_\rho^{(2)}(0_{V_i}, 0_{V_j})} \tau_\varphi^0[\wtilde{\Delta}(t_i t_j)^{\pi_{(i, j)}}]\Big) \\
&= \sum_{P \in \mcal{P}_2(n)} \prod_{(i, j) \in P} \psi(Q(t_i)Q(t_j)),
\end{align*}
where the last equality follows from our formula \eqref{eq:Q_product} for the trace of a $Q$-product in the case of $n=2$.
\end{proof}

If $\psi$ is positive, Theorem \ref{lem:wick_formula} proves that $(Q(t_i))_{i = 1}^n$ is jointly complex Gaussian. Indeed, to determine the parameters, we define the real and imaginary parts of our $Q$-variables
\begin{align*}
  X_i = \Re(Q(t_i)) = \frac{Q(t_i) + Q(t_i)^*}{2}; \\
  Y_i = \Im(Q(t_i)) = \frac{Q(t_i) - Q(t_i)^*}{2i}.
\end{align*}
Let
\begin{align*}
  r_i = \psi(Q(t_i)^*Q(t_i)) \geq 0; \\
  a_i + b_i\sqrt{-1} = \psi(Q(t_i)Q(t_i)) \in \C,
\end{align*}
where $a_i, b_i \in \R$ and $a_i^2 + b_i^2 \leq r_i^2$ by the Cauchy-Schwarz inequality. Since $Q(t_i)^* = Q(t_i^*)$, our Wick formula implies that $(X_i, Y_i)_{i \in [n]}$ is jointly real Gaussian with
\begin{align*}
  X_i &\deq \mcal{N}_\R\Big(0, \frac{r_i + a_i}{2}\Big); \\
  Y_i &\deq \mcal{N}_\R\Big(0, \frac{r_i - a_i}{2}\Big); \\
  \E[X_iY_i] &= \frac{b_i}{2}; \\
  \E[X_iX_j] &= \psi(\Re(Q(t_i))\Re(Q(t_j))); \\
  \E[X_iY_j] &= \psi(\Re(Q(t_i))\Im(Q(t_j))); \\
  \E[Y_iY_j] &= \psi(\Im(Q(t_i))\Im(Q(t_j))).
\end{align*}
For notational convenience, let $X_{n+i} = Y_i$ for $i \in [n]$. The non-negativity of the covariance matrix 
\[
(\Gamma(i, j))_{i, j \in [2n]} = (\E[X_iX_j])_{i, j \in [2n]}
\]
follows from the positivity of the trace $\psi$. Indeed, if $\hat{\mbf{x}} = (x_1, \ldots, x_{2n}) \in \R^{2n}$, then\small
\[
\langle \Gamma\hat{\mbf{x}}, \hat{\mbf{x}}\rangle = \sum_{i, j \in [2n]} x_ix_j\E[X_iX_j] = \psi\Big(\Big(\sum_{i \in [n]} x_i \Re(Q(t_i)) + \sum_{i \in [n]} x_{n + i} \Im(Q(t_i))\Big)^2\Big) \geq 0
\]\normalsize
since
\[
\Big(\sum_{i \in [n]} x_i \Re(Q(t_i)) + \sum_{i \in [n]} x_{n + i} \Im(Q(t_i))\Big)^* = \sum_{i \in [n]} x_i \Re(Q(t_i)) + \sum_{i \in [n]} x_{n + i} \Im(Q(t_i))
\]
is self-adjoint. It follows that $(Q(t_i))_{i = 1}^n = (\Re(Q(t_i)) + \sqrt{-1}\Im(Q(t_i)))_{i = 1}^n$ is jointly complex Gaussian.

Rearranging our formula for the $Q$-variable, we have the identity
\begin{equation}\label{eq:Q_identity}
  t = Q(t) - \sum_{\substack{\pi \in \mcal{P}(V) \\ \text{s.t. } \pi \neq 0_V}} \mu(0_V, \pi)t^\pi.
\end{equation}
Note that if $\pi \neq 0_V$, then $t^\pi \not\in \mcal{D}$ since the underlying graph will no longer be a tree ($t^\pi$ has the same number of edges as $t$ but strictly fewer vertices) and the root might no longer be a leaf. However, using our cycle pruning algorithm, $t^\pi$ is equivalent to a polynomial in graph monomials in $\mcal{D}$, each of which has strictly fewer vertices than $t^\pi$. We can then iterate the identity \eqref{eq:Q_identity} on each of these graph monomials. Note that the procedure will eventually stop since the number of vertices is strictly decreasing and $Q\Big(\ \underset{\txio}{\cdot}\ \Big) = \underset{\txio}{\cdot}$. This shows that every graph monomial in $\mcal{D}$ is equivalent to a polynomial in $Q$-variables from $\mcal{D}$, and so we think of the diagonal algebra $\Delta(\mcal{B})$ as being generated by Gaussian random variables.

\begin{thm}\label{eq:gaussianity}
Let $[Q(\mcal{D})]$ denote the image of $Q(\mcal{D})$ in $\mcal{G}(\mcal{A}) = \C\mcal{G}\langle\mcal{A}\rangle/\mcal{I}$. The unital algebra generated by $[Q(\mcal{D})]$ is equal to $\Delta(\mcal{B})$ up to degeneracy. In other words, if $\Delta_Q(\mcal{B})$ is the unital algebra generated by $[Q(\mcal{D})]$, then
\[
  \Delta_Q(\mcal{B}) \equiv \Delta(\mcal{B}) \text{ (mod $\psi$)}.
\]  
\end{thm}
\begin{proof}
Let $t \in \mcal{D}$. Iterating the $Q$-identity \eqref{eq:Q_identity} as described in the paragraph above shows that there exists a $p \in \Delta_Q(\mcal{B})$ such that $p \equiv t \text{ (mod $\psi$)}$. Since $\mcal{D}$ generates $\Delta_{\op{tree}}(\mcal{B})$ as an algebra, the result then follows from our earlier result $\Delta_{\op{tree}}(\mcal{B}) \equiv \Delta(\mcal{B}) \text{ (mod $\psi$)}$.
\end{proof}

In general, the covariance of two $Q$-variables $Q(t)$ and $Q(t')$ must be explicitly computed to determine the dependence structure, which can quite laborious. Fortunately, our $Q$ operator ``diagonalizes'' the resulting Gaussian process in the sense that $\psi(Q(t)Q(t')) = 0$ unless $t, t' \in \mcal{D}$ are anti-isomorphic as rooted digraphs. We will prove this result through a short sequence of intermediate steps.

To begin, let $T = \wtilde{\Delta}(tt')$. Recall that
\[
\psi(Q(t)Q(t')) = \sum_{\pi \in \mcal{P}_\rho(0_V, 0_{V'})} \tau_\varphi^0\big[T^\pi\big].
\]
Indeed, by construction, the operator $Q$ prevents self-identifications within each tree. To emphasize this point, we say that a partition $\pi \in \mcal{P}_\rho(0_V, 0_{V'})$ keeps each side \emph{solid} (for convenience, we also call such a partition \emph{solid}).  Note that the solidity of $\pi$ implies that each vertex $v \in V$ is either left untouched or identified with a single vertex $v' \in V'$. 

Now, suppose that there exists a partition $\pi \in \mcal{P}_\rho(0_V, 0_{V'})$ such that $T^\pi$ is an oriented cactus. If $\psi(Q(t)Q(t')) \neq 0$, then such a partition must exist. We claim that $\pi$ induces a pair partition of the vertices $V \sqcup V'$ that defines an anti-isomorphism between $t$ and $t'$. We start at the level of leaves.
\begin{lemma}\label{lem:leaf_bijection}
Let $\pi \in \mcal{P}_\rho(0_V, 0_{V'})$ be such that $T^\pi$ is a cactus. Then each leaf $v$ of $t$ is identified with leaf $v'$ of $t'$. In particular, the leaves of $t$ and the leaves of $t'$ are in a $\pi$-dependent bijection.
\end{lemma}
\begin{proof}
Let $v \neq \rho$ be a non-root leaf of $t$. Since a cactus has no leaves, we know that $\pi$ necessarily identifies $v$ with a vertex $v' \in V'\setminus\{\rho'\}$. Indeed, $\rho = \rho' \in T$, so the solidity of $\pi$ ensures that $v' \in V'\setminus\{\rho'\}$. Suppose, for a contradiction, that $v'$ is not a leaf of $t'$. In that case, there exists a descendant leaf $w'$ of $v'$ in $t'$. The same reasoning implies that $\pi$ identifies $w'$ with a vertex $w \in V\setminus\{\rho, v\}$. The closest ancestors $v \wedge w$ and $v' \wedge w' = v'$ then form a 3-connection in $T^\pi$: one path comes from going down the edges from $v \wedge w$ to the root $\rho$ and then up the edges from the root $\rho'$ to $v'$: a second path comes going up the edges from $v \wedge w$ to $v \overset{\pi}{\sim} v'$; and a third path comes from going up the edges from $v \wedge w$ to $w$ and then down the edges from $w'$ to $v'$. See Figure \ref{fig2.15:leaf} for an illustration. Since a cactus has no 3-connections, it must be that $v \wedge w \overset{\pi}{\sim} v'$. Moreover, we know that $v \neq v \wedge w$ since $v \neq w$ and $v$ is a leaf. By assumption, we already have the identification $v \overset{\pi}{\sim} v'$, which then implies $v \overset{\pi}{\sim} v \wedge w$, contradicting the solidity of $\pi$.

In the case of the roots, $\rho$ and $\rho'$ are identified as part of the multiplication $tt'$. So, we can think of each leaf $v$ of $t$ as being identified with a leaf $v'$ of $t'$ via $\pi$ with the convention that $\rho \overset{\pi}{\sim} \rho'$. The partition $\pi$ then induces a bijection $f_\pi: V_{\text{leaf}} \to V_{\text{leaf}}'$.
\end{proof}

\begin{figure}
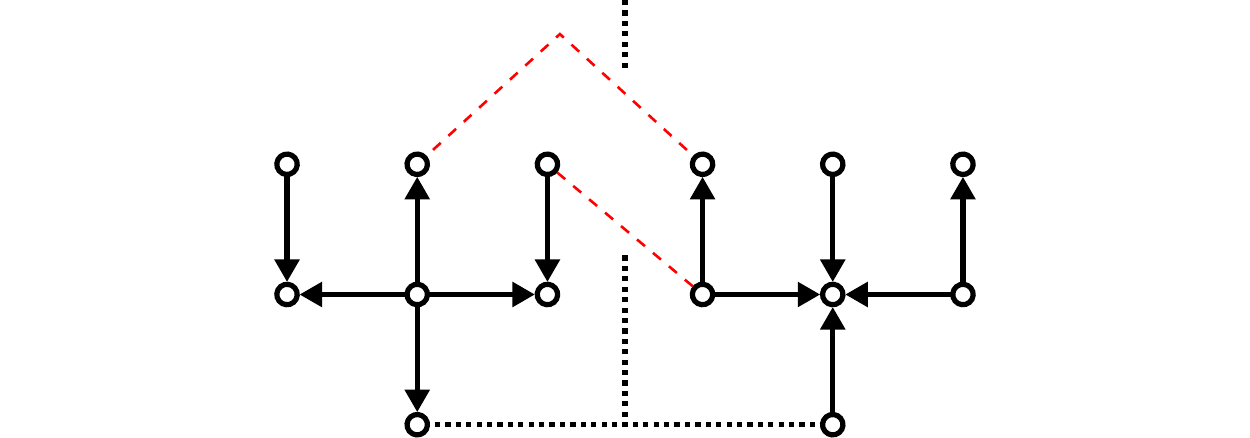
\caption{An example of a forbidden leaf to non-leaf identification leading to a 3-connection between $v \wedge w$ and $v' = v' \wedge w'$.}\label{fig2.15:leaf}
\end{figure}

Next, we extend this bijection to the entire set of vertices.

\begin{lemma}\label{lem:vertex_bijection}
Let $\pi \in \mcal{P}_\rho(0_V, 0_{V'})$ be such that $T^\pi$ is an oriented cactus. Then $\pi$ induces a pair partition of $V \sqcup V'$. In particular, the vertices of $t$ and the vertices of $t'$ are in a $\pi$-dependent bijection.
\end{lemma}
\begin{proof}
We already know that $\pi$ pairs the leaves of $t$ with the leaves of $t'$. Suppose that $v \overset{\pi}{\sim} v'$ and $w \overset{\pi}{\sim} w'$ are two such non-root pairings. Since $\deg(\rho) = \deg(\rho') = 1$, this implies that $v \wedge w \neq \rho$ and $v' \wedge w' \neq \rho'$ form a 3-connection: one path comes from going down the edges from $v \wedge w$ to $\rho$ and then up the edges from $\rho'$ to $v' \wedge w'$; a second path comes from going up the edges from $v \wedge w$ to $v$ and then down the edges from $v'$ to $v' \wedge w'$; and a third path comes from going up the edges from $v \wedge w$ to $w$ and then down the edges from $w'$ to $v' \wedge w'$. Since $T^\pi$ is a cactus, it must be that $v \wedge w \overset{\pi}{\sim} v' \wedge w'$. Note that every vertex $x \in V$ with $\deg(x) \geq 3$ can be written as the closest ancestor $x = v \wedge w$ of two distinct leaves $v \neq w$ in the rooted tree $t$. Such a vertex $x$ must then be identified with the vertex $x' = v' \wedge w' \in V'$. 

It remains to consider the vertices $x \in V$ of degree two. By assumption \ref{D_degree}, it must be that $\deg^+(x) = 2$ (and $\deg^-(x) = 0$) or $\deg^-(x) = 2$ (and $\deg^+(x) = 0$). Since every vertex in an oriented cactus has equal indegree and outdegree, it must be that $x$ is identified with a vertex $x' \in V'$. In fact, $x'$ must also be a vertex of degree two with
\[
  \deg^+(x)\deg^+(x') = \deg^-(x)\deg^-(x') = 0
\]
to balance the indegree and outdegree in $T^\pi$. To see this, note that Lemma \ref{lem:leaf_bijection} already implies that $x'$ cannot be a leaf. Moreover, our work above shows that if $\deg(x') \geq 3$, then $x' = v' \wedge w'$ for a pair of leaves $v' \neq w' \in V'$. In this case, $x' \overset{\pi}{\sim} v \wedge w$ for a pair of leaves $v \neq w \in V$. But then $x \overset{\pi}{\sim} v \wedge w$, where $x \neq v \wedge w$ since $\deg(v \wedge w) \geq 3$, contradicting the solidity of $\pi$. 
\end{proof}

Finally, we show that this bijection reverses the direction of adjacency.

\begin{lemma}\label{lem:Q_adjacency}
Let $\pi \in \mcal{P}_\rho(0_V, 0_{V'})$ be such that $T^\pi$ is an oriented cactus. Then the bijection $f_\pi: V \to V'$ induced by $\pi$ reverses adjacency:
\[
  \overset{v}{\cdot} \rightarrow \overset{w}{\cdot} \quad \iff \quad \overset{f_\pi(v)}{\cdot} \leftarrow \overset{f_\pi(w)}{\cdot}.
\]
\end{lemma}
\begin{proof}
Let $v \in t$ be a leaf of depth $n$. We think of $t$ as a line from the root $\rho$ to $v$ with branches attached to the vertices along this path. To make this precise, we enumerate the vertices on this path $v_0 = \rho, v_1, \ldots, v_n = v$. We also enumerate the sub-trees $t_i \subset t$ attached to the vertices $v_i$ along this path (the aforementioned branches). We think of the $t_i$ as rooted trees with $\op{root}(t_i) = v_i$. In particular, $t_0$ and $t_n$ are the trivial tree consisting of a single vertex. By our earlier work, we know that $\pi$ identifies $v_n$ with a leaf $v'\in t'$, say of depth $m$. We write $v_0' = \rho', v_1', \ldots, v_m' = v'$ and $t_0', \ldots, t_m'$ for the analogous decomposition in $t'$.

Consider a vertex $v_i \neq v_0, v_n$. We know that $v_i$ is identified with a vertex $x' \in V'\setminus\{\rho', v_m'\}$, which implies that $x' \in t_j'$ for some $j \in [m-1]$. Suppose, for a contradiction, that $x' \neq v_j'$. Then $v_i \neq v_j'$ form a 3-connection in $T_{v_i \sim x'}$: one path comes from going down the edges from $v_i$ to $\rho$ in $t$ and then up the edges from $\rho'$ to $v_j'$ in $t'$; a second path comes from going up the edges from $v_i$ to $v_n$ in $t$ and then down the edges from $v_m'$ to $v_j'$ in $t'$; and a third path comes from going down the edges from $v_i \sim x'$ to $v_j'$ in $t_j' \subset t'$. Since $T^\pi$ is an oriented cactus, it must be that $v_i \overset{\pi}{\sim} v_j'$. By assumption, we already have the identification $v_i \overset{\pi}{\sim} x'$. This implies that $v_j' \overset{\pi}{\sim} x'$, contradicting the solidity of $\pi$. We conclude that each vertex $v_i$ along the path from $v_0$ to $v_n$ in $t$ is identified with a vertex $v_j'$ along the path from $v_0'$ to $v_m'$ in $t'$.

By symmetry, it follows that $m = n$. We claim that $v_i \overset{\pi}{\sim} v_i'$ for each $i = 0, \ldots, n$. Indeed, we already know this to be true for $i = 0, n$. Suppose, for a contradiction, that $i_1 := \min\{i \in [n]: v_i \overset{\pi}{\not\sim} v_i'\} \geq 1$. Then
\[
v_{i_1} \overset{\pi}{\sim} v_j' \quad \text{and} \quad v_k \overset{\pi}{\sim} v_{i_1}'
\]
for some $j, k > i_1$ and $v_{i_1} \neq v_{i_1}' \in T^\pi$ form a 3-connection: one path comes from going down the edges from $v_{i_1}$ to $\rho$ in $t$ and then up the edges from $\rho'$ to $v_{i_1}'$ in $t'$; a second path comes from going down the edges from $v_{i_1} \overset{\pi}{\sim} v_j'$ to $v_i'$ in $t'$; and a third path comes from going up the edges from $v_{i_1}$ to $v_k \overset{\pi}{\sim} v_{i_1}'$ in $t$. But then $T^\pi$ cannot possibly be a cactus.

The fact that the induced bijection $f_\pi: V \to V'$ preserves \emph{undirected} adjacency now follows. Indeed, let $v$ and $w$ be adjacent vertices in $V$. Without loss of generality, we may assume that $v$ is the ancestor of $w$. Let $l$ be any descendant leaf of $w$, say of depth $n$. The unique path from the root $\rho$ to $l$ then passes through $v$ and $w$. In the notation above, $v_0 = \rho$, $v_n = l$, and
\[
v_{i - 1} = v \quad \text{and} \quad v_i = w
\]
for some $i \in [n]$. Our work above then shows that $f_\pi(v_i) = v_i'$ is adjacent to $f_\pi(v_{i+1}) = v_{i+1}'$ in $t'$ for each $i = 0, \ldots, n-1$. To see that $f_\pi$ reverses the direction of adjacency, one simply needs to use the fact that every vertex in an oriented cactus has equal indegree and outdegree and proceed inductively from the leaves.

\end{proof}

\begin{eg}\label{eg:gaussian_degrees}
Let $\mcal{A}$ be the unital $*$-algebra generated by a standard semicircular random variable $a$. Then
\begin{align*}
  \op{rDeg}(a) &= \underset{\txio}{\phantom{\overset{1}{\phantom{\ccdot}}}\vvdownarrow \overset{a}{\phantom{\ccdot}}} \ \equiv Q\bigg( \ \underset{\txio}{\phantom{\overset{1}{\phantom{\ccdot}}}\vvdownarrow \overset{a}{\phantom{\ccdot}}} \ \bigg) \text{ (mod $\psi$)}, \\
  \op{cDeg}(a) &= \underset{\txio}{\phantom{\overset{1}{\phantom{\ccdot}}}\vvuparrow \overset{a}{\phantom{\ccdot}}} \ \equiv Q\bigg( \ \underset{\txio}{\phantom{\overset{1}{\phantom{\ccdot}}}\vvuparrow \overset{a}{\phantom{\ccdot}}} \ \bigg) \text{ (mod $\psi$)}
\end{align*}
are jointly complex Gaussian with covariance matrix
\[
  \Gamma = \begin{pmatrix}
    1 & 0\\
    0 & 1
  \end{pmatrix}
\]
and pseudo-covariance matrix
\[
  C = \begin{pmatrix}
    0 & 1\\
    1 & 0
  \end{pmatrix}.
\]
In other words, $(\op{rDeg}(a), \op{cDeg}(a)) \deq (Z, \overline{Z})$ for $Z$ standard complex normal. Indeed, the self-adjointness $a^* = a$ implies that $\op{rDeg}(a)^* = \op{cDeg}(a)$.
\end{eg}

\section{Applications to random multi-matrix models}\label{sec:applications}
\subsection{The cactus-cumulant correspondence}\label{sec:cactus_cumulant}

We draw the reader's attention to the similarities in the formulas relating the trace to the free cumulants \eqref{eq:free_cumulants_mobius}-\eqref{eq:free_cumulants_sum} and those relating the traffic state to the injective traffic state \eqref{eq:injective_mobius}-\eqref{eq:injective_sum}. Since the trace is defined in terms of the traffic state \eqref{eq:trace_traffic_state}, one might hope for a nice formula relating the injective traffic state to the free cumulants. In the case of cactus-type traffics, such a correspondence can be made explicit.

\begin{defn}[Cactus-type]\label{defn:cactus_type}
Let $(\mcal{A}, \tau)$ be an algebraic traffic space. A family of traffics $\mbf{a} = (a_i)_{i \in I}$ is said to be of \emph{cactus-type} if its injective traffic distribution
\[
  \upsilon_{\mbf{a}}^0: \C\mcal{T}\langle\mbf{x}\rangle \to \C, \qquad T \mapsto \tau^0[T(\mbf{a})]
\]
is supported on cactus graphs and multiplicative with respect to the pads:
\[
  \upsilon_{\mbf{a}}^0[T] = \prod_{C \in \op{Pads}(T)} \upsilon_{\mbf{a}}^0[C].
\]
\end{defn}

\begin{eg}\label{eg:cactus_type}
Any family of random variables $(a_i)_{i \in I}$ in a tracial ncps $(\mcal{A}, \varphi)$ is of cactus-type in the UE traffic space $(\mcal{G}(\mcal{A}), \tau_\varphi)$ by construction.
\end{eg}

We recall a simple characterization of the free cumulants.

\begin{prop}[{\cite[Remark 11.19]{NS06}}]\label{prop:free_cumulants}
Let $(a_i)_{i \in I}$ be a family of random variables in a npcs $(\mcal{A}, \varphi)$. Suppose that a family of complex numbers
\[
  (\tilde{\kappa}_\pi[a_{i(1)}, \ldots, a_{i(n)}])_{n \in \N, \pi \in \mcal{NC}(n), i \in \mcal{F}([n], I)}
\]
satisfies:
\begin{enumerate}[label=(\roman*)]
\item (Multiplicativity)
\[
\tilde{\kappa}_\pi[a_{i(1)}, \ldots, a_{i(n)}] = \prod_{B \in \pi} \tilde{\kappa}(B)[a_{i(1)}, \ldots, a_{i(n)}],
\]
where $B = (j_1 < \cdots < j_m)$ is a block and
\[
  \tilde{\kappa}(B)[a_{i(1)}, \ldots, a_{i(n)}] : = \tilde{\kappa}_{1_m}[a_{i(j_1)}, \ldots, a_{i(j_m)}];
\]
\item (M\"{o}bius formula)
\[
  \varphi(a_{i(1)} \cdots a_{i(n)}) = \sum_{\pi \in \mcal{NC}(n)} \tilde{\kappa}_\pi[a_{i(1)}, \ldots, a_{i(n)}]
\]
for any $n \in \N$ and $i \in \mcal{F}([n], I)$.
\end{enumerate}
Then the $\tilde{\kappa}$ are the free cumulants of $(a_i)_{i \in I}$:
\[
\kappa_\pi[a_{i(1)}, \ldots, a_{i(n)}] = \tilde{\kappa}_\pi[a_{i(1)}, \ldots, a_{i(n)}]
\]
for any $n \in \N$, $\pi \in \mcal{NC}(n)$, and $i \in \mcal{F}([n], I)$.
\end{prop}

The correspondence between non-crossing partitions and cactus quotients then allows us to translate between free cumulants and the injective traffic distribution.

\begin{lemma}\label{lem:cactus_cumulant}
Let $(\mcal{A}, \tau)$ be an algebraic traffic space. If $(a_i)_{i \in I} \subset \mcal{A}$ is a family of cactus-type traffics, then
\begin{equation}\label{eq:cycle_cumulant_equality}
  \begin{tikzpicture}[baseline=(current  bounding  box.center), shorten > = 1.5pt]
    \node at (-3.375, 0) {$\kappa_n[a_{i(1)}^{\hat{\intercal}(1)}, \ldots, a_{i(n)}^{\hat{\intercal}(n)}] = \tau^0\Bigg[$};
    \node at (1.875, 0) {$\Bigg]$.};
    \draw[fill=black] (.6, 0) circle (1pt);
    \draw[fill=black] (-.6, 0) circle (1pt);
    \draw[fill=black] (.3, .5196) circle (1pt);
    \draw[fill=black] (.3, -.5196) circle (1pt);
    \draw[fill=black] (-.3, .5196) circle (1pt);
    \draw[fill=black] (-.3, -.5196) circle (1pt);
    \draw[semithick, ->] (.6,0) to node[pos=.625, right] {${\scriptstyle a_{i(n-1)}^{\hat{\intercal}(n-1)}}$} (.3,-.5196);
    \draw[semithick, ->] (.3,-.5196) to node[midway, below] {${\scriptstyle \cdots}$} (-.3,-.5196);
    \draw[semithick, ->] (-.3,-.5196) to node[pos=.125, left] {${\scriptstyle a_{i(3)}^{\hat{\intercal}(3)}}$} (-.6, 0);
    \draw[semithick, ->] (-.6, 0) to node[pos=.75, left] {${\scriptstyle a_{i(2)}^{\hat{\intercal}(2)}}$} (-.3, .5196);
    \draw[semithick, ->] (-.3, .5196) to node[pos=.5, above] {${\scriptstyle a_{i(1)}^{\hat{\intercal}(1)}}$} (.3, .5196);
    \draw[semithick, ->] (.3, .5196) to node[pos=.25, right] {${\scriptstyle a_{i(n)}^{\hat{\intercal}(n)}}$} (.6, 0);
  \end{tikzpicture}
\end{equation}
\end{lemma}
\begin{proof}
By definition,
\[
    \begin{tikzpicture}[shorten > = 1.5pt]
    \node at (-3.25, 0) {$\varphi_\tau(a_{i(1)}^{\hat{\intercal}(1)} \cdots a_{i(n)}^{\hat{\intercal}(n)}) = \tau\Bigg[$};
    \node at (1.875, 0) {$\Bigg]$.};
    \draw[fill=black] (.6, 0) circle (1pt);
    \draw[fill=black] (-.6, 0) circle (1pt);
    \draw[fill=black] (.3, .5196) circle (1pt);
    \draw[fill=black] (.3, -.5196) circle (1pt);
    \draw[fill=black] (-.3, .5196) circle (1pt);
    \draw[fill=black] (-.3, -.5196) circle (1pt);
    \draw[semithick, ->] (.6,0) to node[pos=.625, right] {${\scriptstyle a_{i(n-1)}^{\hat{\intercal}(n-1)}}$} (.3,-.5196);
    \draw[semithick, ->] (.3,-.5196) to node[midway, below] {${\scriptstyle \cdots}$} (-.3,-.5196);
    \draw[semithick, ->] (-.3,-.5196) to node[pos=.125, left] {${\scriptstyle a_{i(3)}^{\hat{\intercal}(3)}}$} (-.6, 0);
    \draw[semithick, ->] (-.6, 0) to node[pos=.75, left] {${\scriptstyle a_{i(2)}^{\hat{\intercal}(2)}}$} (-.3, .5196);
    \draw[semithick, ->] (-.3, .5196) to node[pos=.5, above] {${\scriptstyle a_{i(1)}^{\hat{\intercal}(1)}}$} (.3, .5196);
    \draw[semithick, ->] (.3, .5196) to node[pos=.25, right] {${\scriptstyle a_{i(n)}^{\hat{\intercal}(n)}}$} (.6, 0);
  \end{tikzpicture}
\]
We define $C_n(a_{i(1)}^{\hat{\intercal}(1)}, \ldots, a_{i(n)}^{\hat{\intercal}(n)}) = (V, E, \source, \target, \gamma)$ to be the test graph above, namely, $V = (v_k)_{k=\overline{1}}^{\overline{n}}$, $E = (e_k)_{k = 1}^n$, $\gamma(e_k) = a_{i(k)}$, and
\begin{enumerate}
\item $\source(e_k) = v_{\overline{k+1}}$ and $\target(e_k) = v_{\overline{k}}$ if $\hat{\intercal}(k) = 1$;
\item $\source(e_k) = v_{\overline{k}}$ and $\target(e_k) = v_{\overline{k+1}}$ if $\hat{\intercal}(k) = \intercal$.
\end{enumerate}
Our indexing of the vertices (resp., edges) allows us to think of $\mcal{P}(V) \cong \mcal{P}(\overline{n})$ (resp., $\mcal{P}(E) \cong \mcal{P}(n)$) as convenient. We can then use the interlacing
\[
  \overline{1} < 1 < \cdots < \overline{n} < n
\]
to define the Kreweras complement of a partition $\pi \in \mcal{NC}(V) \cup \mcal{NC}(E)$.
 
If the $(a_i)_{i \in I}$ are of cactus-type, then Proposition \ref{prop:cactus_non_crossing_partition} tells us that
\begin{align*}
  \tau[C_n(a_{i(1)}^{\hat{\intercal}(1)}, \ldots, a_{i(n)}^{\hat{\intercal}(n)})] &=  \sum_{\pi \in \mcal{P}(V)} \tau^0[C_n^\pi(a_{i(1)}^{\hat{\intercal}(1)}, \ldots, a_{i(n)}^{\hat{\intercal}(n)})] \\
                                                                                  &= \sum_{\pi \in \mcal{NC}(V)} \tau^0[C_n^\pi(a_{i(1)}^{\hat{\intercal}(1)}, \ldots, a_{i(n)}^{\hat{\intercal}(n)})] \\
                                                                                  &= \sum_{\pi \in \mcal{NC}(V)} \prod_{B \in K(\pi)} \tau^0[C_{\#(B)}(a_{i(j_1)}^{\hat{\intercal}(j_1)}, \ldots, a_{i(j_{\#(B)})}^{\hat{\intercal}(j_{\#(B)})})],                                  
\end{align*}
where in the product we think of $\pi$ as an element of $\mcal{NC}(\overline{n})$ so that $K(\pi) \in \mcal{NC}(n)$ and $B \in K(\pi)$ is a block of the form $B = (j_1 < \cdots < j_{\#(B)})$. Using the criteria in Proposition \ref{prop:free_cumulants}, we conclude that
\[
  \kappa_{\sigma}[a_{i(1)}^{\hat{\intercal}(1)}, \ldots, a_{i(n)}^{\hat{\intercal}(n)}] = \tau^0[C_n^{K(\sigma)}(a_{i(1)}^{\hat{\intercal}(1)}, \ldots, a_{i(n)}^{\hat{\intercal}(n)})],
\]
where on the left-hand side we think of $\sigma \in \mcal{NC}(n)$ and on the right-hand side we think of $\sigma \in \mcal{NC}(E)$ so that $K(\sigma) \in \mcal{NC}(V)$. In particular, if $\sigma = 1_n$, then $K(\sigma) = 0_V$ and the result follows.
\end{proof}

\begin{cor}\label{cor:cactus_type_distribution}
Let $(\mcal{A}, \tau)$ be an algebraic traffic space. Suppose that $(a_i)_{i \in I} \subset \mcal{A}$ is a family of cactus-type traffics. If $\mcal{A}_1$ is the unital algebra generated by $(a_i)_{i \in I}$, then $\mcal{A}_1 \cup \mcal{A}_1^\intercal$ and $\Theta(\mcal{A}_1)$ are freely independent. If we further assume that the injective traffic distribution of $(a_i)_{i \in I}$ is supported on oriented cacti, then $\mcal{A}_1$ and $\mcal{A}_1^\intercal$ are freely independent.
\end{cor}
\begin{proof}
The result essentially follows from Lemma \ref{lem:cactus_cumulant} and our work in \S\ref{sec:free_product}.
\end{proof}

\begin{rem}\label{rem:*-cactus_type}
If $(\mcal{A}, \tau)$ has the additional structure of a $\mcal{G}^*$-algebra, then we can define \emph{$*$-cactus type traffics} in the obvious way. The natural extensions of Lemma \ref{lem:cactus_cumulant} and Corollary \ref{cor:cactus_type_distribution} still hold in this setting. In particular, one can also include $*$-labels $\varepsilon: [n] \to \{1, *\}$ in the variables $(a_{i(k)}^{\varepsilon(k)})^{\hat{\intercal}(k)}$ appearing in \eqref{eq:cycle_cumulant_equality}, and we now define $\mcal{A}_1$ to be the unital $*$-algebra generated by $(a_i)_{i \in I}$ in Corollary \ref{cor:cactus_type_distribution}.
\end{rem}

\subsection{Examples of cactus-type random matrices}
Cactus-type traffics frequently arise in the large dimension limit of random matrices. In such cases, the cactus-cumulant correspondence can be used to calculate the joint asymptotics. To illustrate this principle, we consider some classical random matrix ensembles.

\begin{defn}[Wigner matrix]\label{defn:wigner_matrix}
Let $I$ be an index set. For each $i \in I$ and $N \in \N$, let $(\mbf{X}_N^{(i)}(j, k))_{1 \leq j < k \leq N}$ and $(\mbf{X}_N^{(i)}(j,j))_{1 \leq j \leq N}$ be independent families of random variables: the former, complex-valued, centered, and of unit variance; the latter, real-valued and of finite variance. We further assume that
\[
  \sup_{N \in \N} \sup_{i \in I_0} \sup_{1 \leq j \leq k \leq N} \E[|\mbf{X}_N^{(i)}(j, k)|^\ell] \leq m_\ell^{(I_0)} < \infty, \qquad \forall I_0 \subset I: \#(I_0) < \infty,
\]
where the $(\mbf{X}_N^{(i)}(j, k))_{1 \leq j \leq k \leq N,\, i \in I}$ are independent with pseudo-variance
\[
\E[\mbf{X}_N^{(i)}(j, k)^2] = \beta_i \in [-1, 1], \qquad \forall j < k.
\]
We call the random self-adjoint matrices defined by
\[
  \mbf{W}_N^{(i)}(j, k) = \frac{1}{\sqrt{N}}\mbf{X}_N^{(i)}(j, k)
\]
a family of independent \emph{Wigner matrices}.
\end{defn}

\begin{prop}[\cite{Mal11,Au18a}]\label{prop:wigner_calculation}
Let $\mcal{W}_N = (\mbf{W}_N^{(i)})_{i \in I}$ be a family of independent Wigner matrices. Then $\mcal{W}_N$ converges in traffic distribution to a family of cactus-type traffics $\mbf{a} = (a_i)_{i \in I}$. In particular, the only non-trivial contributions come from cacti with cycle type
\begin{align*}
  \upsilon_{\mbf{a}}^0\Big[ \ \cdot \overset{x_i}{\underset{x_i}{\leftrightarrows}} \cdot \ \Big] &= 1; \\
  \upsilon_{\mbf{a}}^0\Big[ \ \cdot \overset{x_i}{\underset{x_i}{\leftleftarrows}} \cdot \ \Big] &= \beta_i.
\end{align*}
  
\end{prop}

\begin{cor}\label{cor:wigner_calculation}
With the same notation as above, let $r_i = \op{rDeg}(a_i)$. Then $(a_i, a_i^\intercal)_{i \in I}$ and $(r_i)_{i \in I}$ are freely independent with $(a_i, a_i^\intercal)_{i \in I}$ semicircular with covariance
\[
  \kappa_2[a_{i(1)}^{\hat{\intercal}(1)}, a_{i(2)}^{\hat{\intercal}(2)}] = \begin{dcases}
    1 &\text{if } i(1) = i(2) \text{ and } \hat{\intercal}(1) = \hat{\intercal}(2); \\
    \beta_{i(1)} & \text{if } i(1) = i(2) \text{ and } \hat{\intercal}(1) \neq \hat{\intercal}(2); \\
    0 &\text{if } i(1) \neq i(2), \\ 
    \end{dcases}
\]
and $(r_i)_{i \in I}$ complex Gaussian with covariance
\[
  \kappa_2[r_{i(1)}, r_{i(2)}^*] = \begin{dcases}
    1 &\text{if } i(1) = i(2); \\
    0 &\text{if } i(1) \neq i(2), \\ 
  \end{dcases}
\]
and pseudo-covariance
\[
  \kappa_2[r_{i(1)}, r_{i(2)}] = \begin{dcases}
    \beta_{i(1)} & \text{if } i(1) = i(2); \\
    0 &\text{if } i(1) \neq i(2). \\ 
  \end{dcases}
\]
\end{cor}

In the case of $\beta = 1$, Bryc, Dembo, and Jiang showed that the empirical spectral distribution of the Markov matrix $\mbf{M}_N = \mbf{W}_N - \mbf{D}_N$ converges weakly almost surely to the free convolution $\mcal{SC}(0, 1) \boxplus \mcal{N}(0, 1) = \frac{1}{\sqrt{2\pi}}(4-x^2)_+^{1/2}\, dx \boxplus \frac{1}{\sqrt{2\pi}} e^{-x^2/2}\, dx$ \cite{BDJ06}. Of course, this would be the expected answer if $\mbf{W}_N$ and $\mbf{D}_N = \op{rDeg}(\mbf{W}_N)$ were independent, but the latter matrix is completely determined by the former. Instead, we see that asymptotic freeness is a generic phenomenon for dependent random matrices of cactus-type. Similarly, Mingo and Popa showed that freeness from the transpose occurs in unitarily invariant ensembles \cite{MP16}. This was extended by C\'{e}bron, Dahlqvist, and Male using the structure of the UE traffic space \cite{CDM16}. For example, in the case of $\beta = 0$, this implies that $\mbf{W}_N$ and $\mbf{W}_N^\intercal$ are asymptotically free. Corollary \ref{cor:wigner_calculation} completes the description for general $\beta \in [-1, 1]$. We note that for strictly complex $\beta \in \mathbb{D}$, the limiting traffic is not of cactus-type: while the injective traffic distribution is supported on cacti, it is not multiplicative with respect to the pads \cite{Au18a}. In the same paper, it is shown that random band matrices in certain bandwidth regimes are also of cactus-type.

One can also obtain cactus-type traffics from non-Hermitian ensembles.

\begin{defn}[Ginibre matrix]\label{defn:ginibre_matrix}
Let $I$ be an index set. For each $i \in I$ and $N \in \N$, let $(\mbf{Y}_N^{(i)}(j, k))_{1 \leq j, k \leq N}$ be an independent family of random variables, where the off-diagonal entries are centered and of unit variance. We further assume that 
\[
  \sup_{N \in \N} \sup_{i \in I_0} \sup_{1 \leq j, k \leq N} \E[|\mbf{Y}_N^{(i)}(j, k)|^\ell] \leq m_\ell^{(I_0)} < \infty, \qquad \forall I_0 \subset I: \#(I_0) < \infty,
\]
where the $(\mbf{Y}_N^{(i)}(j, k))_{1 \leq j \leq k \leq N,\, i \in I}$ are independent with pseudo-variance
\[
  \E[\mbf{Y}_N^{(i)}(j, k)^2] = \zeta_i, \qquad \forall j \neq k.
\]
We call the random matrices defined by
\[
  \mbf{G}_N^{(i)}(j, k) = \frac{1}{\sqrt{N}}\mbf{Y}_N^{(i)}(j, k)
\]
a family of independent \emph{Ginibre matrices}.
\end{defn}

\begin{prop}[\cite{Au18b}]\label{prop:ginibre_calculation}
Let $\mcal{Y}_N = (\mbf{Y}_N^{(i)})_{i \in I}$ be a family of independent Ginibre matrices. Then $\mcal{Y}_N$ converges in $*$-traffic distribution to a family of $*$-cactus-type traffics $\mbf{a} = (a_i)_{i \in I}$. In particular, the only non-trivial contributions come from cacti with cycle type
\begin{align*}
  \upsilon_{\mbf{a}}^0\Big[ \ \cdot \overset{x_i}{\underset{x_i^*}{\leftrightarrows}} \cdot \ \Big] &= 1; \\
  \upsilon_{\mbf{a}}^0\Big[ \ \cdot \overset{x_i}{\underset{x_i}{\leftleftarrows}} \cdot \ \Big] &= \zeta_i; \\
  \upsilon_{\mbf{a}}^0\Big[ \ \cdot \overset{x_i^*}{\underset{x_i^*}{\leftleftarrows}} \cdot \ \Big] &= \overline{\zeta}_i.    
\end{align*}
\end{prop}

\begin{cor}\label{cor:ginibre_calculation}
With the same notation as above, let $r_i = \op{rDeg}(a_i)$ and $c_i = \op{cDeg}(a_i)$. Then $(a_i, a_i^\intercal)_{i \in I}$ and $(r_i, c_i)_{i \in I}$ are freely independent with $(a_i, a_i^\intercal)_{i \in I}$ circular with covariance
\[
  \kappa_2[a_{i(1)}^{\hat{\intercal}(1)}, (a_{i(2)}^{\hat{\intercal}(2)})^*] = \begin{dcases}
    1 &\text{if } i(1) = i(2) \text{ and } \hat{\intercal}(1) = \hat{\intercal}(2); \\
    0 &\text{else}, \\
    \end{dcases}
\]
and pseudo-covariance
\[
  \kappa_2[a_{i(1)}^{\hat{\intercal}(1)}, a_{i(2)}^{\hat{\intercal}(2)}] = \begin{dcases}
    \zeta_{i(1)} & \text{if } i(1) = i(2) \text{ and } \hat{\intercal}(1) \neq \hat{\intercal}(2); \\
    0 &\text{else}, \\ 
  \end{dcases}
\]  
and $(r_i, c_i)_{i \in I}$ complex Gaussian with covariance
\[
  \kappa_2[r_{i(1)}, r_{i(2)}^*] = \kappa_2[c_{i(1)}, c_{i(2)}^*] = \begin{dcases}
    1 &\text{if } i(1) = i(2); \\
    0 &\text{if } i(1) \neq i(2), \\ 
  \end{dcases}
\qquad
  \kappa_2[r_{i(1)}, c_{i(2)}^*] = 0,
\]
and pseudo-covariance
\[
  \kappa_2[r_{i(1)}, r_{i(2)}] = \kappa_2[c_{i(1)}, c_{i(2)}] = \begin{dcases}
    \zeta_{i(1)} & \text{if } i(1) = i(2); \\
    0 &\text{if } i(1) \neq i(2), \\ 
  \end{dcases}
\qquad
  \kappa_2[r_{i(1)}, c_{i(2)}] = 0.
\]
\end{cor}

For example, one can consider the analogue of the Markov matrix construction from \cite{BDJ06} for a real Wishart matrix $\mbf{L}_N = \mbf{G}_N\mbf{G}_N^\intercal$. Corollary \ref{cor:ginibre_calculation} implies that the empirical spectral distribution of $\mbf{L}_N - \op{rDeg}(\mbf{L}_N)$ converges to the free convolution $\mcal{MP}(1, 1) \boxplus \mcal{N}(-1, 1) = \frac{1}{2\pi x} ((4-x)x)_+^{1/2}\, dx \boxplus \frac{1}{\sqrt{2\pi}} e^{-(x+1)^2/2}\, dx$.

Our last example shows that cactus-type traffics can also arise in the absence of independent entries.

\begin{prop}[\cite{Mal11,Au18b}]\label{prop:haar_calculation}
Let $\mcal{U}_N = (\mbf{U}_N^{(i)})_{i \in I}$ be a family of independent Haar distributed unitary matrices. Then $\mcal{U}_N$ converges in $*$-traffic distribution to a family of $*$-cactus-type traffics $\mbf{a} = (a_i)_{i \in I}$. In particular, the only non-trivial contributions come from cacti with cycle type
\[
  \begin{tikzpicture}[shorten > = 1.5pt]
    \node at (-1.25, 0) {$\upsilon_{\mbf{a}}^0\Bigg[$};
    \node at (3.325, 0) {$\Bigg] = (-1)^{\frac{\#(C)}{2} -1}\Cat(\frac{\#(C)}{2} - 1)$,};
    \draw[fill=black] (.6, 0) circle (1pt);
    \draw[fill=black] (-.6, 0) circle (1pt);
    \draw[fill=black] (.3, .5196) circle (1pt);
    \draw[fill=black] (.3, -.5196) circle (1pt);
    \draw[fill=black] (-.3, .5196) circle (1pt);
    \draw[fill=black] (-.3, -.5196) circle (1pt);
    \draw[semithick, ->] (.6,0) to node[pos=.625, right] {${\scriptstyle x_i}$} (.3,-.5196);
    \draw[semithick, ->] (.3,-.5196) to node[midway, below] {${\scriptstyle \cdots}$} (-.3,-.5196);
    \draw[semithick, ->] (-.3,-.5196) to node[pos=.375, left] {${\scriptstyle x_i}$} (-.6, 0);
    \draw[semithick, ->] (-.6, 0) to node[pos=.625, left] {${\scriptstyle x_i^*}$} (-.3, .5196);
    \draw[semithick, ->] (-.3, .5196) to node[midway, above] {${\scriptstyle x_i}$} (.3, .5196);
    \draw[semithick, ->] (.3, .5196) to node[pos=.375, right] {${\scriptstyle x_i^*}$} (.6, 0);
  \end{tikzpicture}
\]
where $\#(C)$ is the length of the cycle and $\Cat(\ell) = \frac{\binom{2\ell}{\ell}}{\ell + 1}$ is the $\ell$th Catalan number.

Let $\mcal{O}_N = (\mbf{O}_N^{(i)})_{i \in I}$ be a family of independent Haar distributed orthogonal matrices. Then $\mcal{O}_N$ converges in $*$-traffic distribution to a family of $*$-cactus-type traffics $\mbf{b} = (b_i)_{i \in I}$. In particular, the only non-trivial contributions come from cacti with cycle type as in the unitary case except that one can also interchange any edge \begin{tikzpicture}[shorten > = 1.5pt]
    \draw[fill=black] (0, 0) circle (1pt);
    \draw[fill=black] (.6, 0) circle (1pt);
    \draw[semithick, ->] (0,0) to node[pos=.625, above] {${\scriptstyle x_i^{\varepsilon(e)}}$} (.6,0);
\end{tikzpicture}
with an edge in the opposite direction labeled by its adjoint
\begin{tikzpicture}[shorten > = 1.5pt]
    \draw[fill=black] (0, 0) circle (1pt);
    \draw[fill=black] (.6, 0) circle (1pt);
    \draw[semithick, ->] (.6,0) to node[pos=.375, above] {${\scriptstyle (x_i^{\varepsilon(e)})^*}$} (0,0);
\end{tikzpicture}.
\end{prop}

\begin{cor}\label{cor:haar_calculation}
With the same notation as above, let $r_i = \op{rDeg}(a_i)$ and $c_i = \op{cDeg}(a_i)$. Then $(a_i, a_i^\intercal)_{i \in I}$ and $(r_i, c_i)_{i \in I}$ are freely independent with $(a_i, a_i^\intercal)_{i \in I}$ a family of free Haar unitaries and $(r_i, c_i)_{i \in I}$ standard independent complex Gaussian.
  
If instead $r_i = \op{rDeg}(b_i)$ and $c_i = \op{cDeg}(b_i)$, then $(b_i)_{i \in I}$ and $(r_i, c_i)_{i \in I}$ are freely independent with $(b_i)_{i \in I}$ a family of free Haar unitaries and $(r_i, c_i)_{i \in I}$ standard independent real Gaussian.
\end{cor}

Of course, a Haar distributed unitary matrix is unitarily invariant, and so the asymptotic freeness of $\mbf{U}_N$ and $\mbf{U}_N^\intercal$ also follows from earlier results of Mingo and Popa \cite{MP16}.

\bibliographystyle{amsalpha}
\bibliography{rigid_structures}

\end{document}

%% file: fig1_cactus.pdf_tex
\begingroup%
  \makeatletter%
  \providecommand\color[2][]{%
    \errmessage{(Inkscape) Color is used for the text in Inkscape, but the package 'color.sty' is not loaded}%
    \renewcommand\color[2][]{}%
  }%
  \providecommand\transparent[1]{%
    \errmessage{(Inkscape) Transparency is used (non-zero) for the text in Inkscape, but the package 'transparent.sty' is not loaded}%
    \renewcommand\transparent[1]{}%
  }%
  \providecommand\rotatebox[2]{#2}%
  \newcommand*\fsize{\dimexpr\f@size pt\relax}%
  \newcommand*\lineheight[1]{\fontsize{\fsize}{#1\fsize}\selectfont}%
  \ifx\svgwidth\undefined%
    \setlength{\unitlength}{360bp}%
    \ifx\svgscale\undefined%
      \relax%
    \else%
      \setlength{\unitlength}{\unitlength * \real{\svgscale}}%
    \fi%
  \else%
    \setlength{\unitlength}{\svgwidth}%
  \fi%
  \global\let\svgwidth\undefined%
  \global\let\svgscale\undefined%
  \makeatother%
  \begin{picture}(1,0.375)%
    \lineheight{1}%
    \setlength\tabcolsep{0pt}%
    \put(0.15821934,1.97633934){\color[rgb]{0,0,0}\makebox(0,0)[lt]{\begin{minipage}{0.14438897\unitlength}\raggedright \end{minipage}}}%
    \put(0,0){\includegraphics[width=\unitlength,page=1]{fig1_cactus.pdf}}%
  \end{picture}%
\endgroup%

%% file: fig2-1_cactus.pdf_tex
\begingroup%
  \makeatletter%
  \providecommand\color[2][]{%
    \errmessage{(Inkscape) Color is used for the text in Inkscape, but the package 'color.sty' is not loaded}%
    \renewcommand\color[2][]{}%
  }%
  \providecommand\transparent[1]{%
    \errmessage{(Inkscape) Transparency is used (non-zero) for the text in Inkscape, but the package 'transparent.sty' is not loaded}%
    \renewcommand\transparent[1]{}%
  }%
  \providecommand\rotatebox[2]{#2}%
  \newcommand*\fsize{\dimexpr\f@size pt\relax}%
  \newcommand*\lineheight[1]{\fontsize{\fsize}{#1\fsize}\selectfont}%
  \ifx\svgwidth\undefined%
    \setlength{\unitlength}{360bp}%
    \ifx\svgscale\undefined%
      \relax%
    \else%
      \setlength{\unitlength}{\unitlength * \real{\svgscale}}%
    \fi%
  \else%
    \setlength{\unitlength}{\svgwidth}%
  \fi%
  \global\let\svgwidth\undefined%
  \global\let\svgscale\undefined%
  \makeatother%
  \begin{picture}(1,0.375)%
    \lineheight{1}%
    \setlength\tabcolsep{0pt}%
    \put(0.15821934,2.01862344){\color[rgb]{0,0,0}\makebox(0,0)[lt]{\begin{minipage}{0.14438897\unitlength}\raggedright \end{minipage}}}%
    \put(0,0){\includegraphics[width=\unitlength,page=1]{fig2-1_cactus.pdf}}%
    \put(0.46954141,0.17887885){\color[rgb]{0,0,0}\makebox(0,0)[lt]{\lineheight{1.66666675}\smash{\begin{tabular}[t]{l}$\implies$\end{tabular}}}}%
  \end{picture}%
\endgroup%

%% file: fig2-2_menger.pdf_tex
\begingroup%
  \makeatletter%
  \providecommand\color[2][]{%
    \errmessage{(Inkscape) Color is used for the text in Inkscape, but the package 'color.sty' is not loaded}%
    \renewcommand\color[2][]{}%
  }%
  \providecommand\transparent[1]{%
    \errmessage{(Inkscape) Transparency is used (non-zero) for the text in Inkscape, but the package 'transparent.sty' is not loaded}%
    \renewcommand\transparent[1]{}%
  }%
  \providecommand\rotatebox[2]{#2}%
  \newcommand*\fsize{\dimexpr\f@size pt\relax}%
  \newcommand*\lineheight[1]{\fontsize{\fsize}{#1\fsize}\selectfont}%
  \ifx\svgwidth\undefined%
    \setlength{\unitlength}{360bp}%
    \ifx\svgscale\undefined%
      \relax%
    \else%
      \setlength{\unitlength}{\unitlength * \real{\svgscale}}%
    \fi%
  \else%
    \setlength{\unitlength}{\svgwidth}%
  \fi%
  \global\let\svgwidth\undefined%
  \global\let\svgscale\undefined%
  \makeatother%
  \begin{picture}(1,0.29000001)%
    \lineheight{1}%
    \setlength\tabcolsep{0pt}%
    \put(0.15821935,2.0186235){\color[rgb]{0,0,0}\makebox(0,0)[lt]{\begin{minipage}{0.14438897\unitlength}\raggedright \end{minipage}}}%
    \put(0,0){\includegraphics[width=\unitlength,page=1]{fig2-2_menger.pdf}}%
  \end{picture}%
\endgroup%

%% file: fig2-3_complement_correspondence.pdf_tex
\begingroup%
  \makeatletter%
  \providecommand\color[2][]{%
    \errmessage{(Inkscape) Color is used for the text in Inkscape, but the package 'color.sty' is not loaded}%
    \renewcommand\color[2][]{}%
  }%
  \providecommand\transparent[1]{%
    \errmessage{(Inkscape) Transparency is used (non-zero) for the text in Inkscape, but the package 'transparent.sty' is not loaded}%
    \renewcommand\transparent[1]{}%
  }%
  \providecommand\rotatebox[2]{#2}%
  \newcommand*\fsize{\dimexpr\f@size pt\relax}%
  \newcommand*\lineheight[1]{\fontsize{\fsize}{#1\fsize}\selectfont}%
  \ifx\svgwidth\undefined%
    \setlength{\unitlength}{360bp}%
    \ifx\svgscale\undefined%
      \relax%
    \else%
      \setlength{\unitlength}{\unitlength * \real{\svgscale}}%
    \fi%
  \else%
    \setlength{\unitlength}{\svgwidth}%
  \fi%
  \global\let\svgwidth\undefined%
  \global\let\svgscale\undefined%
  \makeatother%
  \begin{picture}(1,0.4)%
    \lineheight{1}%
    \setlength\tabcolsep{0pt}%
    \put(-0.02224939,1.96678444){\color[rgb]{0,0,0}\makebox(0,0)[lt]{\begin{minipage}{0.14438897\unitlength}\raggedright \end{minipage}}}%
    \put(0,0){\includegraphics[width=\unitlength,page=1]{fig2-3_complement_correspondence.pdf}}%
    \put(-0.00259399,0.08842972){\color[rgb]{0,0,0}\makebox(0,0)[lt]{\lineheight{1.66666663}\smash{\begin{tabular}[t]{l}$1 \ \overline{1} \ 2  \ \overline{2} \ 3 \ \overline{3} \ 4 \ \overline{4} \ 5 \ \overline{5} \ 6  \ \overline{6} \ 7 \ \overline{7} \ 8 \ \overline{8} \ 9 \ \overline{9} \ 10 \ \overline{10}$\end{tabular}}}}%
    \put(0,0){\includegraphics[width=\unitlength,page=2]{fig2-3_complement_correspondence.pdf}}%
    \put(0.52639975,0.19287106){\color[rgb]{0,0,0}\makebox(0,0)[lt]{\lineheight{1.25}\smash{\begin{tabular}[t]{l}$\mapsto$\end{tabular}}}}%
  \end{picture}%
\endgroup%

%% file: fig2-6_complement_example.pdf_tex
\begingroup%
  \makeatletter%
  \providecommand\color[2][]{%
    \errmessage{(Inkscape) Color is used for the text in Inkscape, but the package 'color.sty' is not loaded}%
    \renewcommand\color[2][]{}%
  }%
  \providecommand\transparent[1]{%
    \errmessage{(Inkscape) Transparency is used (non-zero) for the text in Inkscape, but the package 'transparent.sty' is not loaded}%
    \renewcommand\transparent[1]{}%
  }%
  \providecommand\rotatebox[2]{#2}%
  \newcommand*\fsize{\dimexpr\f@size pt\relax}%
  \newcommand*\lineheight[1]{\fontsize{\fsize}{#1\fsize}\selectfont}%
  \ifx\svgwidth\undefined%
    \setlength{\unitlength}{360bp}%
    \ifx\svgscale\undefined%
      \relax%
    \else%
      \setlength{\unitlength}{\unitlength * \real{\svgscale}}%
    \fi%
  \else%
    \setlength{\unitlength}{\svgwidth}%
  \fi%
  \global\let\svgwidth\undefined%
  \global\let\svgscale\undefined%
  \makeatother%
  \begin{picture}(1,0.4)%
    \lineheight{1}%
    \setlength\tabcolsep{0pt}%
    \put(0.06108394,2.01862345){\color[rgb]{0,0,0}\makebox(0,0)[lt]{\begin{minipage}{0.14438897\unitlength}\raggedright \end{minipage}}}%
    \put(0,0){\includegraphics[width=\unitlength,page=1]{fig2-6_complement_example.pdf}}%
    \put(0.53646851,0.19261132){\color[rgb]{0,0,0}\makebox(0,0)[lt]{\lineheight{1.66666663}\smash{\begin{tabular}[t]{l}$\mapsto$\end{tabular}}}}%
    \put(0,0){\includegraphics[width=\unitlength,page=2]{fig2-6_complement_example.pdf}}%
  \end{picture}%
\endgroup%

%% file: fig2-4_butterfly.pdf_tex
\begingroup%
  \makeatletter%
  \providecommand\color[2][]{%
    \errmessage{(Inkscape) Color is used for the text in Inkscape, but the package 'color.sty' is not loaded}%
    \renewcommand\color[2][]{}%
  }%
  \providecommand\transparent[1]{%
    \errmessage{(Inkscape) Transparency is used (non-zero) for the text in Inkscape, but the package 'transparent.sty' is not loaded}%
    \renewcommand\transparent[1]{}%
  }%
  \providecommand\rotatebox[2]{#2}%
  \newcommand*\fsize{\dimexpr\f@size pt\relax}%
  \newcommand*\lineheight[1]{\fontsize{\fsize}{#1\fsize}\selectfont}%
  \ifx\svgwidth\undefined%
    \setlength{\unitlength}{360bp}%
    \ifx\svgscale\undefined%
      \relax%
    \else%
      \setlength{\unitlength}{\unitlength * \real{\svgscale}}%
    \fi%
  \else%
    \setlength{\unitlength}{\svgwidth}%
  \fi%
  \global\let\svgwidth\undefined%
  \global\let\svgscale\undefined%
  \makeatother%
  \begin{picture}(1,0.3)%
    \lineheight{1}%
    \setlength\tabcolsep{0pt}%
    \put(0.19862893,1.96292306){\color[rgb]{0,0,0}\makebox(0,0)[lt]{\begin{minipage}{0.14438897\unitlength}\raggedright \end{minipage}}}%
    \put(0,0){\includegraphics[width=\unitlength,page=1]{fig2-4_butterfly.pdf}}%
    \put(0.48603101,0.13567787){\color[rgb]{0,0,0}\makebox(0,0)[lt]{\lineheight{1.66666663}\smash{\begin{tabular}[t]{l}$\stackrel{\pi}{\mapsto}$\end{tabular}}}}%
    \put(0,0){\includegraphics[width=\unitlength,page=2]{fig2-4_butterfly.pdf}}%
  \end{picture}%
\endgroup%

%% file: fig2-5_pinching.pdf_tex
\begingroup%
  \makeatletter%
  \providecommand\color[2][]{%
    \errmessage{(Inkscape) Color is used for the text in Inkscape, but the package 'color.sty' is not loaded}%
    \renewcommand\color[2][]{}%
  }%
  \providecommand\transparent[1]{%
    \errmessage{(Inkscape) Transparency is used (non-zero) for the text in Inkscape, but the package 'transparent.sty' is not loaded}%
    \renewcommand\transparent[1]{}%
  }%
  \providecommand\rotatebox[2]{#2}%
  \newcommand*\fsize{\dimexpr\f@size pt\relax}%
  \newcommand*\lineheight[1]{\fontsize{\fsize}{#1\fsize}\selectfont}%
  \ifx\svgwidth\undefined%
    \setlength{\unitlength}{360bp}%
    \ifx\svgscale\undefined%
      \relax%
    \else%
      \setlength{\unitlength}{\unitlength * \real{\svgscale}}%
    \fi%
  \else%
    \setlength{\unitlength}{\svgwidth}%
  \fi%
  \global\let\svgwidth\undefined%
  \global\let\svgscale\undefined%
  \makeatother%
  \begin{picture}(1,0.3)%
    \lineheight{1}%
    \setlength\tabcolsep{0pt}%
    \put(0,0){\includegraphics[width=\unitlength,page=1]{fig2-5_pinching.pdf}}%
    \put(0.04862923,1.96292296){\color[rgb]{0,0,0}\makebox(0,0)[lt]{\begin{minipage}{0.14438896\unitlength}\raggedright \end{minipage}}}%
    \put(0,0){\includegraphics[width=\unitlength,page=2]{fig2-5_pinching.pdf}}%
    \put(0.48603102,0.13567788){\color[rgb]{0,0,0}\makebox(0,0)[lt]{\lineheight{1.66666663}\smash{\begin{tabular}[t]{l}$\stackrel{\pi}{\mapsto}$\end{tabular}}}}%
  \end{picture}%
\endgroup%

%% file: fig2-7_quasi.pdf_tex
\begingroup%
  \makeatletter%
  \providecommand\color[2][]{%
    \errmessage{(Inkscape) Color is used for the text in Inkscape, but the package 'color.sty' is not loaded}%
    \renewcommand\color[2][]{}%
  }%
  \providecommand\transparent[1]{%
    \errmessage{(Inkscape) Transparency is used (non-zero) for the text in Inkscape, but the package 'transparent.sty' is not loaded}%
    \renewcommand\transparent[1]{}%
  }%
  \providecommand\rotatebox[2]{#2}%
  \newcommand*\fsize{\dimexpr\f@size pt\relax}%
  \newcommand*\lineheight[1]{\fontsize{\fsize}{#1\fsize}\selectfont}%
  \ifx\svgwidth\undefined%
    \setlength{\unitlength}{360bp}%
    \ifx\svgscale\undefined%
      \relax%
    \else%
      \setlength{\unitlength}{\unitlength * \real{\svgscale}}%
    \fi%
  \else%
    \setlength{\unitlength}{\svgwidth}%
  \fi%
  \global\let\svgwidth\undefined%
  \global\let\svgscale\undefined%
  \makeatother%
  \begin{picture}(1,0.375)%
    \lineheight{1}%
    \setlength\tabcolsep{0pt}%
    \put(0.00822722,1.97633935){\color[rgb]{0,0,0}\makebox(0,0)[lt]{\begin{minipage}{0.14438897\unitlength}\raggedright \end{minipage}}}%
    \put(0,0){\includegraphics[width=\unitlength,page=1]{fig2-7_quasi.pdf}}%
  \end{picture}%
\endgroup%

%% file: fig2-8_flower.pdf_tex
\begingroup%
  \makeatletter%
  \providecommand\color[2][]{%
    \errmessage{(Inkscape) Color is used for the text in Inkscape, but the package 'color.sty' is not loaded}%
    \renewcommand\color[2][]{}%
  }%
  \providecommand\transparent[1]{%
    \errmessage{(Inkscape) Transparency is used (non-zero) for the text in Inkscape, but the package 'transparent.sty' is not loaded}%
    \renewcommand\transparent[1]{}%
  }%
  \providecommand\rotatebox[2]{#2}%
  \newcommand*\fsize{\dimexpr\f@size pt\relax}%
  \newcommand*\lineheight[1]{\fontsize{\fsize}{#1\fsize}\selectfont}%
  \ifx\svgwidth\undefined%
    \setlength{\unitlength}{360bp}%
    \ifx\svgscale\undefined%
      \relax%
    \else%
      \setlength{\unitlength}{\unitlength * \real{\svgscale}}%
    \fi%
  \else%
    \setlength{\unitlength}{\svgwidth}%
  \fi%
  \global\let\svgwidth\undefined%
  \global\let\svgscale\undefined%
  \makeatother%
  \begin{picture}(1,0.55)%
    \lineheight{1}%
    \setlength\tabcolsep{0pt}%
    \put(0.06108395,1.9676079){\color[rgb]{0,0,0}\makebox(0,0)[lt]{\begin{minipage}{0.14438898\unitlength}\raggedright \end{minipage}}}%
    \put(0,0){\includegraphics[width=\unitlength,page=1]{fig2-8_flower.pdf}}%
    \put(0.49115532,0.4879523){\color[rgb]{0,0,0}\makebox(0,0)[lt]{\lineheight{1.66666663}\smash{\begin{tabular}[t]{l}$1$\end{tabular}}}}%
    \put(0.43855121,0.4259559){\color[rgb]{0,0,0}\makebox(0,0)[lt]{\lineheight{1.66666663}\smash{\begin{tabular}[t]{l}$\overline{1}$\end{tabular}}}}%
    \put(0.35990533,0.44524396){\color[rgb]{0,0,0}\makebox(0,0)[lt]{\lineheight{1.66666663}\smash{\begin{tabular}[t]{l}$2$\end{tabular}}}}%
    \put(0.49115532,0.03951487){\color[rgb]{0,0,0}\makebox(0,0)[lt]{\lineheight{1.66666663}\smash{\begin{tabular}[t]{l}$6$\end{tabular}}}}%
    \put(0.27448868,0.19628572){\color[rgb]{0,0,0}\makebox(0,0)[lt]{\lineheight{1.66666663}\smash{\begin{tabular}[t]{l}$4$\end{tabular}}}}%
    \put(0.27448868,0.33430644){\color[rgb]{0,0,0}\makebox(0,0)[lt]{\lineheight{1.66666663}\smash{\begin{tabular}[t]{l}$3$\end{tabular}}}}%
    \put(0.70678035,0.19628572){\color[rgb]{0,0,0}\makebox(0,0)[lt]{\lineheight{1.66666663}\smash{\begin{tabular}[t]{l}$8$\end{tabular}}}}%
    \put(0.623447,0.08274404){\color[rgb]{0,0,0}\makebox(0,0)[lt]{\lineheight{1.66666663}\smash{\begin{tabular}[t]{l}$7$\end{tabular}}}}%
    \put(0.35990533,0.08274404){\color[rgb]{0,0,0}\makebox(0,0)[lt]{\lineheight{1.66666663}\smash{\begin{tabular}[t]{l}$5$\end{tabular}}}}%
    \put(0.70678035,0.33430644){\color[rgb]{0,0,0}\makebox(0,0)[lt]{\lineheight{1.66666663}\smash{\begin{tabular}[t]{l}$9$\end{tabular}}}}%
    \put(0.61459283,0.44524396){\color[rgb]{0,0,0}\makebox(0,0)[lt]{\lineheight{1.66666663}\smash{\begin{tabular}[t]{l}$10$\end{tabular}}}}%
    \put(0.31875952,0.26345587){\color[rgb]{0,0,0}\makebox(0,0)[lt]{\lineheight{1.66666663}\smash{\begin{tabular}[t]{l}$\overline{3}$\end{tabular}}}}%
    \put(0.63386367,0.36345589){\color[rgb]{0,0,0}\makebox(0,0)[lt]{\lineheight{1.66666663}\smash{\begin{tabular}[t]{l}$\overline{9}$\end{tabular}}}}%
    \put(0.35000949,0.16033086){\color[rgb]{0,0,0}\makebox(0,0)[lt]{\lineheight{1.66666663}\smash{\begin{tabular}[t]{l}$\overline{4}$\end{tabular}}}}%
    \put(0.63386367,0.16033086){\color[rgb]{0,0,0}\makebox(0,0)[lt]{\lineheight{1.66666663}\smash{\begin{tabular}[t]{l}$\overline{7}$\end{tabular}}}}%
    \put(0.35000949,0.36345589){\color[rgb]{0,0,0}\makebox(0,0)[lt]{\lineheight{1.66666663}\smash{\begin{tabular}[t]{l}$\overline{2}$\end{tabular}}}}%
    \put(0.43855121,0.09262251){\color[rgb]{0,0,0}\makebox(0,0)[lt]{\lineheight{1.66666663}\smash{\begin{tabular}[t]{l}$\overline{5}$\end{tabular}}}}%
    \put(0.53646782,0.4259559){\color[rgb]{0,0,0}\makebox(0,0)[lt]{\lineheight{1.66666663}\smash{\begin{tabular}[t]{l}$\overline{10}$\end{tabular}}}}%
    \put(0.66511367,0.26345587){\color[rgb]{0,0,0}\makebox(0,0)[lt]{\lineheight{1.66666663}\smash{\begin{tabular}[t]{l}$\overline{8}$\end{tabular}}}}%
    \put(0.54532199,0.09262251){\color[rgb]{0,0,0}\makebox(0,0)[lt]{\lineheight{1.66666663}\smash{\begin{tabular}[t]{l}$\overline{6}$\end{tabular}}}}%
  \end{picture}%
\endgroup%

%% file: fig2-9_petals.pdf_tex
\begingroup%
  \makeatletter%
  \providecommand\color[2][]{%
    \errmessage{(Inkscape) Color is used for the text in Inkscape, but the package 'color.sty' is not loaded}%
    \renewcommand\color[2][]{}%
  }%
  \providecommand\transparent[1]{%
    \errmessage{(Inkscape) Transparency is used (non-zero) for the text in Inkscape, but the package 'transparent.sty' is not loaded}%
    \renewcommand\transparent[1]{}%
  }%
  \providecommand\rotatebox[2]{#2}%
  \newcommand*\fsize{\dimexpr\f@size pt\relax}%
  \newcommand*\lineheight[1]{\fontsize{\fsize}{#1\fsize}\selectfont}%
  \ifx\svgwidth\undefined%
    \setlength{\unitlength}{432bp}%
    \ifx\svgscale\undefined%
      \relax%
    \else%
      \setlength{\unitlength}{\unitlength * \real{\svgscale}}%
    \fi%
  \else%
    \setlength{\unitlength}{\svgwidth}%
  \fi%
  \global\let\svgwidth\undefined%
  \global\let\svgscale\undefined%
  \makeatother%
  \begin{picture}(1,0.45833333)%
    \lineheight{1}%
    \setlength\tabcolsep{0pt}%
    \put(0.0050357,1.63967326){\color[rgb]{0,0,0}\makebox(0,0)[lt]{\begin{minipage}{0.12032415\unitlength}\raggedright \end{minipage}}}%
    \put(0,0){\includegraphics[width=\unitlength,page=1]{fig2-9_petals.pdf}}%
    \put(0.24262075,0.40662694){\color[rgb]{0,0,0}\makebox(0,0)[lt]{\lineheight{1.66666663}\smash{\begin{tabular}[t]{l}$1$\end{tabular}}}}%
    \put(0.19878398,0.35496327){\color[rgb]{0,0,0}\makebox(0,0)[lt]{\lineheight{1.66666663}\smash{\begin{tabular}[t]{l}$\overline{1}$\end{tabular}}}}%
    \put(0.13324576,0.37103665){\color[rgb]{0,0,0}\makebox(0,0)[lt]{\lineheight{1.66666663}\smash{\begin{tabular}[t]{l}$2$\end{tabular}}}}%
    \put(0.24262075,0.03292908){\color[rgb]{0,0,0}\makebox(0,0)[lt]{\lineheight{1.66666663}\smash{\begin{tabular}[t]{l}$6$\end{tabular}}}}%
    \put(0.06206523,0.16357146){\color[rgb]{0,0,0}\makebox(0,0)[lt]{\lineheight{1.66666663}\smash{\begin{tabular}[t]{l}$4$\end{tabular}}}}%
    \put(0.06206523,0.27858872){\color[rgb]{0,0,0}\makebox(0,0)[lt]{\lineheight{1.66666663}\smash{\begin{tabular}[t]{l}$3$\end{tabular}}}}%
    \put(0.42230828,0.16357146){\color[rgb]{0,0,0}\makebox(0,0)[lt]{\lineheight{1.66666663}\smash{\begin{tabular}[t]{l}$8$\end{tabular}}}}%
    \put(0.35286382,0.06895339){\color[rgb]{0,0,0}\makebox(0,0)[lt]{\lineheight{1.66666663}\smash{\begin{tabular}[t]{l}$7$\end{tabular}}}}%
    \put(0.13324576,0.06895339){\color[rgb]{0,0,0}\makebox(0,0)[lt]{\lineheight{1.66666663}\smash{\begin{tabular}[t]{l}$5$\end{tabular}}}}%
    \put(0.42230828,0.27858872){\color[rgb]{0,0,0}\makebox(0,0)[lt]{\lineheight{1.66666663}\smash{\begin{tabular}[t]{l}$9$\end{tabular}}}}%
    \put(0.34548535,0.37103665){\color[rgb]{0,0,0}\makebox(0,0)[lt]{\lineheight{1.66666663}\smash{\begin{tabular}[t]{l}$10$\end{tabular}}}}%
    \put(0.09895759,0.21954658){\color[rgb]{0,0,0}\makebox(0,0)[lt]{\lineheight{1.66666663}\smash{\begin{tabular}[t]{l}$\overline{3}$\end{tabular}}}}%
    \put(0.36154438,0.30287993){\color[rgb]{0,0,0}\makebox(0,0)[lt]{\lineheight{1.66666663}\smash{\begin{tabular}[t]{l}$\overline{9}$\end{tabular}}}}%
    \put(0.12499923,0.13360907){\color[rgb]{0,0,0}\makebox(0,0)[lt]{\lineheight{1.66666663}\smash{\begin{tabular}[t]{l}$\overline{4}$\end{tabular}}}}%
    \put(0.36154438,0.13360907){\color[rgb]{0,0,0}\makebox(0,0)[lt]{\lineheight{1.66666663}\smash{\begin{tabular}[t]{l}$\overline{7}$\end{tabular}}}}%
    \put(0.12499923,0.30287993){\color[rgb]{0,0,0}\makebox(0,0)[lt]{\lineheight{1.66666663}\smash{\begin{tabular}[t]{l}$\overline{2}$\end{tabular}}}}%
    \put(0.19878398,0.07718545){\color[rgb]{0,0,0}\makebox(0,0)[lt]{\lineheight{1.66666663}\smash{\begin{tabular}[t]{l}$\overline{5}$\end{tabular}}}}%
    \put(0.28038117,0.35496327){\color[rgb]{0,0,0}\makebox(0,0)[lt]{\lineheight{1.66666663}\smash{\begin{tabular}[t]{l}$\overline{10}$\end{tabular}}}}%
    \put(0.38758605,0.21954658){\color[rgb]{0,0,0}\makebox(0,0)[lt]{\lineheight{1.66666663}\smash{\begin{tabular}[t]{l}$\overline{8}$\end{tabular}}}}%
    \put(0.28775965,0.07718545){\color[rgb]{0,0,0}\makebox(0,0)[lt]{\lineheight{1.66666663}\smash{\begin{tabular}[t]{l}$\overline{6}$\end{tabular}}}}%
    \put(0,0){\includegraphics[width=\unitlength,page=2]{fig2-9_petals.pdf}}%
    \put(0.75685216,0.24950897){\color[rgb]{0,0,0}\makebox(0,0)[lt]{\lineheight{1.66666675}\smash{\begin{tabular}[t]{l}$1$\end{tabular}}}}%
    \put(0.81848412,0.21391869){\color[rgb]{0,0,0}\makebox(0,0)[lt]{\lineheight{1.66666675}\smash{\begin{tabular}[t]{l}$9$\end{tabular}}}}%
    \put(0.88489044,0.33674857){\color[rgb]{0,0,0}\makebox(0,0)[lt]{\lineheight{1.66666675}\smash{\begin{tabular}[t]{l}$10$\end{tabular}}}}%
    \put(0.59539386,0.33674857){\color[rgb]{0,0,0}\makebox(0,0)[lt]{\lineheight{1.66666675}\smash{\begin{tabular}[t]{l}$3$\end{tabular}}}}%
    \put(0.71475152,0.26296383){\color[rgb]{0,0,0}\makebox(0,0)[lt]{\lineheight{1.66666675}\smash{\begin{tabular}[t]{l}$2$\end{tabular}}}}%
    \put(0.68046332,0.20219993){\color[rgb]{0,0,0}\makebox(0,0)[lt]{\lineheight{1.66666675}\smash{\begin{tabular}[t]{l}$4$\end{tabular}}}}%
    \put(0.65398762,0.10107145){\color[rgb]{0,0,0}\makebox(0,0)[lt]{\lineheight{1.66666675}\smash{\begin{tabular}[t]{l}$5$\end{tabular}}}}%
    \put(0.75815425,0.1206027){\color[rgb]{0,0,0}\makebox(0,0)[lt]{\lineheight{1.66666675}\smash{\begin{tabular}[t]{l}$6$\end{tabular}}}}%
    \put(0.81674801,0.11322423){\color[rgb]{0,0,0}\makebox(0,0)[lt]{\lineheight{1.66666675}\smash{\begin{tabular}[t]{l}$7$\end{tabular}}}}%
    \put(0.81240773,0.17398812){\color[rgb]{0,0,0}\makebox(0,0)[lt]{\lineheight{1.66666675}\smash{\begin{tabular}[t]{l}$8$\end{tabular}}}}%
    \put(0,0){\includegraphics[width=\unitlength,page=3]{fig2-9_petals.pdf}}%
    \put(0.50825503,0.22144127){\color[rgb]{0,0,0}\makebox(0,0)[lt]{\lineheight{1.66666675}\smash{\begin{tabular}[t]{l}$\overset{\pi|_{C}}{\mapsto}$\end{tabular}}}}%
  \end{picture}%
\endgroup%

%% file: fig2-10_bcd.pdf_tex
\begingroup%
  \makeatletter%
  \providecommand\color[2][]{%
    \errmessage{(Inkscape) Color is used for the text in Inkscape, but the package 'color.sty' is not loaded}%
    \renewcommand\color[2][]{}%
  }%
  \providecommand\transparent[1]{%
    \errmessage{(Inkscape) Transparency is used (non-zero) for the text in Inkscape, but the package 'transparent.sty' is not loaded}%
    \renewcommand\transparent[1]{}%
  }%
  \providecommand\rotatebox[2]{#2}%
  \newcommand*\fsize{\dimexpr\f@size pt\relax}%
  \newcommand*\lineheight[1]{\fontsize{\fsize}{#1\fsize}\selectfont}%
  \ifx\svgwidth\undefined%
    \setlength{\unitlength}{467.99979401bp}%
    \ifx\svgscale\undefined%
      \relax%
    \else%
      \setlength{\unitlength}{\unitlength * \real{\svgscale}}%
    \fi%
  \else%
    \setlength{\unitlength}{\svgwidth}%
  \fi%
  \global\let\svgwidth\undefined%
  \global\let\svgscale\undefined%
  \makeatother%
  \begin{picture}(1,0.98076966)%
    \lineheight{1}%
    \setlength\tabcolsep{0pt}%
    \put(0.04698767,1.51260368){\color[rgb]{0,0,0}\makebox(0,0)[lt]{\begin{minipage}{0.11106849\unitlength}\raggedright \end{minipage}}}%
    \put(0,0){\includegraphics[width=\unitlength,page=1]{fig2-10_bcd.pdf}}%
    \put(0.78445539,0.79393497){\color[rgb]{0,0,0}\rotatebox{-90}{\makebox(0,0)[lt]{\lineheight{1.66666675}\smash{\begin{tabular}[t]{l}in\end{tabular}}}}}%
    \put(0,0){\includegraphics[width=\unitlength,page=2]{fig2-10_bcd.pdf}}%
    \put(0.21557934,0.76566159){\color[rgb]{0,0,0}\rotatebox{90}{\makebox(0,0)[lt]{\lineheight{1.66666675}\smash{\begin{tabular}[t]{l}\textcolor{white}{out}\end{tabular}}}}}%
    \put(0,0){\includegraphics[width=\unitlength,page=3]{fig2-10_bcd.pdf}}%
    \put(0.18229957,0.09808432){\color[rgb]{0,0,0}\makebox(0,0)[lt]{\lineheight{1.66666663}\smash{\begin{tabular}[t]{l}$bcd(G) = $\end{tabular}}}}%
    \put(0.11819698,0.77612479){\color[rgb]{0,0,0}\makebox(0,0)[lt]{\lineheight{1.66666663}\smash{\begin{tabular}[t]{l}$G=$\end{tabular}}}}%
  \end{picture}%
\endgroup%

%% file: fig2-11_factor.pdf_tex
\begingroup%
  \makeatletter%
  \providecommand\color[2][]{%
    \errmessage{(Inkscape) Color is used for the text in Inkscape, but the package 'color.sty' is not loaded}%
    \renewcommand\color[2][]{}%
  }%
  \providecommand\transparent[1]{%
    \errmessage{(Inkscape) Transparency is used (non-zero) for the text in Inkscape, but the package 'transparent.sty' is not loaded}%
    \renewcommand\transparent[1]{}%
  }%
  \providecommand\rotatebox[2]{#2}%
  \newcommand*\fsize{\dimexpr\f@size pt\relax}%
  \newcommand*\lineheight[1]{\fontsize{\fsize}{#1\fsize}\selectfont}%
  \ifx\svgwidth\undefined%
    \setlength{\unitlength}{467.99979401bp}%
    \ifx\svgscale\undefined%
      \relax%
    \else%
      \setlength{\unitlength}{\unitlength * \real{\svgscale}}%
    \fi%
  \else%
    \setlength{\unitlength}{\svgwidth}%
  \fi%
  \global\let\svgwidth\undefined%
  \global\let\svgscale\undefined%
  \makeatother%
  \begin{picture}(1,0.53846178)%
    \lineheight{1}%
    \setlength\tabcolsep{0pt}%
    \put(0.04698767,1.51272387){\color[rgb]{0,0,0}\makebox(0,0)[lt]{\begin{minipage}{0.11106849\unitlength}\raggedright \end{minipage}}}%
    \put(0,0){\includegraphics[width=\unitlength,page=1]{fig2-11_factor.pdf}}%
  \end{picture}%
\endgroup%

%% file: fig2-12_segments.pdf_tex
\begingroup%
  \makeatletter%
  \providecommand\color[2][]{%
    \errmessage{(Inkscape) Color is used for the text in Inkscape, but the package 'color.sty' is not loaded}%
    \renewcommand\color[2][]{}%
  }%
  \providecommand\transparent[1]{%
    \errmessage{(Inkscape) Transparency is used (non-zero) for the text in Inkscape, but the package 'transparent.sty' is not loaded}%
    \renewcommand\transparent[1]{}%
  }%
  \providecommand\rotatebox[2]{#2}%
  \newcommand*\fsize{\dimexpr\f@size pt\relax}%
  \newcommand*\lineheight[1]{\fontsize{\fsize}{#1\fsize}\selectfont}%
  \ifx\svgwidth\undefined%
    \setlength{\unitlength}{360bp}%
    \ifx\svgscale\undefined%
      \relax%
    \else%
      \setlength{\unitlength}{\unitlength * \real{\svgscale}}%
    \fi%
  \else%
    \setlength{\unitlength}{\svgwidth}%
  \fi%
  \global\let\svgwidth\undefined%
  \global\let\svgscale\undefined%
  \makeatother%
  \begin{picture}(1,0.55)%
    \lineheight{1}%
    \setlength\tabcolsep{0pt}%
    \put(0.06108395,1.9676079){\color[rgb]{0,0,0}\makebox(0,0)[lt]{\begin{minipage}{0.14438898\unitlength}\raggedright \end{minipage}}}%
    \put(0,0){\includegraphics[width=\unitlength,page=1]{fig2-12_segments.pdf}}%
    \put(0.48907268,0.03951489){\color[rgb]{0,0,0}\makebox(0,0)[lt]{\lineheight{1.66666663}\smash{\begin{tabular}[t]{l}$d_0$\end{tabular}}}}%
    \put(0.27084349,0.19851584){\color[rgb]{0,0,0}\makebox(0,0)[lt]{\lineheight{1.66666663}\smash{\begin{tabular}[t]{l}$d_{8}$\end{tabular}}}}%
    \put(0.70678037,0.34018253){\color[rgb]{0,0,0}\makebox(0,0)[lt]{\lineheight{1.66666663}\smash{\begin{tabular}[t]{l}$d_{3}$\end{tabular}}}}%
    \put(0.62240604,0.45528671){\color[rgb]{0,0,0}\makebox(0,0)[lt]{\lineheight{1.66666663}\smash{\begin{tabular}[t]{l}$d_{4}$\end{tabular}}}}%
  \end{picture}%
\endgroup%

%% file: fig2-13_three_connection.pdf_tex
\begingroup%
  \makeatletter%
  \providecommand\color[2][]{%
    \errmessage{(Inkscape) Color is used for the text in Inkscape, but the package 'color.sty' is not loaded}%
    \renewcommand\color[2][]{}%
  }%
  \providecommand\transparent[1]{%
    \errmessage{(Inkscape) Transparency is used (non-zero) for the text in Inkscape, but the package 'transparent.sty' is not loaded}%
    \renewcommand\transparent[1]{}%
  }%
  \providecommand\rotatebox[2]{#2}%
  \newcommand*\fsize{\dimexpr\f@size pt\relax}%
  \newcommand*\lineheight[1]{\fontsize{\fsize}{#1\fsize}\selectfont}%
  \ifx\svgwidth\undefined%
    \setlength{\unitlength}{360bp}%
    \ifx\svgscale\undefined%
      \relax%
    \else%
      \setlength{\unitlength}{\unitlength * \real{\svgscale}}%
    \fi%
  \else%
    \setlength{\unitlength}{\svgwidth}%
  \fi%
  \global\let\svgwidth\undefined%
  \global\let\svgscale\undefined%
  \makeatother%
  \begin{picture}(1,0.53499999)%
    \lineheight{1}%
    \setlength\tabcolsep{0pt}%
    \put(0,0){\includegraphics[width=\unitlength,page=1]{fig2-13_three_connection.pdf}}%
    \put(0.06108395,1.96010794){\color[rgb]{0,0,0}\makebox(0,0)[lt]{\begin{minipage}{0.14438898\unitlength}\raggedright \end{minipage}}}%
  \end{picture}%
\endgroup%

%% file: fig2-14_ancestor.pdf_tex
\begingroup%
  \makeatletter%
  \providecommand\color[2][]{%
    \errmessage{(Inkscape) Color is used for the text in Inkscape, but the package 'color.sty' is not loaded}%
    \renewcommand\color[2][]{}%
  }%
  \providecommand\transparent[1]{%
    \errmessage{(Inkscape) Transparency is used (non-zero) for the text in Inkscape, but the package 'transparent.sty' is not loaded}%
    \renewcommand\transparent[1]{}%
  }%
  \providecommand\rotatebox[2]{#2}%
  \newcommand*\fsize{\dimexpr\f@size pt\relax}%
  \newcommand*\lineheight[1]{\fontsize{\fsize}{#1\fsize}\selectfont}%
  \ifx\svgwidth\undefined%
    \setlength{\unitlength}{360bp}%
    \ifx\svgscale\undefined%
      \relax%
    \else%
      \setlength{\unitlength}{\unitlength * \real{\svgscale}}%
    \fi%
  \else%
    \setlength{\unitlength}{\svgwidth}%
  \fi%
  \global\let\svgwidth\undefined%
  \global\let\svgscale\undefined%
  \makeatother%
  \begin{picture}(1,0.425)%
    \lineheight{1}%
    \setlength\tabcolsep{0pt}%
    \put(0,0){\includegraphics[width=\unitlength,page=1]{fig2-14_ancestor.pdf}}%
    \put(0.00822722,1.97633932){\color[rgb]{0,0,0}\makebox(0,0)[lt]{\begin{minipage}{0.14438897\unitlength}\raggedright \end{minipage}}}%
    \put(0.36339214,0.03525524){\color[rgb]{0,0,0}\makebox(0,0)[lt]{\lineheight{1.25}\smash{\begin{tabular}[t]{l}$t_i$\end{tabular}}}}%
    \put(0.05729006,0.03525524){\color[rgb]{0,0,0}\makebox(0,0)[lt]{\lineheight{1.25}\smash{\begin{tabular}[t]{l}$t_j$\end{tabular}}}}%
    \put(0.66949422,0.03525524){\color[rgb]{0,0,0}\makebox(0,0)[lt]{\lineheight{1.25}\smash{\begin{tabular}[t]{l}$t_k$\end{tabular}}}}%
    \put(0.38329007,0.1784261){\color[rgb]{0,0,0}\makebox(0,0)[lt]{\lineheight{1.25}\smash{\begin{tabular}[t]{l}$v_i$\end{tabular}}}}%
    \put(0.28552131,0.1784261){\color[rgb]{0,0,0}\makebox(0,0)[lt]{\lineheight{1.25}\smash{\begin{tabular}[t]{l}$v_j$\end{tabular}}}}%
    \put(0.59162339,0.28259278){\color[rgb]{0,0,0}\makebox(0,0)[lt]{\lineheight{1.25}\smash{\begin{tabular}[t]{l}$w_i$\end{tabular}}}}%
    \put(0.68939216,0.28259278){\color[rgb]{0,0,0}\makebox(0,0)[lt]{\lineheight{1.25}\smash{\begin{tabular}[t]{l}$w_k$\end{tabular}}}}%
    \put(0.51071505,0.07425942){\color[rgb]{0,0,0}\makebox(0,0)[lt]{\lineheight{1.25}\smash{\begin{tabular}[t]{l}$v = v_i \wedge w_i$\end{tabular}}}}%
    \put(0.33839213,0.25529256){\color[rgb]{0,0,0}\makebox(0,0)[lt]{\lineheight{1.25}\smash{\begin{tabular}[t]{l}$\pi$\end{tabular}}}}%
    \put(0.64449422,0.35945924){\color[rgb]{0,0,0}\makebox(0,0)[lt]{\lineheight{1.25}\smash{\begin{tabular}[t]{l}$\pi$\end{tabular}}}}%
  \end{picture}%
\endgroup%

%% file: fig2-15_leaf.pdf_tex
\begingroup%
  \makeatletter%
  \providecommand\color[2][]{%
    \errmessage{(Inkscape) Color is used for the text in Inkscape, but the package 'color.sty' is not loaded}%
    \renewcommand\color[2][]{}%
  }%
  \providecommand\transparent[1]{%
    \errmessage{(Inkscape) Transparency is used (non-zero) for the text in Inkscape, but the package 'transparent.sty' is not loaded}%
    \renewcommand\transparent[1]{}%
  }%
  \providecommand\rotatebox[2]{#2}%
  \newcommand*\fsize{\dimexpr\f@size pt\relax}%
  \newcommand*\lineheight[1]{\fontsize{\fsize}{#1\fsize}\selectfont}%
  \ifx\svgwidth\undefined%
    \setlength{\unitlength}{360bp}%
    \ifx\svgscale\undefined%
      \relax%
    \else%
      \setlength{\unitlength}{\unitlength * \real{\svgscale}}%
    \fi%
  \else%
    \setlength{\unitlength}{\svgwidth}%
  \fi%
  \global\let\svgwidth\undefined%
  \global\let\svgscale\undefined%
  \makeatother%
  \begin{picture}(1,0.35)%
    \lineheight{1}%
    \setlength\tabcolsep{0pt}%
    \put(0,0){\includegraphics[width=\unitlength,page=1]{fig2-15_leaf.pdf}}%
    \put(0.00822722,1.97633937){\color[rgb]{0,0,0}\makebox(0,0)[lt]{\begin{minipage}{0.14438897\unitlength}\raggedright \end{minipage}}}%
    \put(0.42541089,0.24887425){\color[rgb]{0,0,0}\makebox(0,0)[lt]{\lineheight{1.25}\smash{\begin{tabular}[t]{l}$v$\end{tabular}}}}%
    \put(0.51123589,0.07348631){\color[rgb]{0,0,0}\makebox(0,0)[lt]{\lineheight{1.25}\smash{\begin{tabular}[t]{l}$v' = v' \wedge w'$\end{tabular}}}}%
    \put(0.54950257,0.24887425){\color[rgb]{0,0,0}\makebox(0,0)[lt]{\lineheight{1.25}\smash{\begin{tabular}[t]{l}$w'$\end{tabular}}}}%
    \put(0.32124421,0.24887425){\color[rgb]{0,0,0}\makebox(0,0)[lt]{\lineheight{1.25}\smash{\begin{tabular}[t]{l}$w$\end{tabular}}}}%
    \put(0.34450258,0.07348631){\color[rgb]{0,0,0}\makebox(0,0)[lt]{\lineheight{1.25}\smash{\begin{tabular}[t]{l}$v \wedge w$\end{tabular}}}}%
    \put(0.43936088,0.33453845){\color[rgb]{0,0,0}\makebox(0,0)[lt]{\lineheight{1.25}\smash{\begin{tabular}[t]{l}$\pi$\end{tabular}}}}%
    \put(0.49144422,0.18736965){\color[rgb]{0,0,0}\makebox(0,0)[lt]{\lineheight{1.25}\smash{\begin{tabular}[t]{l}$\pi$\end{tabular}}}}%
  \end{picture}%
\endgroup%

%% file: rigid_structures.bbl
\newcommand{\etalchar}[1]{$^{#1}$}
\providecommand{\bysame}{\leavevmode\hbox to3em{\hrulefill}\thinspace}
\providecommand{\MR}{\relax\ifhmode\unskip\space\fi MR }
\providecommand{\MRhref}[2]{%
  \href{http://www.ams.org/mathscinet-getitem?mr=#1}{#2}
}
\providecommand{\href}[2]{#2}
\begin{thebibliography}{VDN92}

\bibitem[ACD{\etalchar{+}}]{ACDGM18}
Benson Au, Guillaume C\'{e}bron, Antoine Dahlqvist, Franck Gabriel, and Camille
  Male, \emph{Freeness over the diagonal for large random matrices}, Ann.
  Probab., in press.

\bibitem[Au18a]{Au18a}
Benson Au, \emph{\href{https://doi.org/10.1214/18-EJP205}{Traffic distributions
  of random band matrices}}, Electron. J. Probab. \textbf{23} (2018), paper no.
  77, 48 pp. \MR{3858905}

\bibitem[Au18b]{Au18b}
\bysame, \emph{\href{https://escholarship.org/uc/item/6cb5s2t5}{Rigid
  structures in traffic probability: with a view toward random matrices}},
  Ph.D. thesis, University of California, Berkeley, 2018.

\bibitem[BDJ06]{BDJ06}
W{\l}odzimierz Bryc, Amir Dembo, and Tiefeng Jiang,
  \emph{\href{http://dx.doi.org/10.1214/009117905000000495}{Spectral measure of
  large random {H}ankel, {M}arkov and {T}oeplitz matrices}}, Ann. Probab.
  \textbf{34} (2006), no.~1, 1--38. \MR{2206341 (2007c:60039)}

\bibitem[Bol98]{Bol98}
B\'ela Bollob\'as,
  \emph{\href{https://doi.org/10.1007/978-1-4612-0619-4}{Modern graph theory}},
  Graduate Texts in Mathematics, vol. 184, Springer-Verlag, New York, 1998.
  \MR{1633290}

\bibitem[BS10]{BS10}
Zhidong Bai and Jack~W. Silverstein,
  \emph{\href{https://doi.org/10.1007/978-1-4419-0661-8}{Spectral analysis of
  large dimensional random matrices}}, second ed., Springer Series in
  Statistics, Springer, New York, 2010. \MR{2567175}

\bibitem[CDM]{CDM16}
Guillaume C\'{e}bron, Antoine Dahlqvist, and Camille Male, \emph{Universal
  constructions for spaces of traffics}, Preprint.
  \href{https://arxiv.org/abs/1601.00168v1}{https://arxiv.org/abs/1601.00168v1}.

\bibitem[GR01]{GR01}
Chris Godsil and Gordon Royle,
  \emph{\href{https://doi.org/10.1007/978-1-4613-0163-9}{Algebraic graph
  theory}}, Graduate Texts in Mathematics, vol. 207, Springer-Verlag, New York,
  2001. \MR{1829620}

\bibitem[Jon]{Jon99}
Vaughan~F.R. Jones, \emph{Planar algebras, {I}}, Preprint,
  \href{https://arxiv.org/abs/math/9909027v1}{https://arxiv.org/abs/math/9909027v1}.

\bibitem[Mal]{Mal11}
Camille Male, \emph{\href{https://arxiv.org/abs/1111.4662v8}{Traffic
  distributions and independence: permutation invariant random matrices and the
  three notions of independence}}, Mem. Amer. Math. Soc., in press.

\bibitem[Mal17]{Mal17}
\bysame, \emph{\href{https://doi.org/10.1016/j.jfa.2016.10.001}{The limiting
  distributions of large heavy {W}igner and arbitrary random matrices}}, J.
  Funct. Anal. \textbf{272} (2017), no.~1, 1--46. \MR{3567500}

\bibitem[Men27]{Men27}
Karl Menger, \emph{\href{http://eudml.org/doc/211191}{Zur allgemeinen
  Kurventheorie}}, Fundamenta Mathematicae \textbf{10} (1927), no.~1, 96--115.

\bibitem[MP16]{MP16}
James~A. Mingo and Mihai Popa,
  \emph{\href{https://doi.org/10.1016/j.jfa.2016.05.006}{Freeness and the
  transposes of unitarily invariant random matrices}}, J. Funct. Anal.
  \textbf{271} (2016), no.~4, 883--921. \MR{3507993}

\bibitem[MS12]{MS12}
James~A. Mingo and Roland Speicher,
  \emph{\href{https://doi.org/10.1016/j.jfa.2011.12.010}{Sharp bounds for sums
  associated to graphs of matrices}}, J. Funct. Anal. \textbf{262} (2012),
  no.~5, 2272--2288. \MR{2876405}

\bibitem[MS17]{MS17}
\bysame, \emph{\href{https://doi.org/10.1007/978-1-4939-6942-5}{Free
  probability and random matrices}}, Fields Institute Monographs, vol.~35,
  Springer, New York; Fields Institute for Research in Mathematical Sciences,
  Toronto, ON, 2017. \MR{3585560}

\bibitem[NS06]{NS06}
Alexandru Nica and Roland Speicher,
  \emph{\href{https://doi.org/10.1017/CBO9780511735127}{Lectures on the
  combinatorics of free probability}}, London Mathematical Society Lecture Note
  Series, vol. 335, Cambridge University Press, Cambridge, 2006. \MR{2266879}

\bibitem[VDN92]{VDN92}
D.~V. Voiculescu, K.~J. Dykema, and A.~Nica, \emph{Free random variables}, CRM
  Monograph Series, vol.~1, American Mathematical Society, Providence, RI,
  1992, A noncommutative probability approach to free products with
  applications to random matrices, operator algebras and harmonic analysis on
  free groups. \MR{1217253}

\bibitem[Voi91]{Voi91}
Dan Voiculescu, \emph{\href{https://doi.org/10.1007/BF01245072}{Limit laws for
  random matrices and free products}}, Invent. Math. \textbf{104} (1991),
  no.~1, 201--220. \MR{1094052}

\bibitem[Yau18]{Yau18}
Donald Yau, \emph{\href{https://doi.org/10.1007/978-3-319-95001-3}{Operads of
  wiring diagrams}}, Lecture Notes in Mathematics, vol. 2192, Springer, Cham,
  2018. \MR{3837179}

\end{thebibliography}
